\documentclass[a4paper,11pt]{article}

\usepackage[ngerman, american]{babel} 
\usepackage[utf8]{inputenc}
\usepackage[a4paper, left=3cm, right=3cm, top=2.5cm, bottom=2.5cm]{geometry}
\usepackage{amssymb}
\usepackage{amsmath} 
\usepackage{amsthm}
\usepackage{dsfont}
\usepackage{mathrsfs}
\usepackage{color}
\usepackage{setspace}
\usepackage{calligra}
\DeclareMathAlphabet{\mathcalligra}{T1}{calligra}{m}{n}
\usepackage{hyperref}
\usepackage[nottoc]{tocbibind}
\newcommand{\NN}{\nonumber}
\newcommand{\eps}{{\varepsilon}}        
          
\renewcommand{\phi}{{\varphi}}          
\newcommand{\om}{\omega}

\newcommand{\cB}{\mathcal{B}}
\newcommand{\cC}{\mathcal{C}}

\newcommand{\cE}{\mathcal{E}}

\newcommand{\cO}{\mathcal{O}}

\newcommand{\C}{\mathds{C}}
\newcommand{\R}{\mathds{R}}
\newcommand{\N}{\mathds{N}}

\newcommand{\m}{\mathfrak{m}}
\newcommand{\curlym}{\mbox{\Large\( \mathcalligra{m}\)}\mkern 1mu}

\newcommand{\re}{{\mathrm e}}
\newcommand{\ri}{{\mathrm i}}
\newcommand{\dx}{{\mathrm d}}
\newcommand{\I}{\mathds{1}}

\newcommand{\E}{\mathds{E}}
\renewcommand{\P}{\mathds{P}}
\newcommand{\Cov}{\mathrm{Cov}}
\newcommand{\Var}{\mathrm{Var}}

\newcommand{\diag}[1]{\mathrm{diag}(#1)}
\newcommand{\dvec}[1]{\mathrm{dvec}(#1)}

\newcommand{\dist}{\mathop{\mathrm{dist}}}
\newcommand{\Tr}{\mathop{\mathrm{Tr}}}

\newcommand{\Id}{\mathrm{Id}}

\newcommand{\NCA}{\smash{\overrightarrow{NCP}}}

\newtheorem{theorem}{Theorem}[section]         
\newtheorem{lemma}[theorem]{Lemma}             
\newtheorem{corollary}[theorem]{Corollary}
\newtheorem{definition}[theorem]{Definition}

\newtheorem{assumption}[theorem]{Assumption}

\newtheoremstyle{myrem}%                % Name
  {}%                                     % Space above
  {}%                                     % Space below
  {}%                                     % Body font
  {}%                                     % Indent amount
  {\bfseries}%                            % Theorem head font
  {.}%                                    % Punctuation after theorem head
  { }%                                    % Space after theorem head, ' ', or \newline
  {\thmname{#1}\thmnumber{ #2}\thmnote{\normalfont{ (#3)}}}
\theoremstyle{myrem}
\newtheorem*{remark}{Remark}
\newtheorem{example}[theorem]{Example}

\numberwithin{equation}{section} % Adjusts Numbering of Equations

\begin{document}
\title{\vspace{-2cm} Multi-Point Functional Central Limit Theorem for Wigner Matrices}
\author{Jana Reker\thanks{IST Austria, Am Campus 1, 3400 Klosterneuburg, Austria. E-Mail: jana.reker$@$ist.ac.at.}}
\maketitle

\begin{abstract}
Consider the random variable~$\mathrm{Tr}( f_1(W)A_1\dots f_k(W)A_k)$ where $W$ is an $N\times N$ Hermitian Wigner matrix, $k\in\mathds{N}$, and choose (possibly $N$-dependent) regular functions~$f_1,\dots, f_k$ as well as bounded deterministic matrices~$A_1,\dots,A_k$. We give a functional central limit theorem showing that the fluctuations around the expectation are Gaussian. Moreover, we determine the limiting covariance structure and give explicit error bounds in terms of the scaling of $f_1,\dots,f_k$ and the number of traceless matrices among $A_1,\dots,A_k$, thus extending the results of~\cite{CES-functCLT} to products of arbitrary length $k\geq2$. As an application, we consider the fluctuation of $\mathrm{Tr}(\re^{\ri tW}A_1\re^{-\ri tW}A_2)$ around its thermal value $\Tr(A_1)\Tr(A_2)$ when $t$ is large and give an explicit formula for the variance.
\end{abstract}


% ----------------------------------------------------------- Introduction ----------------------------------------------------------------------------------------------
\section{Introduction}\label{sect-intro}		%\ref{sect-intro}

The eigenvalues $\{\lambda_j\}_{j=1}^N$ of a large $N\times N$ Hermitian random matrix $W$ constitute a strongly correlated system of random points on the real line. Due to the strong dependence, classical central limit theorems (CLTs) aimed at independent or weakly dependent random variables do not apply. However, the linear statistics $\Tr f(W)=\sum_{j=1}^Nf(\lambda_j)$ with a regular test function $f:\R\rightarrow\R$ have a variance of order one (see~\cite{KhorunzhyKhoruzhenkoPastur1995}) and, in fact, satisfy a central limit theorem with a Gaussian limit, as shown, e.g., in~\cite{KhorunzhyKhoruzhenkoPastur1996} for the Wigner case and in~\cite{Johansson1998} for invariant ensembles, see also \cite{SinaiSoshnikov1998A, SinaiSoshnikov1998B}. Remarkably, the effect of the dependent random variables only manifests in the anomalous scaling, and removing the classical $N^{-1/2}$ prefactor fully compensates for the correlations. We emphasize that this question is well-studied for Wigner matrices, see, e.g.,~\cite{Guionnet2002, BaiYao2005, LytovaPasturCLT, Shcherbina2011, SosoeWong2013, BaoHe2021, LandonSosoe2022} and that recent work by Diaz and Mingo~\cite{DiazMingo2022} establishes a CLT for a large class of random matrix models and expresses the limiting covariance structure in terms of a Fréchet integral.

\medskip
Note that the information obtained from a CLT is twofold: It characterizes the fluctuations of the linear statistics around its mean as Gaussian and simultaneously identifies the limiting variance or, more generally, the limiting covariance structure. To generalize the CLT for $\sum_jf(\lambda_j)$, the linear statistics can be modified in different ways. First, one may replace the $N$-independent function $f$ by a function of the build
\begin{equation}\label{eq-mesofunct}
f(x)=g(N^\gamma(x-E))
\end{equation}
where $g$ is a regular $N$-independent function, $E\in\R$ lies in the limiting spectrum of $W$, and $N^{-\gamma}$ is larger than the typical eigenvalue spacing around $E$. Considering the linear statistics for a function $f$ that is concentrated around a value $E$ on a mesoscopic scale allows us to zoom into the spectrum and thus study the problem locally. For Wigner matrices, this problem was studied by He and Knowles in \cite{HeKnowles2017, He2019, HeKnowles2020}, yielding a tracial CLT for the bulk spectrum that spans the entire mesoscopic regime. Similar questions have also been studied for other models, including deformed Wigner matrices \cite{JiLee2020, LiSchnelliXu2021}, generalized Wigner~\cite{LiXu2020}, and Wigner-type~\cite{Vova2023} matrices, sample covariance matrices \cite{BaiSilverstein2004, LiSchnelliXu2021}, Haar distributed random matrices on the classical compact groups \cite{Silverstein1990, Silverstein2022, Soshnikov2000}, $\beta$-ensembles~\cite{BorotGuionnet2013,  Shcherbina2013, BekermanLodhia2018, BekermanLebleSerfaty2018, LambertLedouxWebb2019, BourgardeModyPain2022}, free sums \cite{BaoSchnelliXu2022}, and non-Hermitian random matrices~\cite{CES-nonHermCLT, ErdoesJi2021}. See also~\cite{CES-functCLT} and references therein for a discussion of further examples and previous results.

\medskip
The second generalization addresses that $\sum_{j=1}^Nf(\lambda_j)$ is inherently \textit{tracial}, i.e., the statistics only involve the eigenvalues of the random matrix, but not its eigenvectors. By testing $f(W)$ against a bounded \textit{deterministic} matrix $A$ with $\|A\|\leq1$, i.e., by modifying the centered statistics to the form
\begin{equation}\label{eq-linearstatmatrix}
\Tr[f(W)A]-\E\Tr[f(W)A]=\sum_{j=1}^Nf(\lambda_j)\langle\mathbf{u}_j,A\mathbf{u}_j\rangle-\E[\dots],
\end{equation}
the normalized eigenvectors $\mathbf{u}_1,\dots,\mathbf{u}_N$ of $W$ enter into the problem. In the Wigner case, Lytova~\cite{Lytova2013} obtained a CLT for~\eqref{eq-linearstatmatrix} on macroscopic scales including an explicit formula for the limiting variance. We refer to the CLTs that also involve eigenvectors as \textit{functional} in contrast to the tracial CLTs above. The recent paper~\cite{CES-functCLT} extended these results to all mesoscopic scales and further established that decomposing the matrix $A$ in~\eqref{eq-linearstatmatrix} according to
\begin{displaymath}
A=\langle A\rangle\Id+\mathring{A}_d+\mathring{A}_{od},\quad \langle A\rangle:=\frac{1}{N}\Tr A,
\end{displaymath}
gives rise to three asymptotically independent fluctuation modes. Here, $\Id$ denotes the identity matrix, and $\mathring{A}_d$ and $\mathring{A}_{od}$ denote the diagonal and off-diagonal components of $\smash{\mathring{A}=A-\langle A\rangle\Id}$, the traceless part of $A$, respectively. Moreover, the results in~\cite{CES-functCLT} show that the modes corresponding to the tracial and traceless part of $A$ fluctuate on different scales in the mesoscopic regime and two modes of the build~\eqref{eq-linearstatmatrix} are asymptotically independent if the involved functions live on different scales.

\medskip
In this work, we study a third generalization of the original linear statistics $\sum_{j=1}^Nf(\lambda_j)$, which extends~\eqref{eq-linearstatmatrix}  from involving one (possibly $N$-dependent, mesoscopically scaled) function of $W$ and one (possibly traceless) bounded deterministic matrix to alternating products involving $k\in\N$  functions  and bounded deterministic matrices, respectively. More precisely, we consider the fluctuation of the statistics
\begin{equation}\label{eq-defYsimple}
Y:=\langle f_1(W)A_1f_2(W)A_2\dots f_k(W)A_k\rangle-\E\langle f_1(W)A_1f_2(W)A_2\dots f_k(W)A_k\rangle,
\end{equation}
show that $Y$ satisfies a CLT with a Gaussian limit and give the limiting covariance structure as well as explicit error estimates. This generalizes~\cite[Thm.~2.4]{CES-functCLT} to arbitrary $k\geq1$. We refer to the result as a \emph{multi-point functional CLT}. Similar to the results in~\cite{CES-functCLT}, we further verify that two modes are asymptotically independent if the functions $f_j$ are rescaled to different scales or around different numbers $E_j$ via~\eqref{eq-mesofunct}. However, while the $k=1$ case only allows for two relevant classes of deterministic matrices (corresponding to the tracial and traceless modes, respectively), considering $k\geq2$ further allows us to pinpoint the size of the limiting covariance explicitly in terms of the lengths of the matrix products and the number of traceless matrices involved. We further find that two modes of the build~\eqref{eq-defYsimple} are asymptotically independent whenever the total number of traceless matrices involved is odd.

\medskip
A key ingredient for studying the fluctuation of~\eqref{eq-defYsimple} is information on the $1/N$ correction to $\E\langle f_1(W)A_1\dots F_k(W)A_k\rangle$, which was included in the error terms of previous results (cf.~\cite[Cor.~2.7]{CES-optimalLL}). Before considering the CLT, we hence give an expansion of the expectation. Note that the leading term of this expansion was already identified in~\cite{CES-thermalization,CES-optimalLL}. As the corresponding local laws are obtained by induction, the limiting object naturally arises through a recursion. The explicit form of the expectation obtained in~\cite{CES-thermalization} from solving the recursion mirrors the combinatorics encountered in (\textit{first-order}) free probability, e.g., for the alternating moments $\E \langle W_1D_1\dots W_kD_k\rangle$ of a finite family of independent Wigner matrices $(W_j )_j$ and a finite family of deterministic matrices $(D_j)_j$ (see, e.g., \cite[Sect. ~4.4]{MSBook}). Note, however, that free probability methods are typically restricted to ($N$-independent) polynomials and often require an independent family of Wigner matrices, while the resolvent approach presented in~\cite{CES-thermalization,CES-optimalLL} applies to a much wider class of functions including resolvents and mesoscopically rescaled Sobolev functions. In a similar spirit, the limiting covariance in our CLT also naturally arises through a recursion which can be solved to obtain an explicit formula. We carry out the necessary combinatorics in the companion paper~\cite{JRcompanion} to show that the parallels to free probability identified in~\cite{CES-thermalization} for the expectation continue to hold for the fluctuations. More precisely, the structure of the limiting covariance in our CLT mirrors the combinatorics in \textit{second-order} free probability theory (see \cite[Ch.~5]{MSBook} and \cite{CollinsMingoSniadySpeicher2007} for an introduction) and, in the special case $f_j(x)=x$, correctly reproduces the structure of the fluctuation moments of Wigner and deterministic matrices that was recently computed in~\cite{MaleMingoPecheSpeicher2020}. To avoid introducing additional notation, we work with the recursive definitions in the present paper and only refer to the formulas in~\cite{JRcompanion} for explicit computations and examples.

\medskip
Lastly, as an application of the functional CLT, we consider the special case $f_j(x)=\re^{\ri t_jx}$ with $t_j\in\R$. Interpreting $W$ as the Hamiltonian of a mean-field quantum system and the deterministic bounded matrix $A$ as an observable, the quantity
\begin{displaymath}
A(t):=\re^{\ri tW}A\re^{-\ri t W}
\end{displaymath}
describes the Heisenberg time evolution of $A$. In this context, applying the CLT for the linear statistics~\eqref{eq-defYsimple} yields information about the fluctuations around the equilibrium in certain thermalization problems. For $k=1$, the main interest lies in a CLT for averages of diagonal eigenvector overlaps $\langle \mathbf{u}_j,A\mathbf{u}_j\rangle$ (see~\cite[Thm.~2.3]{CES-functCLT}) due to their connection to the fluctuations in the eigenstate thermalization hypothesis (see~\cite{Deutsch1991}) which is referred to as quantum unique ergodicity in mathematics (see~\cite{RudnickSarnak1994}, further references can be found in~\cite{CES-ETH}). For $k\geq2$, the statistics in~\eqref{eq-defYsimple} translate to the simultaneous time evolution of different observables in the same quantum system. It is expected that two observables $A_1(t)$ and $A_2$ become \emph{thermalized} for $t\gg1$, i.e., that
\begin{displaymath}
\langle A_1(t)A_2\rangle\approx\langle A_1\rangle\langle A_2\rangle
\end{displaymath}
in the large $t$ regime. More precisely, if both $A_1$ and $A_2$ are traceless we have 
\begin{equation}\label{eq-thermal}
\langle A_1(t)A_2\rangle=\langle A_1A_2\rangle\frac{J_1(2t)^2}{t^2}+\frac{\xi(t)}{N}+\cO\Big(\frac{N^\eps}{N^{3/2}}\Big)
\end{equation}
for any fixed $t\in\R$, where $J_1$ denotes a Bessel function of the first kind and $\xi(t)$ is a centered Gaussian random variable with a $t$-dependent variance. The first term of~\eqref{eq-thermal} was established in the recent paper~\cite{CES-thermalization} in the form of a \emph{law of large numbers}-type result with an effective but non-optimal error bound. Applying our functional CLT for $k=2$ shows that the fluctuations around the thermal value are Gaussian and thus gives the second term of the expansion. Considering asymptotics for $t\gg1$ after letting $N\rightarrow\infty$ further yields an explicit expansion for the variance in the regime that is relevant for thermalization.

\medskip
We conclude this section with a brief overview of the paper. After introducing some commonly used notations, we collect our assumptions on the Wigner matrix $W$ in Assumption~\ref{as-Wigner}. We then briefly recall the optimal multi-resolvent local law~\cite[Thm.~2.5]{CES-optimalLL}, which constitutes one of the key tools for the analysis. The main results of the paper are then given in Section~\ref{sect-results2}. We start by giving a precise expansion of the expectation $\E\langle f_1(W)A_1\dots f_k(W)A_k\rangle$ beyond the leading term (Theorem~\ref{thm-Eexp}). Considering the fluctuations of the statistics in~\eqref{eq-defYsimple}, we then establish a CLT and give an explicit formula for the limiting covariance (Theorem~\ref{thm-functCLT}, Corollary~\ref{cor-covarianceLL}). This is followed by a discussion of the result, including the asymptotics in the mesoscopic regime (Theorem~\ref{thm-bulkasympt}), sufficient conditions for two modes to be asymptotically independent (Corollary~\ref{cor-uncorrelated}) as well as the case of multiple independent Wigner matrices. We conclude Section~\ref{sect-results2} by applying the functional CLT to thermalization problems. In Section~\ref{sect-results1}, we consider the special case of the resolvents $f_j(W):=G(z_j)=(W-z_j)^{-1}$ for some suitable spectral parameters $z_j\in\C$, which provides the key ingredient for the proof of our main results. Here, the first step is introducing a recursively defined set function $\cE[\cdot]$ (Definition~\ref{def-E}), which we then identify as the subleading $\frac{1}{N}$ term of the expectation $\E\langle G(z_1)A_1\dots G(z_k)A_k\rangle$. This added resolution is the main tool in proving the CLT in the case that all functions $f_j$ are resolvents (Theorem~\ref{thm-resolventCLTmeso}). The role of the limiting covariance in the theorem is played by a recursively defined set function $\m[\cdot|\cdot]$ (Definition~\ref{def-M}). Lastly, the proofs are given in Section~\ref{sect-proofs}. To keep the presentation concise, some routine calculations are deferred to the appendix.

\medskip
\textbf{Acknowledgements:} I am very grateful to László Erd\H{o}s for suggesting the topic and many valuable discussions during my work on the project. Partially supported by ERC Advanced Grant "RMTBeyond" No.~101020331.

\subsection{Notation and Conventions}\label{sect-prelim}
We start by introducing some notation used throughout the paper. For two positive quantities $f,g$, we write $f\lesssim g$ and $f\sim g$ whenever there exist (deterministic, $N$-independent) constants $c,C>0$ such that $f\leq Cg$ and $cg\leq f\leq Cg$, respectively. We denote the Hermitian conjugate of a matrix $A$ by $A^*$ and the complex conjugate of a scalar~$z\in\C$ by $\overline{z}$. Moreover, $\|\cdot\|$ denotes the operator norm, $\mathrm{Tr}(\cdot)$ is the usual trace and $\langle\cdot\rangle=N^{-1}\Tr(\cdot)$. We further denote the covariance of two complex random variables $Y_1,Y_2$ by $\Cov(Y_1,Y_2)$ and follow the convention
\begin{displaymath}
\Cov(Y_1,Y_2)=\E(Y_1-\E Y_1)\overline{(Y_2-\E Y_2)},
\end{displaymath}
i.e., the covariance is linear in the first and anti-linear in the second entry. For $k,a,b\in \N$ with $a\leq b$, we set $[k]=\{1,\dots,k\}$ and adopt the interval notation $[a,b]=\{a,a+1,\dots,b\}$. We further write $\langle a,b]$ or $[a,b\rangle$ to indicate that $a$ or $b$ are excluded from the interval, respectively. Ordered sets are denoted by $(\dots)$ instead of $\{\dots\}$.

\medskip
Given a matrix $A\in\C^{N\times N}$, the traceless part of $A$ is denoted by $\mathring{A}:=A-\langle A\rangle \Id$ where $\Id$ denotes the identity matrix. Further, $\mathbf{a}:=\diag{A}$ denotes the diagonal matrix obtained from extracting only the diagonal entries of $A$ and $A_1\odot A_2$ denotes the entry-wise (or Hadamard) product of two matrices $A_1$ and $A_2$. For a Hermitian matrix $W$ and $z_1,\dots,z_k\in\C\setminus\R$, we write the corresponding resolvents as $G_j=G(z_j):=(W-z_j)^{-1}$ and index products of resolvents using the interval notation
\begin{displaymath}
G_{[a,b]}:=G_aG_{a+1}\dots G_b
\end{displaymath}
for $a,b\in\N$ with $a\leq b$. Recalling that angled brackets indicate that an edge point of the interval is excluded, we write $G_{\langle a,b]}$ and $G_{[a,b\rangle}$ to exclude $G_a$ or $G_b$ from the product, respectively. Moreover, $G_{\emptyset}$ is interpreted as zero. Note that this notation differs slightly from~\cite{CES-thermalization,CES-optimalLL}. As we often consider alternating products of resolvents with deterministic matrices $A_1,\dots,A_k$, define $T_j:=G_jA_j$ and apply the same interval notation as above to write
\begin{equation}\label{eq-reschain}
T_{[k]}:=T_1\dots T_k=G_1A_1\dots G_kA_k,\quad T_{[a,b]}:=T_aT_{a+1}\dots T_b.
\end{equation}
Again, angled brackets are used to exclude $T_a$ or $T_b$ from the product, respectively, and $T_{\emptyset}$ is interpreted as zero. We call a product of the type~\eqref{eq-reschain} \emph{resolvent chain} of length $k$.

\medskip
Throughout the paper, we assume $W$ to be an $N\times N$ complex\footnote{The same method applies to the real case with only small modifications. For simplicity of the presentation, we restrict the following analysis to the complex case only.} Wigner matrix satisfying the following assumptions.
\begin{assumption}\label{as-Wigner}
The matrix elements of $W$ are independent up to Hermitian symmetry $\smash{W_{ij}=\overline{W_{ji}}}$ and we assume identical distribution in the sense that there is a centered real random variable $\chi_d$ and a centered complex random variable $\chi_{od}$ such that $\smash{W_{ij}\overset{d}{=}N^{-1/2}\chi_{od}}$ for~${i<j}$ and $\smash{W_{jj}\overset{d}{=}N^{-1/2}\chi_{d}}$, respectively. We further assume that $\E|\chi_{od}|^2=\E\chi_d^2=1$ as well as the existence of all moments of $\chi_d$ and $\chi_{od}$, i.e., there exist constants $C_p>0$ for any $p\in\N$ such that
\begin{displaymath}
\E|\chi_d|^p+\E|\chi_{od}|^p\leq C_p.
\end{displaymath}
Lastly, we assume that the pseudo-variance vanishes, i.e.,
\begin{displaymath}
\sigma:=\E\chi_{od}^2=0.
\end{displaymath}
\end{assumption}
We further introduce the notation
\begin{equation}\label{eq-kappa4}
\kappa_4:=\E|\chi_{od}|^4-2
\end{equation}
for the normalized fourth cumulant of the off-diagonal entries. Note that the notation matches~\cite{CES-functCLT}, however, we restrict the model to complex matrices with vanishing pseudo-variance, i.e., $\E W_{ij}^2=0$ for $i\neq j$, for technical simplicity. The more general model from~\cite{CES-functCLT} is studied for macroscopic scales in the companion paper~\cite{JRcompanion}, and the necessary modifications for an extension to mesoscopic scales are sketched.

\medskip
The eigenvalue density profile of $W$ is described by the semicircle law
\begin{equation}\label{eq-scdensity}
\rho_{sc}(x):=\frac{\sqrt{x^2-4}}{2\pi}\I_{[-2,2]}(x)
\end{equation}
which mainly enters our analysis in the form of its Stieltjes transform
\begin{equation}\label{eq-defm}
m(z):=\int\frac{\rho_{sc}(x)}{z-x}\dx x,\quad z\in\C\setminus\R.
\end{equation}
We remind the reader that $m(z)$ is the unique solution of the Dyson equation
\begin{equation}\label{eq-mselfcon}
-\frac{1}{m(z)}=m(z)+z,\quad \Im z\Im m(z)>0
\end{equation}
and that its derivative satisfies
\begin{equation}\label{eq-mselfconderived}
m'(z)=\frac{m(z)^2}{1-m(z)^2}.
\end{equation}
Given fixed $z_1,\dots,z_k\in\C\setminus\R$, set $m_j=m(z_j)$ and $m_j'=m'(z_j)$, respectively, and let
\begin{equation}\label{eq-defq}
q_{i,j}=\frac{m_im_j}{1-m_im_j},
\end{equation}
possibly setting $q_{j,j}=m_j'$ whenever $i=j$. Moreover, we define the \emph{iterated divided difference} for finite multi-sets $\{z_1,\dots, z_k\}\subset\C\setminus\R$ recursively by
\begin{equation}\label{eq-itdivdif}
m[z_1,\dots,z_k]:=\frac{m[z_2,\dots,z_k]-m[z_1,\dots,z_{k-1}]}{z_k-z_1}
\end{equation}
whenever there are two distinct $z_1\neq z_k$ among $z_1,\dots,z_k$ and otherwise set
\begin{displaymath}
m[\underbrace{z,\dots,z}_{k\text{ times}}]:=\frac{m^{(k-1)}(z)}{(k-1)!}
\end{displaymath}
where $m^{(k-1)}$ is the $(k-1)$-th derivative of the function $m$ in~\eqref{eq-defm}. Note that this is well-defined in the sense that $m[z_1,\dots,z_k]$ is independent of the ordering of the multi-set $\{z_1,\dots,z_k\}$. We abbreviate $m[1,\dots,k]:=m[z_1,\dots,z_k]$ and note that $q_{i,j}$ in~\eqref{eq-defq} coincides with $m[i,j]$.

\subsection{Preliminaries: Multi-Resolvent Local Laws}
Before considering the fluctuations, we briefly recall the optimal multi-resolvent local law~\cite[Thm.~2.5]{CES-optimalLL}, which characterizes the deterministic approximation of $\langle T_{[1,k]}\rangle$. We start by introducing the commonly used definition of stochastic domination.

\begin{definition}[Stochastic domination]
Let
\begin{displaymath}
X=\Big\{X^{(N)}(u)\Big|N\in\N,u\in U^{(N)}\Big\} \ \text{and}\ Y=\Big\{Y^{(N)}(u)\Big|N\in\N,u\in U^{(N)}\Big\}
\end{displaymath}
be two families of non-negative random variables that are indexed by $N$ and possibly some other parameter $u$. We say that $X$ is \emph{stochastically dominated} by $Y$, denoted by $X\prec Y$ or $X=\cO_{\prec}(Y)$, if, for all $\eps,C>0$ we have 
\begin{displaymath}
\sup_{u\in U^{(N)}}\P\Big(X^{(N)}(u)>N^\eps Y^{(N)}(u)\Big)\leq N^{-C}
\end{displaymath}
for large enough $N\geq N_0(\eps,C)$.
\end{definition}

Given $z_1,\dots,z_k\in\C$ and matrices $A_1,\dots,A_k$, we define the set function\footnote{Note that $M_{[k]}$ depends on $(z_j)_{j\in[k]}$ and $(A_j)_{j\in[k]}$, i.e., both the spectral parameters and the deterministic matrices are indexed by the same set. We hence interpret $M_{(\cdot)}$ as a function of the (ordered) index set to match the notation in the following sections.} $M_{[k]}=M_{[1,k]}$ through the recursion
\begin{equation}\label{eq-matrixMrec}
M_{[k]}=m_1\Big(A_1M_{[2,k]}+q_{1,k}\langle A_1M_{[2,k]}\rangle+\sum_{j=2}^{k-1}\langle M_{[1,j]}\rangle\big(M_{[j,k]}+q_{1,k}\langle M_{[j,k]}\rangle\big)\Big)
\end{equation}
with initial condition $M_{\emptyset}=0$. We remark that an explicit (non-recursive) formula for $M_{[k]}$ was derived in~\cite[Thm.~2.6]{CES-thermalization}, however, we will not use it in the present paper. Analogously to~\eqref{eq-matrixMrec}, we may define $M_S$ for any (cyclically) ordered set $S=(s_1,\dots,s_k)$ instead of an interval. In this case, we write
\begin{equation}\label{eq-matrixMordered}
M_S=M_{(s_1,\dots,s_k)}.
\end{equation}
The set function $M_{[k]}$ plays the role of the deterministic approximation of $T_{[1,k\rangle} G_k$ in the following multi-resolvent local laws.

\begin{theorem}[Multi-resolvent local law, {\cite[Thm.~2.5]{CES-optimalLL}}]\label{thm-multiG-LL}
Fix $\zeta>0$ and $k\in \N$. Let $z_1,\dots,z_k\in\C\setminus\R$ with $\max_j|z_j|\leq N^{100}$ and $d:=\min_j\mathrm{dist}(z_j,[-2,2])$, deterministic matrices $A_1,\dots,A_k\in\C^{N\times N}$ with $\|A_i\|\lesssim 1$ such that $a$ out of them are traceless. Set further $\eta_*:=\min_j|\Im z_j|\geq N^{-1+\zeta}$. Recalling that $T_j=G_jA_j$, we have the averaged local law
\begin{equation}\label{eq-multiGaveraged}
\langle T_{[1,k]}\rangle=\langle M_{[k]}A_k\rangle+\cO_{\prec}\Big(\frac{1}{N\eta_*\ \eta_*^{k-a/2-1}}\Big),
\end{equation}
and for $\mathbf{x},\mathbf{y}\in \C^N$ with $\|\mathbf{x}\|,\|\mathbf{y}\|\lesssim 1$ we have the isotropic local law
\begin{equation}\label{eq-multiGisotropic}
\langle \mathbf{x},T_{[1,k\rangle}G_k\mathbf{y}\rangle=\langle\mathbf{x},M_{[k]}\mathbf{y}\rangle+\cO_{\prec}\Big(\frac{1}{\sqrt{N\eta_*}\ \eta_*^{k-a/2-1}}\Big).
\end{equation}
\end{theorem}

As we frequently encounter $\langle M_{[k]}A_k\rangle$ in the following sections, we introduce the notation
\begin{equation}\label{eq-defM}
\m[T_1,\dots,T_k]=\m[z_1,A_1,\dots,z_k,A_k]:=\langle M_{[k]}A_k\rangle
\end{equation}
and remark that the function $\m[\cdot]$ satisfies a recursion similar to \eqref{eq-matrixMrec}. The arguments in the notation $\m[T_1,\dots,T_k]$ indicate the deterministic approximation of $\langle T_1\dots T_k\rangle$. Whenever $A_1=\dots=A_k=\Id$, it follows that the deterministic approximation is given by the iterated divided differences, i.e.,
\begin{equation}\label{eq-1slotdivdif}
\m[G_1,\dots,G_k]=m[1,\dots,k],
\end{equation}
which can be seen from the resolvent identity
\begin{equation}\label{eq-resolventid}
G_jG_{j-1}=\frac{G_j-G_{j-1}}{z_j-z_{j-1}}
\end{equation}
and the averaged local law~\eqref{eq-multiGaveraged}. Note that~\eqref{eq-multiGaveraged} and~\eqref{eq-multiGisotropic} may also be applied for any product $T_{s_1}\dots T_{s_{k-1}}G_{s_k}$ that is indexed by a (cyclically) ordered set $S=(s_1,\dots,s_k)$ instead of an interval. In this case, the deterministic approximation is given by~\eqref{eq-matrixMordered}.

\medskip
We further note the following a priori bounds for $m[\cdot]$, $\m[\cdot]$, and $M_{[\cdot]}$ (cf.~Lemma~2.4 and Appendix~A of~\cite{CES-optimalLL}).
\begin{lemma}\label{lem-mbounds}
Let $k\in\N$, pick spectral parameters $z_1,\dots,z_k$ and deterministic matrices such that $a$ matrices among $A_1,\dots,A_k$ are traceless. Further, set $\eta_*=\min_j |\Im z_j|$ and assume $d:=\min_j\dist(z_j,[-2,2])\leq1$. Then,
\begin{align*}
|m[1,\dots,k]|&\lesssim\frac{1}{\eta_*^{k-1}},\\
|\m[T_1,\dots,T_k]|&\lesssim \frac{1}{\eta_*^{k-1-\lceil a/2\rceil}},\\
\Big|(M_{[k]})_{ij}\Big|\leq\|M_{[k]}\|&\lesssim\frac{1}{\eta_*^{k-1-\lceil a'/2\rceil}},
\end{align*}
where $a'$ denotes the number of traceless matrices among $A_1,\dots,A_{k-1}$ and $\lceil x\rceil$ denotes the upper integrer part of $x\in\R$. The above bounds are sharp\footnote{The bounds are "sharp" in the sense that they are optimal in the class of bounds involving only $\eta_*$ in the small $\eta_*$ regime, see also~\cite{CES-optimalLL}.} whenever not all $\Im z_j$ have the same sign.
\end{lemma}

Theorem~\ref{thm-multiG-LL} together with the optimality of the bounds in Lemma~\ref{lem-mbounds} asserts that the deterministic $M_{[k]}$ is indeed the leading order approximation of $T_{[1,k\rangle}G_k$. In particular, the error terms in~\eqref{eq-multiGaveraged} and~\eqref{eq-multiGisotropic} are smaller than the natural upper bound on their leading term by a factor of $(N\eta_*)^{-1}$ and $(N\eta_*)^{-1/2}$, respectively.

\section{Main Results}\label{sect-results2}

The main result of the present paper is a functional CLT for the centered statistics
\begin{equation}\label{eq-defY}
Y^{(k,a)}_{\alpha}:=\langle f_1(W)A_1\dots f_k(W)A_k\rangle-\E\langle f_1(W)A_1\dots f_k(W)A_k\rangle,
\end{equation}
where $\alpha$ is a multi-index containing the deterministic matrices and test functions involved and $a$ denotes the number of traceless matrices among $A_1,\dots,A_k$. Note that we omit the superscripts of $\smash{Y^{(k,a)}_{\alpha}}$ whenever $a$ or $k$ are not used explicitly. The test functions $f_1,\dots,f_k$ are chosen according to the following set of assumptions.

\begin{assumption}[Test functions]\label{as-functions}
For $k,p\in\N$ let $g_1,\dots,g_k\in H^p_0(\R)$ be ($N$-independent) real-valued compactly supported test functions with $\|g_j\|_{H^p_0}\lesssim1$. Fixing $\delta,\gamma\geq0$ as well as $\gamma_1,\dots,\gamma_k\geq0$ either as
\begin{itemize}
\item[(1)] [Macro] $\delta=\gamma=\gamma_1=\dots=\gamma_k=0$ or as
\item[(2)] [Meso] $\delta>0$, $\gamma\in(0,1)$, and $0< \gamma_j\leq\gamma$,
\end{itemize}
we pick ($N$-independent) reference energies $E_j\in[-2+\delta,2-\delta]$ for $j=1,\dots,k$. Lastly, we define the test function rescaled to a scale $N^{-\gamma_j}$ around $E_j$ by
\begin{equation}\label{eq-testfunct}
f_j(x):=g_j(N^{\gamma_j}(x-E_j)).
\end{equation}
\end{assumption}

Note that Assumption~\ref{as-functions} includes both the \textit{macroscopic} scale (Case 1) and the bulk regime for  \textit{mesoscopic} scales (Case 2). We remark that the restriction to real-valued test functions is only for simplicity. Extending the results in this section to complex-valued test functions only requires minor modifications to the proofs in Section~\ref{sect-proofs}. Moreover, as one can decompose any matrix $A_j$ in $Y_\alpha$ as $A_j=\langle A_j\rangle\Id+\mathring{A}_j$, by multi-linearity, it is sufficient to consider $Y_\alpha$ for deterministic matrices $A_j$ that are either traceless or equal to the identity matrix.

\medskip
Throughout the paper, we denote the multi-index $\alpha$ in the form
\begin{displaymath}
\alpha:=((g_1,\gamma_1,E_1,A_1),\dots,(g_k,\gamma_k,E_k,A_k))
\end{displaymath}
with $g_j$, $\gamma_j$, and $E_j$ chosen following Assumption~\ref{as-functions}. Moreover, we introduce $F_j:=f_j(W)A_j$ and use the interval notation
\begin{displaymath}
F_{[i,j]}:=f_i(W)A_i\dots f_j(W)A_j
\end{displaymath}
for $i<j$ as well as $F_{\emptyset}=0$. Note that $(g_j,\gamma_j,E_j,A_j)$ and $F_j$ contain the same information. For this reason, we will occasionally abuse notation and use both quantities interchangeably.

\medskip
We further introduce the random variables
\begin{equation}\label{def-modes}
X^{(k,a)}_{\alpha}:=\langle T_{[1,k]}\rangle-\E \langle T_{[1,k]}\rangle =\langle G_1A_1\dots G_kA_k\rangle-\E \langle G_1A_1\dots G_kA_k\rangle,
\end{equation}
as a special case of~\eqref{eq-defY}. By a suitable functional calculus (cf.~\cite{Davies1995}), information on~\eqref{def-modes} carries over to the general statistics~\eqref{eq-defY}, thus yielding a key tool for the proof of our main results. We, therefore, consider the analog of the results in Sections~\ref{sect-expectation} and~\ref{sect-functCLT} for the resolvent case separately in Section~\ref{sect-results1}. Throughout the paper, we write
\begin{displaymath}
\alpha=((z_1,A_1),\dots,(z_k,A_k))
\end{displaymath}
for the multi-index in~\eqref{def-modes} containing the spectral parameters $z_1,\dots,z_k$ appearing in the resolvents as well as the deterministic matrices. Whenever we do not need the number $k$ of resolvents (resp. deterministic matrices) in the product or the number $a$ of traceless deterministic matrices among $A_1,\dots,A_k$ explicitly, we again omit the superscript, and further occasionally abuse notation to use $(z_j,A_j)$ and $T_j=G_jA_j$ interchangeably. In the context of~\eqref{eq-testfunct}, we may interpret the resolvent $G(z)$ as a function rescaled to scale $|\Im z|^{-1}$ around $\Re z$ (even though the corresponding function $g$ is not compactly supported). The analog of Assumption~\ref{as-functions} for the spectral parameters now reads as follows.

\begin{assumption}[Spectral parameters]\label{as-spectralmeso}
Let $k\in\N$. Fixing $\delta,\zeta\geq0$ either as
\begin{itemize}
\item[(1)] [Macro] $\delta=\zeta=0$ or as
\item[(2)] [Meso] $\delta>0$ and $\zeta\in(0,1)$,
\end{itemize}
pick ($N$-independent) reference energies $E_j\in[-2+\delta,2-\delta]$. We choose the spectral parameters $z_1,\dots,z_k\in\C$ such that $z_j=E_j+\ri\eta_j$ with $|\eta_j|\gtrsim N^{-1+\zeta}$ and $\max_j|z_j|\leq N^{100}$.
\end{assumption}
Note that we consider spectral parameters $z_j$ for which $|\Im z_j|$ is either of order one (macroscopic scale) or only slightly above the typical eigenvalue spacing (mesoscopic scales). Whenever $|\eta_j|$ is small, we further restrict to the bulk regime, i.e., those $z_j$ for which $\Re z_j$ is bounded away from the boundary of the support of the semicircle density at $\pm2$.

\subsection{The $\frac1N$ Term of $\E\langle f_1(W)A_1\dots f_k(W)A_k\rangle$}\label{sect-expectation}
We start our analysis by considering an expansion of $\E\langle F_{[1,k]}\rangle$ which identifies the subleading $1/N$ term. To state the theorem, we introduce a set function $\cE[\cdot]$ that plays the role of the $1/N$ term for the resolvent case $\E\langle T_{[1,k]}\rangle$. Note that $\cE[\cdot]$ characterizes the error of order $1/N$ that is obtained from interchanging $\langle T_{[1,k]}\rangle-\E\langle T_{[1,k]}\rangle$ and $\langle T_{[1,k]}\rangle-\m[T_1,\dots,T_k]$, i.e., it relates $X_\alpha$ in~\eqref{def-modes} to the bounds in the local law~\eqref{eq-multiGaveraged}. The proof of Lemma~\ref{lem-mEexchangemeso} is carried out in Section~\ref{sect-mEexchangemeso}.

\begin{lemma}\label{lem-mEexchangemeso}
Let $k\in\N$, $W$ be a Wigner matrix satisfying Assumption~\ref{as-Wigner}, and fix spectral parameters $z_1,\dots,z_k$ satisfying  Assumption~\ref{as-spectralmeso} as well as deterministic matrices $A_1,\dots,A_k$ with $\|A_j\|\lesssim1$. Moreover, assume that $a$ matrices among $A_1,\dots,A_k$ are traceless. Then there exists a set function $\cE[\cdot]$ (defined recursively in Definition~\ref{def-E} below) such that
\begin{equation}\label{eq-mEexchangemeso}
\E\langle T_1\dots T_k\rangle=\m[T_1,\dots,T_k]+\frac{\kappa_4}{N}\cE[T_1,\dots,T_k]+\cO\Big(\frac{N^\eps}{N\, \sqrt{N\eta_*}\ \eta_*^{k-a/2}}\Big)
\end{equation}
with $\m[\cdot]$ as in~\eqref{eq-defM}, $\kappa_4$ as in~\eqref{eq-kappa4}, and $\eta_*:=\min_j|\Im z_j|$.
\end{lemma}
We remark that (cf. Lemma~\ref{lem-Esize} below)
\begin{displaymath}
\cE[T_1,\dots,T_k]\lesssim\frac{1}{\eta_*^{k-1-\lceil a/2\rceil}},
\end{displaymath}
i.e., the error term in~\eqref{eq-mEexchangemeso} is indeed smaller than the deterministic leading term. A discussion of  the properties of $\cE[\cdot]$ is included in Section~\ref{sect-mEexchange} below. We now give the expansion of $\E\langle F_{[1,k]}\rangle$.

\begin{theorem}\label{thm-Eexp}
Let $k\in\N$ and pick deterministic matrices $A_1,\dots,A_k\in\C^{N\times N}$ with $\|A_j\|\lesssim1$ such that $a$ out of them are traceless. Let further $W$ be a Wigner matrix satisfying Assumption~\ref{as-Wigner} and let $f_1,\dots,f_k$ be test functions satisfying Assumption~\ref{as-functions} with $p=k-\lfloor a/2\rfloor+1$. Then, for any $\eps>0$, we have the expansion
\begin{align}
\E \langle F_{[1,k]}\rangle&=\int_{\R^k}\int_{[0,10]^{k}}\Big[\prod_{j=1}^k(\partial_{\overline{z}}(f_j)_{\C,p})(z_j)\Big]\m[G(z_1)A_1,\dots,G(z_k)A_k]\dx \eta_{[k]}\dx x_{[k]}\NN\\
&\quad+\frac{\kappa_4}{N\pi^k}\int_{\R^k}\int_{[0,10]^{k}}\Big[\prod_{j=1}^k(\partial_{\overline{z}}(f_j)_{\C,p})(z_j)\Big]\cE[G(z_1)A_1,\dots,G(z_k)A_k]\dx \eta_{[k]}\dx x_{[k]}\NN\\
&\quad+\cO\Big(\frac{N^\eps\max_j\|f_j\|_{H^p}}{N^{3/2}}\Big)\label{eq-Eexpansion}
\end{align}
where we write $z_j=x_j+\ri\eta_j$, $\dx x_{[k]}=\dx x_1\dx x_2\dots,\dx x_k$ as well as $\dx \eta_{[k]}=\dx \eta_1\dx \eta_2\dots\dx \eta_k$, and $(f_j)_{\C,p}$ denotes the almost analytic extension of $f_j$ of order $p$.
\end{theorem}
Theorem~\ref{thm-Eexp} follows from Lemma~\ref{lem-mEexchangemeso} and the Helffer-Sjöstrand formula (see~\cite{Davies1995}). As a similar argument will be used for the more involved proof of the multi-point functional CLT in Theorem~\ref{thm-functCLT}, we omit the details here.

\medskip
It follows from (93) in~\cite{CES-functCLT} that $\cE[T_1]$ is given by
\begin{align}
\cE[T_1]=\langle A_1\rangle \frac{m_1^5}{1-m_1^2}=\langle A_1\rangle m_1'm_1^3.\label{eq-Esmallest}
\end{align}
Hence, computing the second integral in~\eqref{eq-Eexpansion} shows that the $1/N$ term $\cE[f_1(W)A_1]$ of $\E\langle f_1(W)A_1\rangle$ is
\begin{displaymath}
\cE[f_1(W)A_1]=\int_{-2}^2f_1(x)\rho_{sc}(x)\dx x+\frac{\kappa_4}{2\pi}\int_{-2}^2\frac{(x^4-4x^2+2)f_1(x)}{\sqrt{4-x^2}}\dx x-\frac{f_1(0)}{2},
\end{displaymath}
where $\rho_{sc}$ denotes the density of the semicircle law in~\eqref{eq-scdensity}. We remark that this formula was already included in~\cite[Thm.~2.4]{CES-functCLT}. Theorem~\ref{thm-Eexp} hence generalizes Equation~(21) in~\cite{CES-functCLT} to arbitrary $k\geq1$ in the setting of Assumptions~\ref{as-Wigner} and~\ref{as-functions}.

\subsection{Statement of the Multi-Point Functional CLT}\label{sect-functCLT}
We now state our main result, the multi-point functional CLT for the statistics $Y_\alpha$ in~\eqref{eq-defY}. To give the limiting covariance structure explicitly, we introduce a set function $\m[\cdot|\cdot]$ to play the role of the deterministic approximation of the (appropriately scaled) covariance\footnote{Note the similarity between the notations $\m[\cdot]$ and $\m[\cdot|\cdot]$, which take one and two resolvent chains as arguments, respectively.} of $\langle T_{[1,k]}\rangle$ and $\langle T_{[k+1,k+\ell]}\rangle$ in the same way that $M_{[k]}$ and $\m[\cdot]$ do for the expectation of $T_{[1,k\rangle}G_k$ (see Theorem~\ref{thm-multiG-LL} as well as~\eqref{eq-defM}). Recall that we use $(z_j,A_j)$ and $T_j$ interchangeably. In particular, we may write
\begin{align*}
\m[\alpha|\beta]=\m[T_1,\dots,T_k|T_{k+1},\dots,T_{k+\ell}]
\end{align*}
where the two multi-indices $\alpha$ and $\beta$ contain the spectral parameters and deterministic matrices in $T_1,\dots,T_k$ and $T_{k+1},\dots,T_{k+\ell}$, respectively.

\begin{lemma}\label{lem-covapproxmeso}
Fix $k,\ell\in\N$, let $\alpha,\beta$ be two multi-indices of length $k$ and $\ell$, respectively, and let $W$ be a Wigner matrix satisfying Assumption~\ref{as-Wigner}. Pick two sets of spectral parameters $z_1,\dots,z_k$ and $z_{k+1},\dots,z_{k+\ell}$ that either both satisfy Case 1 or both satisfy Case~2 of Assumption~\ref{as-spectralmeso}, and denote $\eta_*=\min_j|\Im z_j|$. Moreover, pick deterministic matrices $A_1,\dots,A_{k+\ell}$ with $\|A_j\|\lesssim1$ such that $a$ matrices among $A_1,\dots,A_k$ and $b$ matrices among $A_{k+1},\dots,A_{k+\ell}$ are traceless. Then there exists a set function $\m[\cdot|\cdot]$ (defined recursively in Definition~\ref{def-M} below) such that
\begin{align}
&N^2\E X^{(k,a)}_\alpha X^{(\ell,b)}_\beta=\m[\alpha|\beta]+\cO\Big(\frac{N^\eps}{\sqrt{N\eta_*}\ \eta_*^{k-a/2}\eta_*^{\ell-b/2}}\Big)\label{eq-2ndorderLL1meso}
\end{align}
for any $\eps>0$.
\end{lemma}

We remark that (cf.~\eqref{eq-msize} below)
\begin{displaymath}
|\m[T_1,\dots,T_k|T_{k+1},\dots,T_{k+\ell}]|\lesssim\frac{1}{\eta_*^{k+\ell-\lceil(a+b)/2\rceil}},
\end{displaymath}
i.e., the error term in~\eqref{eq-2ndorderLL1meso} is smaller than the deterministic leading term at least by a factor $(N\eta_*)^{-1/2}$. The statistics $X_\alpha$ and the corresponding CLT, as well as the properties of the function $\m[\cdot|\cdot]$ are discussed in more detail in Section~\ref{sect-resolventCLT} below. Moreover, explicit (non-recursive) formulas for $\m[\cdot|\cdot]$ are derived in the companion paper~\cite{JRcompanion}. We further recall the following definition.
\begin{definition}
Consider two functions of the Wigner matrix $W$ in Assumption~\ref{as-Wigner}, which we denote as $N$-dependent random variables $X^{(N)}$ and $Y^{(N)}$. We say that $X^{(N)}=Y^{(N)}+\cO(\eps)$ \emph{in the sense of moments} if for any polynomial $\psi$ it holds that
\begin{displaymath}
\E \psi(X^{(N)})=\E \psi(Y^{(N)})+\cO(\eps),
\end{displaymath}
where the implicit constant in $\cO(\cdot)$ only depends on the polynomial $\psi$ and the constants in Assumption~\ref{as-Wigner}.
\end{definition}

The main result of the paper can now be stated as follows.

\begin{theorem}[Multi-point functional CLT]\label{thm-functCLT}
Under the assumptions of Theorem~\ref{thm-Eexp} it holds that, for any $\eps>0$, the centered statistics~\eqref{eq-defY} are approximately distributed~(in the sense of moments) as
\begin{equation}
NY^{(k,a)}_{\alpha}=\xi(\alpha)+\cO\Big(\frac{N^\eps\max_j\|f_j\|_{H^p}}{\sqrt{N}}\Big)
\end{equation}
with a centered ($N$-dependent) Gaussian process $\xi(\alpha)$ satisfying
\begin{align}
&\E[\xi(\alpha)\xi(\beta)]\label{eq-covfunctionsgeneral}\\
&=\frac{1}{\pi^{k+\ell}}\int_{\R^k}\dx x_{[k]}\int_{[0,10]^k}\dx\eta_{[k]}\Big[\prod_{i=1}^k(\partial_{\overline{z}}(f_i)_{\C,p})(z_i)\Big]\int_{\R^\ell}\dx x_{[k+1,k+\ell]}\int_{[0,10]^\ell}\dx\eta_{[k+1,k+\ell]}\NN\\
&\quad\times\Big[\prod_{j=k+1}^\ell(\partial_{\overline{z}}(f_j)_{\C,q})(z_j)\Big]\m[G(z_1)A_1,\dots,G(z_k)A_k|G(z_{k+1})A_{k+1},\dots,G(z_{k+\ell})A_{k+\ell}].\NN
\end{align}
Here, $z_j=x_j+\ri\eta_j$, $\dx x_{[i,j]}=\dx x_i\dx x_{i+1}\dots,\dx x_{j}$ as well as $\dx \eta_{[i,j]}=\dx \eta_i\dx \eta_{i+1}\dots,\dx \eta_{j}$ for $i<j$, and $(f_j)_{\C,p}$ denotes the almost analytic extension of $f_j$ of order $p$. Further, $\beta$ denotes another multi-index of length $\ell$ containing the deterministic matrices $A_{k+1},\dots,A_{k+\ell}$ with $\|A_j\|\lesssim1$ out of which $b$ are traceless, as well as the test functions $f_{k+1},\dots,f_{k+\ell}$ satisfying Assumption~\ref{as-functions} with $q=\ell-\lfloor b/2\rfloor+1$. Recall that $\m[\cdot|\cdot]$ was introduced in Lemma~\ref{lem-covapproxmeso}.
\end{theorem}

The key ingredient for the proof of Theorem~\ref{thm-functCLT} is the case of all $f_j$ being resolvents, which we discuss in detail in Section~\ref{sect-resolventCLT} below. The full multi-point functional CLT is then obtained from the resolvent CLT using the Helffer-Sjöstrand formula (cf.~\cite{Davies1995}). We carry out the argument in Section~\ref{sect-functCLTproof}. 

\subsection{Discussion of the Multi-Point Functional CLT for Mesoscopic Scales}
In this section, we consider the mesoscopic regime of Theorem~\ref{thm-functCLT} (Case 2 of Assumption~\ref{as-functions}) in more detail. We start by noting that the set function $\m[\cdot|\cdot]$ in Lemma~\ref{lem-covapproxmeso} can be decomposed as
\begin{equation}\label{eq-Mdecomp}
\m[\cdot|\cdot]=\m_{GUE}[\cdot|\cdot]+\kappa_4\m_\kappa[\cdot|\cdot],
\end{equation}
with functions $\m_{GUE}[\cdot|\cdot]$ and $\m_{\kappa}[\cdot|\cdot]$ that do not depend on any parameters of the underlying Wigner matrix $W$. Equation~\eqref{eq-Mdecomp} induces a similar decomposition for the limiting covariance in~\eqref{eq-covfunctionsgeneral}. We remark that the two contributions are not of comparable size in the mesoscopic regime and that the summand with prefactor $\kappa_4$ is of lower order. This reduces the leading term in~\eqref{eq-covfunctionsgeneral} to the case $\kappa_4=0$, thus simplifying it considerably. We give a brief example in the resolvent case to illustrate this phenomenon. More general bounds are given in Lemma~\ref{lem-sizems} below. Recall that we may interpret the resolvent $G(z)$ as a function rescaled to scale $|\Im z|^{-1}$ around $\Re z$ to match~\eqref{eq-testfunct}.

\begin{example}\label{ex-mesobounds}
For $k=\ell=1$, we have by (92) of~\cite{CES-functCLT} (or Definition~\ref{def-M} below) that
\begin{align}
\m[T_1|T_2]&:=\langle A_1A_2\rangle\frac{m_1^2m_2^2}{(1-m_1m_2)}+\langle \mathbf{a}_1\mathbf{a}_2\rangle\cdot\kappa_4m_1^3m_2^3\label{eq-Msmallest}\\
&\quad+\langle A_1\rangle \langle A_2\rangle\Big(\frac{m_1'm_2'}{(1-m_1m_2)^2}-\frac{m_1^2m_2^2}{(1-m_1m_2)}+2\kappa_4m_1m_1'm_2m_2'-\kappa_4m_1^3m_2^3\Big),\NN
\end{align}
where $\mathbf{a}_i$ denotes the diagonal part of the matrix $A_i$. Assuming that $\|A_1\|,\|A_2\|\lesssim 1$ and $|\Im z_1|,|\Im z_2|\geq\eta_*$, a brief computation using the explicit explicit form
\begin{displaymath}
m(z)=\frac{-z+\sqrt{z^2-4}}{2}
\end{displaymath}
of the solution to~\eqref{eq-mselfcon} shows that $|(1-m_1m_2)^{-1}|\lesssim \eta_*^{-1}$. Hence, the function $\m[T_1|T_2]$ in~\eqref{eq-Msmallest} satisfies the bounds
\begin{align*}
\m[T_1|T_2]\lesssim \begin{cases} \eta_*^{-2},\quad&\text{if }\langle A_1\rangle\langle A_2\rangle\neq0,\\ \eta_*^{-1},\quad &\text{if }\langle A_1\rangle\langle A_2\rangle=0.\end{cases}
\end{align*}
Both inequalities are sharp if $\Im z_1$ and $\Im z_2$ have opposite signs. We further note that
\begin{displaymath}
\m_{GUE}[T_1|T_2]\lesssim \begin{cases} \eta_*^{-2},\quad&\text{if }\langle A_1\rangle\langle A_2\rangle\neq0,\\ \eta_*^{-1},\quad &\text{if }\langle A_1\rangle\langle A_2\rangle=0,\end{cases}\quad \m_{\kappa}[T_1|T_2]=\cO(1),
\end{displaymath}
i.e., the two parts of $\m[\cdot|\cdot]$ (cf. decomposition in~\eqref{eq-Mdecomp}) do not contribute equally unless $\eta_*\gtrsim1$ (macroscopic regime).
\end{example}

We hence restrict the following discussion to the case $\kappa_4=0$. Even with this simplification, the limiting covariance in Theorem~\ref{thm-functCLT} may be tedious to compute using the recursive definition of $\m[\cdot|\cdot]$ alone. In the companion paper~\cite{JRcompanion}, we consider the recursion defining $\m[\cdot|\cdot]$ in detail and derive explicit formulas. Combining \cite[Thm.~2.4]{JRcompanion} with~\eqref{eq-covfunctionsgeneral} then yields a more direct way of computing the limiting covariance which is fully explicit whenever $\kappa_4=0$. The result is given in Corollary~\ref{cor-covarianceLL} below. We emphasize that $\E\xi(\alpha)\xi(\beta)$ is a sum of terms that decompose into a product of a function of the deterministic matrices $A_1,\dots,A_k$ and an expression in the test functions $f_1,\dots,f_{k+\ell}$, respectively. Moreover, the combinatorics underlying the summation mirror those encountered in second-order free probability theory. The proof of Corollary~\ref{cor-covarianceLL} is given in Section~\ref{sect-bigformula}.

\medskip
To state the result, we denote by $NCP(k)$ the set of non-crossing partitions of $[k]$, by $\NCA(k,\ell)$ the set of non-crossing permutations of the $(k,\ell)$-annulus, and by $K(\pi)$ the Kreweras complement associated with an element $\pi\in NCP(k)$ or $\pi\in\NCA(k,\ell)$, respectively. Moreover, the first and second-order cumulants $h_\circ$ and $h_{\circ\circ}$ associated with set functions $h[\cdot]$ and $h[\cdot|\cdot]$ are computed recursively from the moment-cumulant relations
\begin{align}
h[S]&=\sum_{\pi\in NCP(S)}\prod_{B\in\pi}h_{\circ}[B],\label{eq-mcrelation1}\\
h[S_1|S_2]&=\sum_{\pi\in \NCA(|S_1|,|S_2|)}\prod_{B\in\pi}h_{\circ}[B]+\sum_{\substack{\pi_1\times\pi_2\in NCP(|S_1|)\times NCP(|S_2|),\\ U_1\in\pi_1,U_2\in\pi_2\text{ marked}}} h_{\circ\circ}[U_1|U_2]\prod_{\substack{B\in \pi_1\setminus U_1\\ \cup\pi_2\setminus U_2}}h_{\circ}[B].\label{eq-mcrelation2}
\end{align}

The full definitions and some illustrative examples are, e.g., given in~\cite[Sect.~1]{JRcompanion} or~\cite{MSBook}.

\begin{corollary}\label{cor-covarianceLL}
Consider the setup of Theorem~\ref{thm-functCLT} for $\kappa_4=0$. Then,
\begin{align}
&\E\xi(\alpha)\xi(\beta)=\sum_{\pi\in \NCA(k,\ell)}\Big(\prod_{B\in K(\pi)}\Big\langle \prod_{j\in B}A_j\Big\rangle\Big)\prod_{B\in\pi}\Phi_{\pi,B}(f_j|j\in B)\label{eq-covfunctionsprecise}\\
&\quad+\sum_{\substack{\pi_1\times\pi_2\in NCP(k)\times NCP(\ell),\\ U_1\in\pi_1,U_2\in\pi_2\text{ marked}}}\Big(\prod_{\substack{B_1\in K(\pi_1),\\ B_2\in K(\pi_2)}}\Big\langle \prod_{j\in B_1} A_j\Big\rangle\Big\langle \prod_{j\in B_2} A_j\Big\rangle\Big)\Phi_{\pi_1\times\pi_2,U_1\times U_2}(f_1,\dots,f_{k+\ell}).\NN
\end{align}
The functions $\Phi_{\pi,B}$ and $\Phi_{\pi_1\times\pi_2,U_1\times U_2}$ in~\eqref{eq-covfunctionsprecise} are given by
\begin{equation}\label{eq-phis1}
\Phi_{\pi,B}(f_j|j\in B):=\mathrm{sc}_{\circ}[B],
\end{equation}
where $\mathrm{sc}_{\circ}[\cdot]$ denotes the first-order free cumulant function associated with
\begin{equation}\label{eq-defsc}
\mathrm{sc}[i_1,\dots,i_n]:=\int_{-2}^2\Big[\prod_{j=1}^n f_{i_j}(x)\Big]\rho_{sc}(x)\dx x,
\end{equation}
with $\rho_{sc}$ as in~\eqref{eq-scdensity}, and 
\begin{equation}\label{eq-phis2}
\Phi_{\pi_1\times\pi_2,U_1\times U_2}(f_1,\dots,f_{k+\ell}):=\mathrm{sc}_{\circ\circ}[U_1|U_2]\prod_{\substack{B_1\in \pi_1\setminus U_1,\\ B_2\in\pi_2\setminus U_2}}\mathrm{sc}_{\circ}[B_1]\mathrm{sc}_{\circ}[B_2].
\end{equation}
Here, $\mathrm{sc}_{\circ\circ}[\cdot|\cdot]$ denotes the second-order free cumulants associated with $\mathrm{sc}[\cdot]$ in~\eqref{eq-defsc} and
\begin{equation}\label{eq-defsc2}
\mathrm{sc}[i_1,\dots,i_n|i_{n+1},\dots,i_{n+m}]:=\frac12\int_{-2}^2\int_{-2}^2\Big(\prod_{j=1}^nf_{i_j}(x)\Big)'\Big(\prod_{j=1}^mf_{i_{n+j}}(y)\Big)'u(x,y)\dx x\dx y,
\end{equation}
where the integral kernel $u:[-2,2]\times[-2,2]\rightarrow\R$ is given by
\begin{equation}\label{eq-kernel}
u(x,y):=\frac{1}{4\pi^2}\ln\Big[\frac{(\sqrt{4-x^2}+\sqrt{4-y^2})^2(xy+4-\sqrt{4-x^2}\sqrt{4-y^2})}{(\sqrt{4-x^2}-\sqrt{4-y^2})^2(xy+4+\sqrt{4-x^2}\sqrt{4-y^2})}\Big].
\end{equation}
\end{corollary}

We remark that the structure of~\eqref{eq-covfunctionsprecise} resembles the formula in~\cite[Thm.~6]{MaleMingoPecheSpeicher2020} for the covariance of alternating products of GUE and deterministic matrices in second-order free probability. This connection is discussed further in the companion paper~\cite{JRcompanion}. In particular, applying Theorem~\ref{thm-functCLT} and~\eqref{eq-covfunctionsprecise} for $f_1(x)=\dots=f_{k+\ell}(x)=x$ reproduces the corresponding formulas in~\cite{MaleMingoPecheSpeicher2020} (cf.~\cite[Cor.~2.11]{JRcompanion}).

\medskip
Theorem~\ref{thm-functCLT} and Corollary~\ref{cor-covarianceLL} identify the limiting process $\xi(\alpha)$ in terms of the test functions $f_1,\dots,f_{k+\ell}$. Similar to~\cite[Sect.~2.3]{CES-functCLT} in the case $k=\ell=1$, we can make use of the mesoscopic scaling~\eqref{eq-testfunct} to give asymptotic formulas in terms of the functions $g_1,\dots,g_{k+\ell}$ whenever $\gamma_j>0$ for all $j$. The key quantities $\mathrm{sc}[\cdot]$ and $\mathrm{sc}[\cdot|\cdot]$ characterizing the covariance structure of two modes $\smash{Y^{(k,a)}_{\alpha}}$ and $\smash{Y^{(\ell,b)}_{\beta}}$ can then be conveniently expressed in terms of the $\smash{L^2}$ and $\smash{\dot{H}^{1/2}}$ inner products
\begin{displaymath}
\langle f,g\rangle_{L^2}:=\int_{\R}f(x)g(x)\dx x,\quad \langle f,g\rangle_{\dot{H}^{1/2}}:=\int_{\R^2}\frac{f(x)-f(y)}{x-y}\frac{g(x)-g(y)}{x-y}\dx x\dx y.
\end{displaymath}

\begin{theorem}[Bulk scaling asymptotics]\label{thm-bulkasympt}
Under the assumptions of Theorem~\ref{thm-functCLT}, pick test functions $f_1,\dots,f_{k+\ell}$ that satisfy Assumption~\ref{as-functions} with some $\delta,\gamma_j>0$.
\begin{itemize}
\item[(i)] Whenever $\gamma_1=\dots=\gamma_{n+m}=\gamma$ and the functions $f_{i_1},\dots,f_{i_{n+m}}$ are rescaled around the same $E_0\in[-2+\delta,2-\delta]$, it holds that
\begin{align*}
N^\gamma\mathrm{sc}[i_1,\dots,i_{n+m}]&=\rho_{sc}(E_0)\Big\langle\prod_{j=1}^ng_{i_j},\prod_{j=n+1}^{n+m}g_{i_j}\Big\rangle_{L^2}+\cO\Big(N^{-\gamma}\Big),\\
\mathrm{sc}[i_1,\dots,i_n|i_{n+1},\dots,i_{n+m}]&=\frac{1}{4\pi^2}\Big\langle\prod_{j=1}^ng_{i_j},\prod_{j=n+1}^{n+m}g_{i_j}\Big\rangle_{\dot{H}^{1/2}}+\cO\Big(N^{-\gamma}\Big).
\end{align*}
\item[(ii)] If $\gamma_{i_1}=\dots=\gamma_{i_{m+n}}=\gamma$, but the functions $f_{i_1},\dots,f_{i_{n+m}}$ are not rescaled around the same energy, we have the bounds
\begin{align*}
N^{\gamma}\mathrm{sc}[i_1,\dots,i_{n+m}]&=\cO\Big(N^{-\gamma}\Big),\\
\mathrm{sc}[i_1,\dots,i_n|i_{n+1},\dots,i_{n+m}]&=\cO\Big(N^{-\gamma}\Big).
\end{align*}
Recall that $E_{i_1},\dots,E_{i_{n+m}}$ are fixed and $N$-independent.
\item[(iii)] If $E_{i_1}=\dots=E_{i_{n+m}}$, but the scales $\gamma_1,\dots,\gamma_{n+m}$ do not all coincide, we have the bounds
\begin{align*}
N^{\gamma_{\min}}\mathrm{sc}[i_1,\dots,i_{n+m}]&=\cO\Big(N^{-(\gamma_{\min,2}-\gamma_{\min}}\Big),\\
\mathrm{sc}[i_1,\dots,i_n|i_{n+1},\dots,i_{n+m}]&=\cO\Big(N^{-(\gamma_{\min,2}-\gamma_{\min}}\Big).
\end{align*}
with $\gamma_{\min}=\min_j\gamma_{i_j}$ and $\gamma_{\min,2}=\min\{\gamma_{i_j}| \gamma_{i_j}>\gamma_{\min}\}$.
\end{itemize}
The implicit constants in the error terms depend only on $\delta$ and the scaling exponents $\gamma_j$ as well as the test functions $g_1,\dots,g_{k+\ell}$ through $\|g_j\|_{H^p_0}$ and $|\mathrm{supp}g_j|$.
\end{theorem}

The proof of Theorem~\ref{thm-bulkasympt} follows from the definition of $\mathrm{sc}[\cdot]$ in~\eqref{eq-defsc} and (careful) integration by parts of~\eqref{eq-defsc2}. We omit the details.

\medskip
Next, we discuss conditions under which two modes $Y_{\alpha}$ and $Y_{\beta}$ in Theorem~\ref{thm-functCLT} are asymptotically independent. Note that the case $k=\ell=1$ of Corollary~\ref{cor-uncorrelated} was already discussed in~\cite[Thm.~2.13]{CES-functCLT}.

\begin{corollary}[Independent modes]\label{cor-uncorrelated}
Under the assumptions of Theorem~\ref{thm-functCLT}, let $\alpha$ and $\beta$ be multi-indices such that the test functions $f_1,\dots,f_k$ associated with $\alpha$ are all rescaled around a common reference energy $E_\alpha$ on the scale $\gamma_\alpha>0$ and the test functions $f_{k+1},\dots,f_{k+\ell}$ associated with $\beta$ are all rescaled around $E_\beta$ on the scale $\gamma_\beta>0$. Further, assume that $\alpha$ and $\beta$ contain $a$ and $b$ traceless matrices, respectively, and denote by $\xi(\alpha)$ and $\xi(\beta)$ the corresponding limiting Gaussian processes  in Theorem~\ref{thm-functCLT}.
\begin{itemize}
\item[(i)] If $E_\alpha\neq E_\beta$, then the processes $\xi(\alpha)$ and $\xi(\beta)$ are asymptotically independent in the sense that
\begin{displaymath}
\ \ \big|\E[\xi(\alpha)\xi(\beta)]\big|\lesssim N^{-\min\{\gamma_\alpha,\gamma_\beta\}}.
\end{displaymath}
\item[(ii)] If $\gamma_\alpha\neq\gamma_\beta$, then the processes $\xi(\alpha)$ and $\xi(\beta)$ are asymptotically independent in the sense that
\begin{displaymath}
\big|\E[\xi(\alpha)\xi(\beta)]\big|\lesssim N^{-|\gamma_\alpha-\gamma_\beta|}.
\end{displaymath}
\item[(iii)] If $a+b$ is odd, then the processes $\xi(\alpha)$ and $\xi(\beta)$ are always asymptotically independent in the sense that
\begin{displaymath}
\E[\xi(\alpha)\xi(\beta)]=\cO\Big(\frac{N^\eps\max_{i\in[k]}\|f_i\|_{H^p}\max_{j\in[k+1,k+\ell]}\|f_j\|_{H^q}}{\sqrt{N}}\Big),
\end{displaymath}
where $p=k-\lfloor a/2\rfloor+1$ and $q=\ell-\lfloor b/2\rfloor+1$, i.e., the leading deterministic term in~\eqref{eq-covfunctionsgeneral} has the same size as the error term in Theorem~\ref{thm-functCLT}.
\end{itemize}
\end{corollary}

The proof of Corollary~\ref{cor-uncorrelated} is immediate from Theorem~\ref{thm-bulkasympt} as well as the remark on independent modes in the resolvent CLT below Theorem~\ref{thm-resolventCLTmeso}. 

\begin{remark}[Multiple independent Wigner matrices]
The high-probability sense of Theorem~\ref{thm-functCLT} allows us to generalize the result to multiple independent Wigner matrices by resolving the individual matrices iteratively while conditioning on all others. We remark that a similar mechanism has also been applied for computing the deterministic approximation of $\langle f_1(W_{i_1})A_1\dots f_k(W_{i_k})A_k\rangle$ if $W_{i_1},\dots,W_{i_k}$ are taken, possibly with repetitions, from a family of independent Wigner matrices (see~\cite[Ext.~2.13]{CES-thermalization}). We give an example in the case $k=\ell=2$. Let $W,W'$ denote two independent GUE matrices and pick bounded deterministic matrices $A_1,\dots,A_4$ with $\langle A_j\rangle\neq0$ as well as test functions $f_1,\dots,f_4$ satisfying Case 2 of Assumption~\ref{as-functions} with $p=3$. Then,
\begin{align*}
&N^2\E\big[\big(\langle f_1(W)A_1f_2(W')A_2\rangle-\E\langle f_1(W)A_2f_2(W')A_2\rangle\big)\\
&\quad\times\big(\langle f_3(W)A_3f_4(W')A_4\rangle-\E\langle f_3(W)A_3f_4(W')A_4\rangle\big)\big]\\
&=\mathrm{sc}_{\circ\circ}[1|3]\langle A_1f_2(W')A_2\rangle\langle A_3f_4(W')A_4\rangle+\mathrm{sc}_{\circ}[1,3]\langle A_1f_2(W')A_2A_3f_4(W')A_4\rangle\\
&\quad+\mathrm{sc}_{\circ}[1]\mathrm{sc}_{\circ}[3]N^2\E\big[\big(\langle A_1f_2(W')A_2\rangle-\E\langle A_2f_2(W')A_2\rangle\big)\big(\langle A_3f_4(W')A_4\rangle-\E\langle A_3f_4(W')A_4\rangle\big)\big]\\
&\quad+\cO\Big(\frac{N^\eps\max\{\|f_1\|_{H^3},\|f_3\|_{H^3}\}}{\sqrt{N}}\Big)\\
&=\langle A_1A_2\rangle\langle A_3A_4\rangle(\mathrm{sc}_{\circ\circ}[1|3]\mathrm{sc}_{\circ}[2]\mathrm{sc}_{\circ}[4]+\mathrm{sc}_{\circ\circ}[2|4]\mathrm{sc}_{\circ}[1]\mathrm{sc}_{\circ}[3])+\langle A_1A_4\rangle\langle A_2A_3\rangle\mathrm{sc}_{\circ}[1,3]\mathrm{sc}_{\circ}[2,4]\\
&\quad+\langle A_1A_2A_3A_4\rangle\mathrm{sc}_{\circ}[1,3]\mathrm{sc}_{\circ}[2]\mathrm{sc}_{\circ}[4]+\langle A_2A_1A_4A_3\rangle\mathrm{sc}_{\circ}[1]\mathrm{sc}_{\circ}[2,4]\mathrm{sc}_{\circ}[3]\\
&\quad+\cO\Big(\frac{N^\eps\max_j\|f_j\|_{H^3}}{\sqrt{N}}\Big).
\end{align*}
In the first step, we conditioned on $W'$ and applied Corollary~\ref{cor-covarianceLL}, treating $W'$, and hence $f_j(W')$, as deterministic. After computing the leading term, Corollary~\ref{cor-covarianceLL} is applied again for $W'$. Lastly, the remaining terms are identified using the local law~\cite[Cor.~2.7]{CES-optimalLL}, which yields a total of five summands. In contrast, if $W=W'$, all terms on the right-hand side of~\eqref{eq-covfunctionsgeneral} may contribute. For $k=\ell=2$, this yields 27 terms in total (cf. \cite[Ex.~1.18]{JRcompanion}). Analogous statements hold for an arbitrary number of independent Wigner matrices with possible repetitions. We remark that the underlying combinatorial structure for  $n$ independent GUE matrices is given by the so-called \emph{non-mixing} annular non-crossing permutations resp. \emph{non-mixing} marked partitions for $n$ colors, which was also established in~\cite{MaleMingoPecheSpeicher2020} for the special case $f_1(x)=\dots=f_k(x)=x$ and general Wigner matrices.
\end{remark}

\subsection{Application to Thermalization Problems}
We now specialize Theorem~\ref{thm-functCLT} to the functions $f_j(x)=\re^{\ri t_j x}$ with ($N$-independent) numbers $t_j\in\R$. Recall that
\begin{equation}\label{eq-Heisenberg}
A(t):=\re^{\ri tW}A\re^{-\ri t W}
\end{equation}
describes the Heisenberg time evolution of an observable $A$ and that it follows from~\cite[Cor.~2.9]{CES-thermalization} that the observables $A_1(t)$ and $A_2$ become thermalized for $t\gg1$, i.e., that
\begin{displaymath}
\langle A_1(t)A_2\rangle\approx\langle A_1\rangle\langle A_2\rangle
\end{displaymath}
in the large $t$ regime. Fixing $t\sim1$ and using Theorem~\ref{thm-functCLT}, we readily conclude that the fluctuations around the thermal value are Gaussian and can give the leading terms of the variance explicitly by taking an $t\rightarrow\infty$ limit after letting $N\rightarrow\infty$. As the randomness in $\langle A_1(t)A_2\rangle$ cancels out if either $A_1=\Id$ or $A_2=\Id$, we omit the deterministic term $\langle A_1\rangle\langle A_2\rangle$ from the following result and assume w.l.o.g. that $\langle A_1\rangle=\langle A_2\rangle=0$.

\begin{corollary}\label{cor-thermalization}
Let $\kappa_4=0$, $\langle A_1\rangle=\langle A_2\rangle=0$, and $\|A_1\|,\|A_2\|\lesssim1$. Then,
\begin{displaymath}
\langle A_1(t)A_2\rangle=\langle A_1A_2\rangle\frac{J_1(2t)^2}{t^2}+\frac{\xi(t)}{N}+\cO\Big(\frac{N^\eps}{N^{3/2}}\Big),
\end{displaymath}
where $A_1(t)$ is as defined in~\eqref{eq-Heisenberg}, $J_1$ is a Bessel function of the first kind, and $\xi(t)$ is a centered Gaussian random variable. In the $t\rightarrow\infty$ limit, the variance of $\xi(t)$ satsifies the asymptotics
\begin{align}
\Var[\xi(t)]&=\langle |A_1|^2\rangle\langle |A_2|^2\rangle+\cO\Big(\frac{1}{t^2}\Big)\label{eq-varlimit}
\end{align}
and we further obtain
\begin{align}
\E\xi(t_1)\overline{\xi(t_2)}&=\langle |A_1|^2\rangle\langle |A_2|^2\rangle \Big(\frac{J_1(2(t_1-t_2))}{t_1-t_2}\Big)^2+\cO\Big(\frac{1}{(\min\{t_1,t_2\})^2}\Big).\label{eq-covlimit}
\end{align}
In particular, $\xi(t_1)$ and $\xi(t_2)$ are asymptotically uncorrelated if we take an $|t_1-t_2|\rightarrow\infty$ limit after letting $N\rightarrow\infty$.
\end{corollary}
Corollary~\ref{cor-thermalization} follows directly from Corollary~\ref{cor-covarianceLL} by specifying the test functions. We carry out the details in Section~\ref{app-thermalization}. 

\section{Central Limit Theorem for Resolvents}\label{sect-results1}
In this section, we supply the recursive definitions of the set functions $\cE[\cdot]$ and $\m[\cdot|\cdot]$ introduced in Lemmas~\ref{lem-mEexchangemeso} and~\ref{lem-covapproxmeso}, respectively, and study their properties. We further state the analog of Theorem~\ref{thm-functCLT} for the resolvent case, which constitutes the main ingredient for the proof of the multi-point functional CLT.

\subsection{The $\frac1N$ Term of $\E\langle T_{[1,k]}\rangle$}\label{sect-mEexchange}
As a first step, we revisit the function $\cE[\cdot]$, starting with its recursive definition.

\begin{definition}\label{def-E}
Let $(T_1,\dots,T_k)$ be an ordered set of $\C^{N\times N}$ matrices of the form $T_j=G_jA_j$. We define $\cE[\cdot]$ to be the set function taking values in $\C$ that satisfies the linear recursion with a source term (in the last two lines)
\begin{align}
&\cE[T_1,\dots,T_k]\NN\\
&=m_1\Big(\cE[T_2,\dots,T_{k-1},T_kA_1]+q_{1,k}\cE[T_2,\dots,T_{k-1},G_kA_1]\langle A_k\rangle\NN\\
&\quad+\sum_{j=1}^{k-1}\cE[T_1,\dots, T_{j-1},G_j]\big(\m[T_j,\dots,T_k]+q_{1,k}\m[T_j,\dots,T_{k-1},G_k]\langle A_k\rangle\big)\NN\\
&\quad+\sum_{j=2}^k\m[T_1,\dots,T_{j-1},G_j]\big(\cE[T_j,\dots, T_k]+q_{1,k}\cE[T_j,\dots,T_{k-1},G_k]\langle A_k\rangle\big)\NN\\
&\quad+\sum_{1\leq r\leq s\leq t\leq k}\langle M_{[r]}\odot M_{[s,t]}\rangle\langle M_{[r,s]}\odot(M_{[t,k]}A_k)\rangle\NN\\
&\quad+q_{1,k}\sum_{1\leq r\leq s\leq t\leq k}\langle M_{[r]}\odot M_{[s,t]}\rangle\langle M_{[r,s]}\odot M_{[t,k]}\rangle\langle A_k\rangle\Big)\label{eq-Erecursion}
\end{align}
and the initial condition $\cE[\emptyset]=0$. Recall that $\odot$ denotes the Hadamard product, $M_{[\cdot]}$ was defined through the recursion in~\eqref{eq-matrixMrec}, $\m[\cdot]$ was defined in~\eqref{eq-defM}, and $q_{1,k}=\frac{m_1m_k}{1-m_1m_k}$.
\end{definition}

We remark that $M_{[k]}$ is diagonal whenever $A_1=\dots=A_k=\Id$. In this case, the last two lines of~\eqref{eq-Erecursion} are readily evaluated and the recursion simplifies to
\begin{align*}
\cE[G_1,\dots,G_k]=\frac{m_1}{1-m_1m_k}\Big(&\cE[G_2,\dots,G_k]+\sum_{j=1}^{k-1}\cE[G_1,\dots,G_j]m[j,\dots,k]\\
&+\sum_{j=2}^km[1,\dots,j]\cE[G_j,\dots, G_k]\\
&+\sum_{1\leq r\leq s\leq t\leq k}m[1,\dots,r]m[r,\dots,s]m[s,\dots,t]m[t,\dots,k]\Big),
\end{align*}
where $m[\cdot]$ denotes the iterated divided differences in~\eqref{eq-itdivdif}. Note that $\cE[T_1,\dots,T_k]$ is generally not of order one, but its size is given in terms of $\eta_*$ and the number of traceless matrices among $A_1,\dots,A_k$. We give a simple bound in the following lemma. The proof is carried out in Section~\ref{sect-sizeproofs}.
\begin{lemma}\label{lem-Esize}
Under the assumptions of Lemma~\ref{lem-mEexchangemeso}, let $a$ among the matrices $A_1,\dots,A_k$ be traceless. Then,
\begin{equation}\label{eq-Eestimate}
|\cE[T_1,\dots,T_k]|\lesssim \frac{1}{\eta_*^{k-1-\lceil a/2\rceil}}.
\end{equation}
The bound is sharp if not all $\Im z_j$ have the same sign and $a$ is even.
\end{lemma}

\begin{remark}
Lemma~\ref{lem-Esize} shows that the sub-leading term in~\eqref{eq-mEexchangemeso} involving $\cE[\cdot]$ may be smaller than the $\smash{\cO(N^\eps/(N\sqrt{N\eta_*}\eta_*^{k-a/2})}$ error term in some regimes. However, as we only apply~\eqref{eq-mEexchangemeso} in the form
\begin{displaymath}
\E\langle T_{[1,k]}\rangle=\m[T_1,\dots,T_k]+\cO\Big(\frac{1}{N\eta_*^{k-1-a/2}}+\frac{N^\eps}{N\, \sqrt{N\eta_*}\ \eta_*^{k-a/2}}\Big)
\end{displaymath}
for the proof of the CLT in the resolvent case in Section~\ref{sect-proof-resolventCLT}, a more careful resolution of the error is not needed.
\end{remark}

Note that applying~\eqref{eq-Erecursion} once yields the formula
\begin{align}
\cE[T_1]=\langle A_1\rangle \frac{m_1^5}{1-m_1^2}=\langle A_1\rangle m_1'm_1^3,\label{eq-Esmallest}
\end{align}
from which we readily reobtain~(93) of~\cite{CES-functCLT} by Lemma~\ref{lem-mEexchangemeso}.

\medskip
Having identified the function $\cE[\cdot]$ as the $1/N$ term of $\E\langle T_{[1,k]}\rangle$, we may use identities that are valid on the random matrix side, i.e., the left-hand side of~\eqref{eq-mEexchangemeso}, to derive further identities for $\cE[\cdot]$ (cf. the "meta argument" below~\cite[Lem.~4.1]{CES-optimalLL}). They are listed in Corollary~\ref{cor-Eproperties} below and the proof is given in Appendix~\ref{app-meta}. In fact, these identities can also be proven from Definition~\ref{def-E} directly, however, using~\eqref{eq-mEexchangemeso} allows for a shorter proof. Note that $\m[\cdot]$ satisfies the same properties by~\cite[Lem.~5.4]{CES-thermalization}.

\begin{corollary}\label{cor-Eproperties}
Let $k\in\N$ and $(T_1,\dots,T_k)$ be an ordered set of $\C^{N\times N}$ matrices of the form $T_j=G_jA_j$. Then
\begin{itemize}
\item[(i)] $\cE[\cdot]$ is cyclic in the sense that $\cE[T_1,\dots,T_k]=\cE[T_2,\dots,T_k,T_1]$.
\item[(ii)] Whenever $z_1\neq z_k$ and $A_k=\Id$, we have
\begin{equation}\label{eq-Edivdif1}
\cE[T_1,\dots,T_{k-1},G_k]=\frac{\cE[T_2,\dots,T_{k-1},G_kA_1]-\cE[T_1,\dots,T_{k-1}]}{z_k-z_1}.
\end{equation}
\item[(iii)] Whenever $A_1=\dots=A_k=\Id$ and the spectral parameters $z_1,\dots,z_k$ are distinct, $\cE[\cdot]$ has a divided difference structure, i.e.,
\begin{equation}\label{eq-Edivdif2}
\cE[G_1,\dots,G_k]=\frac{\cE[G_2,\dots,G_k]-\cE[G_1,\dots,G_{k-1}]}{z_k-z_1}
\end{equation}
and we have the closed formula
\begin{equation}\label{eq-Edivdif3}
\cE[G_1,\dots,G_k]=\sum_{j=1}^k\prod_{i\neq j}\frac{1}{z_i-z_j}\cE[G_j]=\sum_{j=1}^k\prod_{i\neq j}\frac{m[i,j]}{m_i-m_j}\cE[G_j]
\end{equation}
with $m[\cdot]$ as in~\ref{eq-1slotdivdif} and $\cE[G_j]=m_j'm_j^3$. Moreover, $\cE[\cdot]$ is invariant under any permutation of $z_1,\dots,z_k$ in this case.
\end{itemize}
\end{corollary}

\subsection{Statement of the Resolvent Central Limit Theorem}\label{sect-resolventCLT}
The main result of this section is a CLT for resolvents that identifies the joint distribution of multiple modes of the type~\eqref{def-modes}, i.e., $\smash{X^{(k_i,a_i)}_{\alpha_i}}$ with different $k_i$, $a_i$, and $\alpha_i$, as asymptotically Gaussian in the sense of moments. We start by defining the set function $\m[\cdot|\cdot]$, which characterizes the limiting covariance of the random variables $\smash{X_{\alpha}}$ and $\smash{X_{\beta}}$ involving two distinct multi-indices $\alpha$ and $\beta$. Note that we only use the following recursive definition in the present work. However, closed formulas are obtained in the companion paper~\cite{JRcompanion}.

\begin{definition}\label{def-M}
Let $S_1=(T_1,\dots,T_{k'})$ and $S_2=(T_{k'+1},\dots,T_{k'+\ell'})$ be two (ordered) finite sets of complex $N\times N$-matrices of the form $T_j=G_jA_j$. We define $\m[\cdot|\cdot]$ as the (deterministic) function of pairs of sets $S_1,S_2$ with values in $\C$ and the following properties:
\begin{itemize}
\item[(i)] Symmetry: $\m[\cdot|\cdot]$ is symmetric under the interchanging of its arguments, i.e., for any sets $B_1\subseteq S_1,B_2\subseteq S_2$ we have
\begin{displaymath}
\m[(T_i,i\in B_1)|(T_j, j\in B_2)]=\m[(T_j, j\in B_2)|(T_i,i\in B_1)].
\end{displaymath}
\item[(ii)] Initial condition: For any sets $B_1\subseteq S_1,B_2\subseteq S_2$ we have
\begin{equation}\label{eq-Minitial}
\m[(T_i,i\in B_1)|\emptyset]=\m[\emptyset|(T_j, j\in B_2)]=0.
\end{equation}
\item[(iii)] Recursion: Let $B_1\subseteq S_1$ and $B_2\subseteq S_2$ be ordered subsets with $|B_1|=k\leq k'$ and $|B_2|=\ell\leq\ell'$ elements, respectively. We index the matrices in $B_1$ by $[k]$ and the matrices in $B_2$ by $[k+1,k+\ell]$. The function $\m[\cdot|\cdot]$ satisfies the following linear recursion
\begin{align}
&\m[T_1,\dots,T_k|T_{k+1},\dots,T_{k+\ell}]\NN\\
&=m_1\Bigg(\m[T_2,\dots,T_{k-1},G_kA_kA_1|T_{k+1},\dots,T_{k+\ell}]\NN\\
&\quad+q_{1,k}\m[T_2,\dots,T_{k-1},G_kA_1|T_{k+1},\dots,T_{k+\ell}]\langle A_k\rangle\label{eq-Mrecursion}\\
&\quad+\sum_{j=1}^{k-1}\m[T_1,\dots,T_{j-1},G_j|T_{k+1},\dots,T_{k+\ell}]\big(\m[T_j,\dots,T_k]+q_{1,k}\m[T_j,\dots,T_{k-1},G_k]\langle A_k\rangle\big)\NN\\
&\quad+\sum_{j=2}^k\m[T_1,\dots,T_{j-1},G_j]\Big(\m[T_j,\dots,T_k|T_{k+1},\dots,T_{k+\ell}]\NN\\
&\quad\quad+q_{1,k}\m[T_j,\dots,T_{k-1},G_k|T_{k+1},\dots,T_{k+\ell}]\langle A_k\rangle\Big)+\mathfrak{s}_{GUE}+\mathfrak{s}_\kappa\Bigg)\NN
\end{align}
where the source terms $\mathfrak{s}_{GUE}$ and $\mathfrak{s}_\kappa$ are given by
\begin{align}
\mathfrak{s}_{GUE}&:=\sum_{j=1}^\ell \Big(\m[T_1,\dots,T_k,T_{k+j},\dots,T_{k+j-1},G_{k+j}]\NN\\
&\quad\quad +q_{1,k}\m[T_1,\dots,T_{k-1},G_k,T_{k+j},\dots,T_{k+j-1},G_{k+j}]\langle A_k\rangle\Big)\label{eq-sourceGUE}\\
\mathfrak{s}_{\kappa}&:=\kappa_4\sum_{r=1}^k\sum_{s=k+1}^{k+\ell}\Big(\sum_{t=k+1}^s\langle M_{[r]}\odot M_{(s,\dots,k+\ell,k+1,\dots,t)}\rangle\langle(M_{[r,k]}A_k)\odot M_{[t,s]}\rangle\NN\\
&\quad\quad+\sum_{t=s}^{k+\ell}\langle M_{[r]}\odot M_{[s,t]}\rangle\langle(M_{[r,k]}A_k)\odot M_{(t,\dots,k+\ell,k+1,\dots,s)}\rangle\Big)\NN\\
&\quad+\kappa_4q_{1,k}\sum_{r=1}^k\sum_{s=k+1}^{k+\ell}\Big(\sum_{t=k+1}^s\langle M_{[r]}\odot M_{(s,\dots,k+\ell,k+1,\dots,t)}\rangle\langle M_{[r,k]}\odot M_{[t,s]}\rangle\NN\\
&\quad\quad+\sum_{t=s}^{k+\ell}\langle M_{[r]}\odot M_{[s,t]}\rangle\langle M_{[r,k]}\odot M_{(t,\dots,k+\ell,k+1,\dots,s)}\rangle\Big)\langle A_k\rangle.\label{eq-sourcekappa}
\end{align}
Recall that $\odot$ denotes the Hadamard product, $q_{1,k}=\frac{m_1m_k}{1-m_1m_k}$ with $m_1,m_k$ as in~\eqref{eq-defm}, $M_{(\dots)}$ was introduced in~\eqref{eq-matrixMordered}, and $\m[\cdot]$ was defined in~\eqref{eq-defM}.
\end{itemize}
\end{definition}

\begin{remark}
The special role of $m_1$ in~\eqref{eq-Mrecursion} is a result of the identity~\eqref{eq-m1identity} used for the proof of Lemma~\ref{lem-covapproxmeso} in Section~\ref{sect-covapproxmeso} below. Similar to the recursion for $\m[\cdot]$ in~\cite[Lem.~4.1]{CES-optimalLL}, it is possible to derive a version of~\eqref{eq-Mrecursion} for every $j=2,\dots,k$ that singles out the factor $m_j$  instead of $m_1$ on the right-hand side, i.e.,~\eqref{eq-Mrecursion} is only one element in a family of equivalent recursions for~$\m[\cdot|\cdot]$.
\end{remark}

The linearity of the recursion and the two different types of source terms induce the decomposition~\eqref{eq-Mdecomp}, where $\m_{GUE}[\cdot|\cdot]$ satisfies~\eqref{eq-Mrecursion} for $\kappa_4=0$, and $\kappa_4\m_\kappa[\cdot|\cdot]$ satisfies~\eqref{eq-Mrecursion} without $\mathfrak{s}_{GUE}$. We remark that by~\cite[Thm.~3.4]{CES-thermalization}, both $\m[\cdot]$ and $M_{(\cdot)}$ are fully expressible as functions of $A_1,\dots,A_{k+\ell}$ and $m_1,\dots,m_{k+\ell}$. Hence, the same holds for the source term $\mathfrak{s}_{GUE}+\mathfrak{s}_\kappa$ in~\eqref{eq-Mrecursion}, eventually making $\m[\cdot|\cdot]$ a function of the same quantities.
Similarly, we have the decomposition
\begin{equation}\label{eq-mdecomp}
m[\cdot|\cdot]=m_{GUE}[\cdot|\cdot]+\kappa_4m_\kappa[\cdot|\cdot],
\end{equation}
for the function $m[\cdot|\cdot]$ defined by the relation
\begin{equation}\label{eq-2slotdivdif}
m[1,\dots,l|k+1,\dots,k+\ell]:=\m[G_1,\dots,G_k|G_{k+1},\dots,G_{k+\ell}]
\end{equation}
in the special case $A_1=\dots=A_k=\Id$.

\medskip
Next, we consider the size of $\m[\cdot|\cdot]$. We have the following bounds, which we prove in Section~\ref{sect-sizeproofs}.
\begin{lemma}\label{lem-sizems}
Under the assumptions of Lemma~\ref{lem-covapproxmeso}, we have the estimates
\begin{align}
\big|\m_{GUE}[T_1,\dots, T_k|T_{k+1},\dots,T_{k+\ell}]\big|&\lesssim\frac{1}{\eta_*^{k+\ell-\lceil (a+b)/2\rceil}},\label{eq-checkerboard}\\
\big|\m_{\kappa}[T_1,\dots, T_k|T_{k+1},\dots,T_{k+\ell}]\big|&\lesssim\frac{1}{\eta_*^{k+\ell-1-\lceil (a+b)/2\rceil}}\label{eq-mkappabound}.
\end{align}
Both bounds are sharp only if not all $\Im z_j$ have the same sign. In particular, $\m_{\kappa}[\cdot|\cdot]$ is dominated by $\m_{GUE}[\cdot|\cdot]$ on all mesoscopic scales and it holds that
\begin{equation}\label{eq-msize}
\big|\m[T_1,\dots, T_k|T_{k+1},\dots,T_{k+\ell}]\big|\lesssim\frac{1}{\eta_*^{k+\ell-\lceil (a+b)/2\rceil}}.
\end{equation}
\end{lemma}

After this preparation, we state a CLT for resolvents which generalizes ~\cite[Thm.~4.1]{CES-functCLT} to handle resolvent chains of arbitrary length in the setting considered. The proof follows by induction on the number of factors $\smash{X_{\alpha_j}^{(k_j,a_j)}}$ using the bounds from Lemma~\ref{lem-covapproxmeso} and its proof as input. We give it in Section~\ref{sect-proof-resolventCLT}.

\begin{theorem}[CLT for resolvents]\label{thm-resolventCLTmeso}
Fix $p\in\N$, let $\alpha_1,\dots,\alpha_p$ be multi-indices, and let $W$ be a Wigner matrix satisfying Assumption~\ref{as-Wigner}. Moreover, for every $j=1,\dots,p$ pick a set of spectral parameters $\smash{z_1^{(j)},\dots,z_{k_j}^{(j)}}$ such that either all sets satisfy Case 1 or all sets satisfy Case~2 of Assumption~\ref{as-spectralmeso}, and denote $\smash{\eta_*=\min_{i,j}|\Im z_i^{(j)}|}$. Moreover, for every $j=1,\dots,p$, pick deterministic matrices $A_1^{(j)},\dots,A_{k_j}^{(j)}$ with $\|A_i^{(j)}\|\lesssim1$ such that $a_j$ of them are traceless. Then,
\begin{align}
N^p\E\Big(\prod_{j=1}^p X_{\alpha_j}^{(k_j,a_j)}\Big)=\sum_{Q\in Pair([p])}\prod_{\{i,j\}\in Q}\m[\alpha_i|\alpha_j]+\cO\Big(\frac{N^\eps}{\sqrt{N\eta_*}\ \prod_{l=1}^p\eta_*^{k_l-a_l/2}}\Big)\label{eq-resolventCLTmeso}
\end{align}
for any $\eps>0$. Here, $\m[\cdot|\cdot]$ is as in Definition~\ref{def-M} and $Pair(S)$ denotes the pairings of a set~$S$. Equation~\eqref{eq-resolventCLTmeso} establishes an asymptotic version of Wick's rule and hence identifies the joint limiting distribution of the random variables $\smash{(X_{\alpha_j}^{(k_j,a_j)})_j}$ as asymptotically complex Gaussian in the sense of moments in the limit $N\eta_*\rightarrow\infty$.
\end{theorem}

\begin{remark}[Independent modes]
Note that~\eqref{eq-msize} implies that two modes $\smash{X^{(k_1,a_1)}_{\alpha_1}}$ and $\smash{X^{(k_2,a_2)}_{\alpha_2}}$ in Theorem~\ref{thm-resolventCLTmeso} are asymptotically uncorrelated whenever $a_1+a_2$ is odd and $\eta_*\ll1$. This feature is exclusive to the mesoscopic regime, as all modes contribute equally for the macroscopic regime, i.e., if $\eta_*\gtrsim1$. In the case $k=1$, we may also write the deterministic matrix as
\begin{equation}\label{eq-modesdecomp}
A=\langle A\rangle\Id+\mathring{A}_d+\mathring{A}_{od}
\end{equation}
with $A_d$ and $A_{od}$ denoting the diagonal and off-diagonal part of $\mathring{A}=A-\langle A\rangle\Id$, respectively. The three resulting modes $\langle A\rangle\Tr f(W)$, $\Tr f(W)\mathring{A}_d$, and $\Tr f(W)\mathring{A}_{od}$ are asymptotically uncorrelated as $N\eta_*\rightarrow\infty$ (cf.~\cite[Thm.~2.4]{CES-functCLT}). Since this is a consequence of $\langle B_1B_2\rangle$ and $\langle B_1\rangle\langle B_2\rangle$ vanishing whenever $B_1\neq B_2$ and $B_1,B_2\in\{\Id,\mathring{A}_d,\mathring{A}_{od}\}$, this phenomenon is exclusive to the $k=1$ case and decomposing $A_1,\dots,A_k$ for $k\geq2$ according to~\eqref{eq-modesdecomp} does not yield $3^k$ uncorrelated modes in general. 
\end{remark}

Similar to Corollary~\ref{cor-Eproperties}, we may use identities that are valid on the random matrix side, i.e., the left-hand side of~\eqref{eq-2ndorderLL1meso}, to derive further identities among the recursively defined quantities $m[\cdot|\cdot]$ and $\m[\cdot|\cdot]$. The proof is analogous to the proof of Corollary~\ref{cor-Eproperties} and hence omitted. We refer to Appendix~\ref{app-meta} for the setup of the necessary "meta argument".

\begin{corollary}\label{cor-mproperties}
Let $S_1,S_2\neq\emptyset$ be two ordered multi-sets. Then
\begin{itemize}
\item[(i)] $m[S_1|S_2]$ is invariant under any permutation of the elements of $S_1$ as well as $S_2$.
\item[(ii)] $\m[\cdot|\cdot]$ is cyclic in the sense that $\m[(T_j,j\in S_1)|T_1,\dots,T_k]=\m[(T_j,j\in S_1)|T_2,\dots,T_k,T_1]$.
\item[(iii)] Whenever the spectral parameters indexed by $S_1$ and $S_2$ are distinct, $m[\cdot|\cdot]$ has an entry-wise divided difference structure, i.e.,
\begin{equation}\label{eq-divdif1}
m[S_1|1,\dots,k]=\frac{m[S_1|2,\dots,k]-m[S_1|1,\dots,k-1]}{z_k-z_1},
\end{equation}
and we have the closed formula
\begin{align}
m[S_1|S_2]&=\sum_{(s,t)\in S_1\times S_2}\Big(\prod_{\substack{i\in S_1,\\i\neq s}}\frac{1}{z_i-z_s}\prod_{\substack{j\in S_2,\\j\neq t}}\frac{1}{z_j-z_t}\Big)m[s|t]\NN\\
&=\sum_{(s,t)\in S_1\times S_2}\Big(\prod_{\substack{i\in S_1,\\i\neq s}}\frac{m[i,s]}{m_i-m_s}\prod_{\substack{j\in S_2,\\j\neq t}}\frac{m[j,t]}{m_j-m_t}\Big)m[s|t]\label{eq-divdif2}
\end{align}
with $m[s|t]=\frac{m_s'm_t'}{(1-m_sm_t)^2}$. Recall that $m[\cdot|\cdot]$ was defined in~\eqref{eq-2slotdivdif}.
\item[(iv)] Whenever $z_1\neq z_k$ and $A_k=\Id$, we further have
\begin{align}\label{eq-divdif3}
&\m[(T_j,j\in S_1)|T_1,\dots,T_{k-1},G_k]\NN\\
&=\frac{\m[(T_j,j\in S_1)|T_2,\dots,T_{k-1},G_kA_1]-\m[(T_j,j\in S_1)|T_1,\dots,T_{k-1}]}{z_k-z_1}.
\end{align}
\end{itemize}
\end{corollary}
Moreover, we have the following alternative integral representation for $m_{GUE}[\cdot|\cdot]$ (cf. decomposition in~\eqref{eq-mdecomp}). The proof of Corollary~\ref{cor-mintegral} is carried out in Section~\ref{sect-covapproxmeso} below.

\begin{corollary}\label{cor-mintegral}
Let $k,\ell\in\N$. Then,
\begin{align}
&m_{GUE}[1,\dots,k|k+1,\dots,k+\ell]\label{eq-integralrep}\\
&\quad=\frac{1}{2}\int\int\Big(\sum_{i=1}^k\frac{1}{(x-z_i)^2}\cdot\prod_{j\neq i}\frac{1}{x-z_j}\Big)\Big(\sum_{i=k+1}^{k+\ell}\frac{1}{(y-z_i)^2}\cdot\prod_{j\neq i}\frac{1}{y-z_j}\Big)u(x,y)\dx x\dx y\NN
\end{align}
with the kernel $u:[-2,2]\times[-2,2]\rightarrow\R$ in~\eqref{eq-kernel}.
\end{corollary}

\begin{remark}
It is readily checked that the kernel $u$ is non-negative and has a logarithmic singularity at $x=y$. Using that the two-body stability operator of the underlying Dyson equation~\eqref{eq-mselfcon} is given by $\cB(z_1,z_2)=1-m(z_1)m(z_2)$, we can also express~\eqref{eq-kernel} in terms of $1\times1$ determinants as
\begin{displaymath}
u(x,y)=-\frac{1}{2\pi^2}\Re\big(\ln(\det[\cB(x+\ri0,y+\ri0)])-\ln(\det[\cB(x+\ri0,y+\ri0)])\big)
\end{displaymath}
to match the formulas in~\cite[Sect.~7]{Vova2023} for $k=\ell=1$ and $A_1=A_2=\Id$. Note that $W$ being a GUE matrix corresponds to the choice $\beta=2$ and $\cC^{(4)}=0$ in the notation of~\cite{Vova2023}. 
\end{remark}

\section{Proofs}\label{sect-proofs}

\subsection{Proof of Lemmas~\ref{lem-Esize} and~\ref{lem-sizems} (Size of $\cE[\cdot]$ and $\m[\cdot|\cdot]$)}\label{sect-sizeproofs}

In this section, we prove the estimates identifying the size of the deterministic approximations $\cE[\cdot]$ and $\m[\cdot|\cdot]$. We start by noting two bounds for $q_{1,2}$ that are used for both proofs.

\begin{lemma}\label{lem-qbounds}
Let $z_1,z_2\in\C$ and define the constants $\eta_*:=\min\{|\Im z_1|,|\Im z_2|\}$ as well as $\zeta:=\pi^{-1}\min\{|\Im m_1|,|\Im m_2|\}$. Then,
\begin{align*}
q_{1,2}=m[1,2]=\frac{m_1m_2}{1-m_1m_2}\lesssim\begin{cases}\zeta^{-1},\quad \text{if }\Im z_1,\Im z_2\text{ have the same sign},\\ \eta_*^{-1}, \quad \text{if }\Im z_1,\Im z_2\text{ have opposite signs}.\end{cases}
\end{align*}
\end{lemma} 

The estimates in Lemma~\ref{lem-qbounds} are immediate from the explicit form $m(z)=(-z+\sqrt{z^2-4})/2$ of the solution to~\eqref{eq-mselfcon}. Next, we establish the estimate for $\cE[\cdot]$.

\begin{proof}[Proof of Lemma~\ref{lem-Esize}]
We show~\eqref{eq-Eestimate} by induction. As the base case $k=1$ readily follows from~\eqref{eq-Esmallest}, assume that the bound in~\eqref{eq-Eestimate} holds for up to $k-1$ matrices $T_1,\dots,T_{k-1}$. W.l.o.g. assume further that $A_k$ is either traceless or equal to the identity matrix.

\medskip
We start by considering the case where $A_k=\Id$ and $\Im z_1$ and $\Im z_k$ have opposite signs. Here,~\eqref{eq-Eestimate} follows immediately by using Corollary~\ref{cor-Eproperties}(ii) and applying the induction hypothesis for $\cE[T_2,\dots,T_{k-1},G_kA_1]$ and $\cE[T_1,\dots,T_{k-1}]$, respectively. Note that $|z_1-z_k|\geq 2\eta_*$ by assumption, which completes the bound for the right-hand side of~\eqref{eq-Edivdif1}.

\medskip
In the remaining cases,~\eqref{eq-Eestimate} follows from the recursion~\eqref{eq-Erecursion}, which allows rewriting $\cE[T_1,\dots,T_k]$ in terms of $\m[\cdot]$ and values of $\cE[\cdot]$ for which the induction hypothesis applies. Whenever $\Im z_1$ and $\Im z_k$ have the same sign, Lemma~\ref{lem-qbounds} yields $q_{1,k}\lesssim \pi/\min\{|\Im m_1|,|\Im m_k|\}$. Note that $\Im m(z)$ can be bounded from below independently of $\eta_*$ for both the macroscopic scale and the the bulk regime of the mesoscopic scales (cf. Assumption~\ref{as-spectralmeso}). Thus, estimating the right-hand side of~\eqref{eq-Erecursion} using Lemma~\ref{lem-mbounds} and the induction hypothesis yields the claim. In particular, we obtain $\eta_*$ with the exponent $-(k-1-\lceil\frac{a}{2}\rceil)$ by using the inequality $\lceil\frac{x}{2}\rceil+\lceil\frac{y}{2}\rceil\geq\lceil\frac{x+y}{2}\rceil$ for $x,y\in\N$ to combine the powers of $\eta_*$ when products with $M_{[\cdot]}$ are considered.

\medskip
Whenever $\Im z_1$ and $\Im z_k$ have opposite signs, the prefactor $q_{1,k}$ is of size $\eta_*^{-1}$. However, it only remains to consider the case $\langle A_k\rangle=0$ for this setting, in which the terms involving $q_{1,k}$ on the right-hand side of~\eqref{eq-Erecursion} do not contribute. Hence,~\eqref{eq-Eestimate} again follows from Lemma~\ref{lem-mbounds} and the induction hypothesis. 
\end{proof}

Next, we show the estimates for $\m_{GUE}[\cdot|\cdot]$, $\m_{\kappa}[\cdot|\cdot]$, and $\m[\cdot|\cdot]$. To illustrate the tools at hand, the bound for $\m_{GUE}[\cdot|\cdot]$ is obtained from the explicit formula in~\cite[Thm.~2.4]{JRcompanion} while the bound for $\m_{\kappa}[\cdot|\cdot]$ is proved using the recursion~\eqref{eq-Mrecursion}. We start with a lemma.

\begin{lemma}\label{lem-moremsizes}
Under the assumptions of Lemma~\ref{lem-covapproxmeso}, let $A_1=\dots=A_{k+\ell}=\Id$ and $\kappa_4=0$. Then,
\begin{align}
|m_{GUE}[1,\dots,k|k+1,\dots,k+\ell]|&\lesssim \frac{1}{\eta_*^{k+\ell}}\label{eq-mGUEbound}\\
|m_{\circ\circ}[1,\dots,k|k+1,\dots,k+\ell]|&\lesssim \frac{1}{\eta_*^{k+\ell}}\label{eq-mcirccircbound}
\end{align}
where $m_{\circ\circ}[\cdot|\cdot]$ denotes the second-order free cumulant function associated with the iterated divided differences $m[\cdot]$ and $m_{GUE}[\cdot|\cdot]$. Both bounds are sharp only if not all $\Im z_j$ have the same sign.
\end{lemma}

\begin{proof}[Proof of Lemma~\ref{lem-moremsizes}]
The bound~\eqref{eq-mGUEbound} follows by induction on $k$ and $\ell$. As the base case $k=\ell=1$ is covered by Example~\ref{ex-mesobounds}, assume that the bound for $m_{GUE}[\cdot|\cdot]$ holds for up to $k-1$ indices in the first argument and a fixed number $\ell$ indices in the second argument. Recall that $m_{GUE}[\cdot|\cdot]$ is symmetric under the interchanging of its arguments by Definition~\ref{def-M}(i) such that it is sufficient to carry out the induction step for one of the arguments only. We distinguish two cases for $z_k$ depending on the sign of its imaginary part.

\medskip
\underline{\smash{Case 1: $\Im z_1$ and $\Im z_k$ have the same sign:}} We note that $|q_{1,k}|\leq \pi/\min\{|\Im m_1|,|\Im m_k|\}$ by Lemma~\ref{lem-qbounds}, where the right-hand side can be bounded from above independently of $\eta_*$ under Assumption~\ref{as-spectralmeso}. Rewriting $m_{GUE}[1,\dots,k|k+1,\dots,k+\ell]$ using the recursion~\eqref{eq-Mrecursion}, the bound~\eqref{eq-mGUEbound} follows directly from the induction hypothesis and the estimate for $m[\cdot]$ from Lemma~\ref{lem-mbounds}.

\medskip
\underline{\smash{Case 2: $\Im z_1$ and $\Im z_k$ have opposite signs:}} Recalling that $m_{GUE}[\cdot|\cdot]$ has a divided difference structure (cf. Corollary~\ref{cor-mproperties}), it follows from the induction hypothesis that
\begin{align*}
&|m_{GUE}[1,\dots,k|k+1,\dots,k+\ell]|\\
&=\Big|\frac{m_{GUE}[2,\dots,k|k+1,\dots,k+\ell]-m_{GUE}[1,\dots,k-1|k+1,\dots,k+\ell]}{z_1-z_k}\Big|\\
&\lesssim\frac{1}{\eta_*^{k-1+\ell}|z_1-z_k|}.
\end{align*}
As $\Im z_1$ and $\Im z_k$ are assumed to have opposite signs, we have $|z_1-z_k|\geq2\eta_*$, which gives~\eqref{eq-mGUEbound}. This concludes the induction step.

\medskip
The bound~\eqref{eq-mcirccircbound} for $m_{\circ\circ}[\cdot|\cdot]$ is an immediate consequence of the second-order moment-cumulant relation~\eqref{eq-mcrelation2} as well as the estimates in~\eqref{eq-mGUEbound} and Lemma~\ref{lem-mbounds}.
\end{proof}

\begin{proof}[Proof of Lemma~\ref{lem-sizems}]
Given Lemma~\ref{lem-moremsizes}, the bound~\eqref{eq-checkerboard} readily follows from~\cite[Thm.~2.4]{JRcompanion}. We omit the details and only remark that a leading contribution is obtained for an annular non-crossing permutation $\pi$ with $|K(\pi)|\leq k+\ell-\lceil (a+b)/2\rceil$ or a marked partition $\pi_1\times\pi_2$ satisfying $|K(\pi_1)|\leq k-\lceil a/2\rceil$ and $|K(\pi_2)|\leq l-\lceil b/2\rceil$, respectively. This implies
\begin{displaymath}
\Big|\sum_{\pi\in \NCA(k,\ell)}\Big(\prod_{B\in K(\pi)}\Big\langle \prod_{j\in B}A_j\Big\rangle\Big)\prod_{B\in\pi}m_{\circ}[B]\Big|\lesssim\Big(\frac{1}{\eta_*}\Big)^{k+\ell-\lceil (a+b)/2\rceil}
\end{displaymath}
as well as
\begin{align*}
&\Big|\sum_{\substack{\pi_1\times\pi_2\in NCP(k)\times NCP(\ell),\\ U_1\in\pi_1,U_2\in\pi_2\text{ marked}}}\Big(\prod_{\substack{B\in K(\pi_1)\\ \cup K(\pi_2)}}\Big\langle \prod_{j\in B} A_j\Big\rangle\Big)m_{\circ\circ}[U_1|U_2]\prod_{\substack{B_1\in \pi_1\setminus U_1,\\ B_2\in\pi_2\setminus U_2}}m_{\circ}[B_1]m_{\circ}[B_2]\Big|\\
&\quad\lesssim\Big(\frac{1}{\eta_*}\Big)^{k+\ell-\lceil a/2\rceil-\lceil b/2\rceil}.
\end{align*}
In particular, the two sums yielding $\m_{GUE}[\cdot|\cdot]$ only contribute equally to~\eqref{eq-checkerboard} if $\lceil\frac{a}{2}\rceil+\lceil\frac{b}{2}\rceil=\lceil\frac{a+b}{2}\rceil$.

\medskip
The bound~\eqref{eq-mkappabound} follows by induction on $k$ and $\ell$. As the base case $k=\ell=1$ is again covered by Example~\ref{ex-mesobounds}, assume that the bound for $\m_{\kappa}[\cdot|\cdot]$ holds for a multi-index of length at most $k-1$ in the first argument and a multi-index of length $\ell$ in the second argument. Recalling that $\m_{\kappa}[\cdot|\cdot]$ is symmetric under the interchanging of its arguments, it is sufficient to carry out the induction step for one of the arguments only. To simplify notation, set $\beta=\{(z_{k+1},A_{k+1}),\dots,(z_{k+\ell},A_{k+\ell})\}$. We further assume w.l.o.g. that each $A_j$ is either traceless or equal to the identity matrix. Similar to the proof of Lemma~\ref{lem-Esize}, we distinguish two cases depending on the deterministic matrices $A_1,\dots,A_k$.

\medskip
\underline{\smash{Case 1 ($\exists j$ such that $A_j=\Id$):}} We start by noting that $\m_{GUE}[\cdot|\cdot]$ satisfies (i)-(iv) of Corollary~\ref{cor-mproperties}. This implies that the same holds for $\m_{\kappa}[\cdot|\cdot]=\m[\cdot|\cdot]-\m_{GUE}[\cdot|\cdot]$. Using the divided difference structure, we rewrite $\m_{\kappa}[T_1,\dots,T_k|T_{k+1},\dots,T_{k+\ell}]$ either as a contour integral
\begin{align}
&\m_{\kappa}[T_1,\dots,T_{j-1},G_j,T_{j+1},\dots,T_k|\beta]\label{eq-mkapparewritten1}\\
&=\frac{1}{2\pi\ri}\int_{\R}\frac{\m_{\kappa}[\dots,T_{j-1},G(x+\ri\eta)A_{j+1},\dots|\beta]-\m_{\kappa}[\dots,T_{j-1},G(x-\ri\eta)A_{j+1},\dots|\beta]}{(x+\ri s\eta-z_j)(x+\ri s\eta-z_{j+1})}\dx x\NN
\end{align}
if $\Im z_j$ and $\Im z_{j+1}$ have the same sign, or as
\begin{align}
&\m_{\kappa}[T_1,\dots,T_{j-1},G_j,T_{j+1},\dots,T_k|\beta]\label{eq-mkapparewritten2}\\
&=\frac{\m_{\kappa}[T_1,\dots,T_{j-1},T_{j+1},\dots,T_k|\beta]-\m_{\kappa}[T_1,\dots,T_{j-1},T_jA_{j+1},T_{j+2},\dots,T_k|\beta]}{z_j-z_{j+1}}\NN
\end{align}
if $\Im z_j$ and $\Im z_{j+1}$ have opposite signs. Estimating~\eqref{eq-mkapparewritten1} and~\eqref{eq-mkapparewritten2} using the induction hypothesis yields~\eqref{eq-mkappabound}.

\medskip
\underline{\smash{Case 2 (all $A_1,\dots,A_k$ traceless):}} As $\m_{\kappa}[\cdot|\cdot]$ solves the recursion~\eqref{eq-Mrecursion} without the source term $\mathfrak{s}_{GUE}$ we can rewrite $\m_{\kappa}[T_1,\dots,T_k|T_{k+1},\dots,T_{k+\ell}$ in terms of $\m[\cdot]$ and values of $\m_{\kappa}[\cdot|\cdot]$ for which the induction hypothesis applies. Note that $\langle A_k\rangle=0$ implies that all terms with prefactor $q_{1,k}$ on the right-hand side of~\eqref{eq-Mrecursion} vanish. The desired estimate thus readily follows from the induction hypothesis and Lemma~\ref{lem-mbounds}. For the source term $\mathfrak{s}_{\kappa}$ in~\eqref{eq-sourcekappa}, we obtain, e.g.,
\begin{align*}
&\big|\langle M_{[r]}\odot M_{(s,\dots,k+\ell,k+1,\dots,t)}\rangle\langle M_{[r,k]}\odot M_{[t,s]}\rangle\big|\\
&\leq\frac{1}{N^2}\sum_{x,y\in[N]}\big|(M_{[r]})_{xx}(M_{(s,\dots,k+\ell,k+1,\dots,t)})_{xx}(M_{[r,k]})_{yy}(M_{[t,s]})_{yy}\big|\\
&\leq\Big(\frac{1}{\eta_*}\Big)^{k+\ell-\lceil(a+b)/2\rceil}
\end{align*}
for any $1\leq r\leq k$ and $k+1\leq s\leq t\leq k+\ell$. Recall that $\lceil\frac{x}{2}\rceil+\lceil\frac{y}{2}\rceil\geq\lceil\frac{x+y}{2}\rceil$ for any $x,y\in\N$, which yields the desired exponent $-(k+\ell-\lceil \frac{a+b}{2}\rceil)$ for $\eta_*$ when bounds are multiplied. This concludes the proof of~\eqref{eq-mkappabound}.

\medskip
The bound~\eqref{eq-msize} for $\m[\cdot|\cdot]$ is immediate form the decomposition~\eqref{eq-Mdecomp} using the estimates~\eqref{eq-checkerboard} and~\eqref{eq-mkappabound}.
\end{proof}

\subsection{Proof of Lemma~\ref{lem-mEexchangemeso} (Expansion for $\E\langle T_{[1,k]}\rangle$)}\label{sect-mEexchangemeso}
We use proof by induction to establish~\eqref{eq-mEexchangemeso}. As the base case $\E\langle T_{\emptyset}\rangle$ is trivial, assume that the expansion in Lemma~\ref{lem-mEexchangemeso} holds for resolvent chains of length up to $k-1$. We further assume w.l.o.g. that each $A_j$ is either traceless or equal to the identity matrix.

\medskip
First, consider resolvent chains that contain at least one deterministic matrix $A_j=\Id$, i.e., that are of the form $T_{[1,k]}=T_{[1,j\rangle}G_jG_{j+1}T_{\langle j+1,k]}$ with the indices $j,j+1$ being interpreted mod $k$ due to the cyclicity of the trace and the function $\cE[\cdot]$. In this case, rewriting the product $G_jG_{j+1}$ allows us to obtain the claim directly from the induction hypothesis. We distinguish two cases depending on the imaginary parts of $z_j$ and $z_{j+1}$.

\medskip
\underline{\smash{Case 1 ($\Im z_j$ and $\Im z_{j+1}$ have the same sign):}} Let $s:=\mathrm{sign}(\Im z_j)=\mathrm{sign}(\Im z_{j+1})$. By the residue theorem, we can write the product $G_jG_{j+1}$ as a contour integral (cf.~\cite[Lem.~3.2]{CES-optimalLL})
\begin{equation}\label{eq-residueGs}
G_jG_{j+1}=\frac{1}{\pi}\int_{\R}\frac{\Im G(x+\ri\eta)}{(x+\ri s\eta-z_j)(x+\ri s\eta-z_{j+1})}\dx x,\quad G_j=G(z_j),
\end{equation}
whenever $0<\eta<\Im z_j,\Im z_{j+1}$ ($s=1$) or $\Im z_j,\Im z_{j+1}<-\eta<0$ ($s=-1$). Note that both $\m[\cdot]$ and $\cE[\cdot]$ have a similar representation, as
\begin{align}
&\m[T_1,\dots,T_{j-1},G_j,T_{j+1},\dots,T_k]\NN\\
&=\frac{1}{2\pi\ri}\int_{\R}\frac{\m[\dots,T_{j-1},G(x+\ri\eta)A_{j+1},\dots]-\m[\dots,T_{j-1},G(x-\ri\eta)A_{j+1},\dots]}{(x+\ri s\eta-z_j)(x+\ri s\eta-z_{j+1})}\dx x\label{eq-residueM}
\end{align}
by the residue theorem and~\cite[Lem.~4.4]{CES-thermalization} as well as
\begin{align}
&\cE[T_1,\dots,T_{j-1},G_j,T_{j+1},\dots,T_k]\NN\\
&=\frac{1}{2\pi\ri}\int_{\R}\frac{\cE[\dots,T_{j-1},G(x+\ri\eta)A_{j+1},\dots]-\cE[\dots,T_{j-1},G(x-\ri\eta)A_{j+1},\dots]}{(x+\ri s\eta-z_j)(x+\ri s\eta-z_{j+1})}\dx x\label{eq-residueE}
\end{align}
using Corollary~\ref{cor-Eproperties}(ii). Recall that the properties listed in Corollary~\ref{cor-Eproperties} can be derived directly from the recursion~\eqref{eq-Erecursion}, i.e., their proof is independent of Lemma~\ref{lem-mEexchangemeso} and the "meta argument" in Appendix~\ref{app-meta}. Rewriting the left-hand side of~\eqref{eq-mEexchangemeso} using~\eqref{eq-residueGs} and~\ref{eq-residueM}, we obtain an integral involving a resolvent chain of length $k-1$. Hence, by the induction hypothesis,
\begin{align*}
&\E(\langle T_{[1,k]}\rangle-\m[T_1,\dots,T_k])\\
&=\frac{\kappa_4}{2\pi\ri N}\int_{\R}\frac{\cE[\dots,T_{j-1},G(x+\ri\eta)A_{j+1},\dots]-\cE[\dots,T_{j-1},G(x-\ri\eta)A_{j+1},\dots]}{(x+\ri s\eta-z_j)(x+\ri s\eta-z_{j+1})}\dx x\\
&\quad+\cO\Big(\frac{N^\eps}{N\, \sqrt{N\eta_*}\ \eta_*^{k-1-a/2}}\Big).
\end{align*}
Evaluating the integral using~\eqref{eq-residueE} gives~\eqref{eq-mEexchangemeso} as desired.

\medskip
\underline{\smash{Case 2 ($\Im z_j$ and $\Im z_{j+1}$ have opposite signs):}} Applying the resolvent identity~\eqref{eq-resolventid} and the divided difference structure of $\m[\cdot]$ (cf.~\cite[Lem.~5.4]{CES-thermalization}) yields
\begin{align*}
&\E(\langle T_{[1,k]}\rangle-\m[T_1,\dots,T_k])\\
&=\E\Big(\frac{\langle T_{[1,j\rangle}G_jA_{j+1}T_{\langle j+1,k]}\rangle-\langle T_{[1,j\rangle}T_{[j+1,k]}\rangle}{z_j-z_{j+1}}\\
&\quad-\frac{\m[T_1,\dots,T_{j-1},G_jA_{j+1},T_{j+2},\dots,T_k]-\m[T_1,\dots,T_{j-1},T_{j+1},T_{j+2},\dots,T_k]}{z_j-z_{j+1}}\Big)\\
&=\frac{\kappa_4}{N}\frac{\cE[T_1,\dots,T_{j-1},G_jA_{j+1},T_{j+2},\dots,T_k]-\cE[T_1,\dots,T_{j-1},T_{j+1},T_{j+2},\dots,T_k]}{z_j-z_{j+1}}\\
&\quad+\cO\Big(\frac{N^\eps}{N\, \sqrt{N\eta_*}\ \eta_*^{k-1-a/2}\ |z_j-z_{j+1}|}\Big)
\end{align*}
by the induction hypothesis. As $\Im z_j$ and $\Im z_{j+1}$ are assumed to have opposite signs, it follows that $|z_j-z_{j+1}|\geq 2\eta_*$. The claim is now immediate from~\eqref{eq-Edivdif1}.

\medskip
It remains to consider $T_{[1,k]}$ for which all matrices $A_1,\dots,A_k$ are traceless. Here, we start the induction step by introducing
\begin{equation}\label{eq-underline}
\underline{Wf(W)}:= Wf(W)-\widetilde{\E}\widetilde{W}(\partial_{\widetilde{W}}f)(W)
\end{equation}
with $\partial_{\widetilde{W}}$ denoting the directional derivative in direction $\smash{\widetilde{W}}$ and $\smash{\widetilde{W}}$ denoting an independent GUE matrix with expectation $\smash{\widetilde{\E}}$. By construction, the renormalization in~\eqref{eq-underline} cancels out the second-order term in the cumulant expansion of $\E Wf(W)$. In particular, $\smash{\E \underline{Wf(W)}=0}$ whenever $W$ itself is a GUE matrix. Applying~\eqref{eq-underline} for the resolvent $f(W)=(W-z)^{-1}$ yields the formulas
\begin{align}
\underline{WG_1}&=WG_1+\langle G_1\rangle G_1,\NN\\
\underline{WT_1\dots T_k}&=\underline{WG_1}A_1T_{[2,k]}+\sum_{j=2}^k\langle T_{[1,j\rangle}G_j\rangle T_{[j,k]},\label{eq-underlinechain}
\end{align}
and we further recall the identities
\begin{align}
\langle G_1-m_1\rangle&=\frac{1}{1-m_1^2}(-m_1\langle\underline{WG_1}\rangle+m_1\langle G_1-m_1\rangle^2)\label{eq-Gidentity1}\\
\langle T_1\rangle-\m[T_1]&=-m_1\langle\underline{WT_1}\rangle+m_1^2\langle G_1-m_1\rangle\langle A_1\rangle+m_1\langle G_1-m_1\rangle\langle T_1-\m[T_1]\rangle\label{eq-Gidentity2}
\end{align}
from (96) in~\cite{CES-functCLT}. To complete the induction step, we need the analog of~\eqref{eq-Gidentity2} for general $k\geq1$. This allows rewriting $N(\langle T_{[1,k]}\rangle-\m[T_1,\dots,T_k])$ in terms of shorter chains, and the claim follows by showing that the expectation matches the right-hand side of~\eqref{eq-Erecursion} up to an $\smash{\cO(N^\eps/(\sqrt{N\eta_*}\eta_*^{k-a/2}))}$ error.

\medskip
A brief calculation (see the proof of the local law~\cite[Thm.~3.4]{CES-thermalization} or~\cite[Lem.~4.1]{CES-optimalLL}) yields 
\begin{align}
\langle T_{[1,k]}\rangle&=m_1\Big(-\langle\underline{WT_{[1,k]}}\rangle+\langle T_{[2,k]}A_1\rangle+\sum_{j=2}^{k-1}\langle T_{[1,j\rangle}G_j\rangle\langle T_{[j,k]}\rangle+\langle G_1-m_1\rangle\langle T_{[1,k]}\rangle\NN\\
&\quad+\langle T_{[1,k\rangle}G_k\rangle\langle(G_k-m_k)A_k\rangle\Big).\label{eq-m1identity}
\end{align}
Next, we rewrite the equation to the form
\begin{align}
&\Big(1+\cO_{\prec}\Big(\frac{1}{N\eta_*}\Big)\Big)(\langle T_{[1,k]}\rangle-\m[T_1,\dots,T_k])\NN\\
&=m_1\Big(-\langle \underline{WT_{[1,k]}}\rangle+(\langle T_{[2,k\rangle}G_kA_kA_1\rangle-\m[T_2,\dots,T_{k-1},G_kA_kA_1])\NN\\
&\quad+\sum_{j=1}^{k-1}(\langle T_{[1,j\rangle}G_j\rangle-\m[T_1,\dots,T_{j-1},G_j])\m[T_j,\dots,T_k]\label{eq-mesodifference1}\\
&\quad+\sum_{j=2}^k\m[T_1,\dots,T_{j-1},G_j](\langle T_{[j,k]}\rangle-\m[T_j,\dots,T_k])\NN\\
&\quad+\sum_{j=2}^{k}(\langle T_{[1,j\rangle}G_j\rangle-\m[T_1,\dots,T_{j-1},G_j])(\langle T_{[j,k]}\rangle-\m[T_j,\dots,T_k])\Big),\NN
\end{align}
where we applied~\eqref{eq-multiGaveraged} for $\langle G_1-m_1\rangle$ on the right-hand side. Recall that $\langle A_k\rangle=0$ and $a=k$ in the case considered. Moving the factor $(1+\cO_{\prec}((N\eta_*)^{-1})$ to the right-hand side, multiplying~\eqref{eq-mesodifference1} with $N$ and taking the expectation yields
\begin{align*}
N\E(\langle T_{[1,k]}\rangle-\m[T_1,\dots,T_k])&=\Big(m_1+\cO_{\prec}\Big(\frac{1}{N\eta_*}\Big)\Big)\Big(-N\E\langle \underline{WT_{[1,k]}}\rangle+\kappa_4\cE[T_1,\dots,T_k]\\
&\quad-\kappa_4\sum_{1\leq r\leq s\leq t\leq k}\langle M_{[r]}\odot M_{[s,t]}\rangle\langle M_{[r,s]}\odot(M_{[t,k]}A_k)\rangle\Big)\\
&\quad+\cO\Big(\frac{N^\eps}{\sqrt{N\eta_*}\ \eta_*^{k/2}}\Big)
\end{align*}
by the induction hypothesis,~\eqref{eq-Erecursion} and the expansion $\frac{1}{1+x}=1+\cO(x)$. We further applied ~\eqref{eq-multiGaveraged} for the last line of~\eqref{eq-mesodifference1} to obtain
\begin{displaymath}
(\langle T_{[1,j\rangle}G_j\rangle-\m[T_1,\dots,T_{j-1},G_j])(\langle T_{[j,k]}\rangle-\m[T_j,\dots,T_k])=\cO_\prec\Big(\frac{1}{N^2\eta_*^{k/2+1}}\Big).
\end{displaymath}
It follows that
\begin{displaymath}
N\E(\langle T_{[1,j\rangle}G_j\rangle-\m[T_1,\dots,T_{j-1},G_j])(\langle T_{[j,k]}\rangle-\m[T_j,\dots,T_k])=\cO\Big(\frac{N^\eps}{(N\eta_*)\ \eta_*^{k/2}}\Big),
\end{displaymath}
i.e., the term is indeed part of the error. Hence,~\eqref{eq-mEexchangemeso} is established if
\begin{equation}\label{eq-targeterror}
N\E\langle \underline{WT_{[1,k]}}\rangle=-\kappa_4\sum_{1\leq r\leq s\leq t\leq k}\langle M_{[r]}\odot M_{[s,t]}\rangle\langle M_{[r,s]}\odot(M_{[t,k]}A_k)\rangle+\cO\Big(\frac{N^\eps}{\sqrt{N\eta_*}\ \eta_*^{k/2}}\Big).
\end{equation}
By cumulant expansion, the underlined term on the left-hand side of~\eqref{eq-targeterror} is given by
\begin{equation}\label{eq-cumuexpmeso}
N\E\langle \underline{WT_{[1,k]}}\rangle=\sum_{n\geq2}\sum_{x,y\in[N]}\sum_{\nu\in\{xy,yx\}^n}\frac{\kappa(xy,\nu)}{n!}\E \partial_{\nu}(T_{[1,k]})_{yx},
\end{equation}
where $\partial_{xy}$ denotes the directional derivative in the direction of the $xy$ entry of $W$ and $\kappa(xy,\nu)$ denotes the joint cumulant of $W_{xy},W_{\nu_1},\dots,W_{\nu_n}$ for any $n$-tuple of double indices $\nu=(\nu_1,\dots,\nu_n)$. Note that the $n=1$ term of the expansion~\eqref{eq-cumuexpmeso} is canceled out by the renormalization~\eqref{eq-underline}.

\medskip
Recall that $\kappa(xy,\nu)\sim N^{-(|\nu|+1)/2}$ by the scaling of $W$. It is hence sufficient to estimate the terms for $n\geq4$ trivially using the bounds from Lemma~\ref{lem-mbounds}, as the factor $N^2$ obtained from the double summation is canceled by the bound for the cumulant. Note that since
\begin{displaymath}
\partial_{xy}(T_{[1,k]})_{vw}=\sum_{r=1}^{k}(T_{[1,r\rangle}G_r)_{vx}(T_{[r,k]})_{yw},
\end{displaymath}
every derivative yields an additional resolvent factor, which increases the size of the bound for the corresponding resolvent chain by $\eta_*^{-1}$. However, each derivative also "breaks" the chain it acts on and increases the total number of resolvent chain entries in the term by one. This compensates for the additional factor $\eta_*^{-1}$ (cf. Lemma~\ref{lem-mbounds}) such that the power of $\eta_*^{-1}$ contained in the bound for $\partial_{\nu}(T_{[1,k]})_{yx}$ does not depend on the order of the derivative. We conclude that
\begin{displaymath}
\sum_{n\geq4}\sum_{x,y}\sum_{\nu\in\{xy,yx\}^n}\frac{\kappa(xy,\nu)}{n!}\E \partial_{\nu}(T_{[1,k]})_{yx}=\cO\Big(\frac{N^\eps}{\sqrt{N\eta_*}\ \eta_*^{k/2}}\Big)
\end{displaymath}
and identify the term as part of the error in~\eqref{eq-mEexchangemeso}. It remains to consider the $n=2$ and $n=3$ contribution to~\eqref{eq-cumuexpmeso}. Recall that we write $T_{[i,j]}=T_i\dots T_j$ for $i\leq j$ and $T_{[i,i\rangle}=T_{\emptyset}=0$.

\medskip
\underline{\smash{Estimate for the $n=2$ term of~\eqref{eq-cumuexpmeso}}:} Evaluating the derivative yields, e.g.,
\begin{align*}
(\partial_{xy})^2(T_{[1,k]})_{yx}&=-\partial_{xy}\sum_{r=1}^k(T_{[1,r\rangle}G_r)_{yx}(T_{[r,k]})_{yx}\\
&=\sum_{r=1}^k\Big(\sum_{s=1}^r(T_{[1,s\rangle}G_s)_{yx}(T_{[s,r\rangle}G_r)_{yx}(T_{[r,k\rangle})_{yx}+\sum_{s=r}^k(T_{[1,r\rangle}G_r)_{yx}(T_{[r,s\rangle}G_s)_{yx}(T_{[s,k]})_{yx}\Big),
\end{align*}
which only involves $xy$ or $yx$ entries of a resolvent chain. Note that the contribution for the case $x=y$ consists of just one sum and is, therefore, of lower order by power counting. We thus refer to any $xy$ or $yx$ entries as off-diagonal below. Further, we obtain
\begin{align*}
\partial_{xy}\partial_{yx}(T_{[1,k]})_{yx}&=\sum_{r=1}^k\Big(\sum_{s=1}^r(T_{[1,s\rangle}G_s)_{yx}(T_{[s,r\rangle}G_r)_{yy}(T_{[r,k]})_{xx}\\
&\quad\quad+\sum_{s=r}^k(T_{[1,r\rangle}G_r)_{yx}(T_{[r,s\rangle}G_s)_{yy}(T_{[s,k]})_{xx}\Big),
\end{align*}
which involves two diagonal and one off-diagonal entries. The other terms arising for $n=2$ are of a similar structure, i.e., they involve either zero or two diagonal entries.

\medskip
As every term involves at least one off-diagonal entry of a resolvent chain, we use a procedure called~\emph{isotropic resummation} to estimate the $x$ and $y$ summations. To illustrate the strategy, consider the sum
\begin{equation}\label{eq-exestimate}
N^{-3/2}\sum_{x,y\in[N]}(T_{[1,r\rangle}G_r)_{yx}(T_{[r,s\rangle}G_s)_{yy}(T_{[s,k]})_{xx}
\end{equation}
with fixed $1\leq r\leq s\leq k$. The factor $N^{-3/2}$ in front of the term accounts for the size of the cumulant $\kappa(xy,xy,yx)$. First, insert the deterministic approximation of the diagonal terms to decompose the resolvent chain entries in~\eqref{eq-exestimate} into a deterministic term satisfying the bounds in Lemma~\ref{lem-mbounds} and a fluctuation that is controlled by the local law~\eqref{eq-multiGisotropic}. Recalling that
\begin{equation}\label{eq-recalled1}
\|M_{[i,j]}A_j\|\lesssim\|M_{[i,j]}\|\lesssim\frac{1}{\eta_*^{(j-i)/2}}
\end{equation}
for $i\leq j$ by Lemma~\ref{lem-mbounds}, as $M_{[i,j]}$ involves $j-i$ traceless matrices and $\|A_j\|\lesssim1$ by assumption, and that
\begin{equation}\label{eq-recalled2}
\max_{x,y\in[N]}|(T_{[i,j\rangle}G_j-M_{[i,j]})A_j)_{xy}|=\cO_{\prec}\Big(\frac{1}{\sqrt{N\eta_*}\ \eta_*^{(j-i)/2}}\Big)
\end{equation}
by the isotropic local law~\eqref{eq-multiGisotropic}, we obtain the bound
\begin{displaymath}
|(T_{[1,r\rangle}G_r)_{yx}|\leq \|M_{[r]}\|+\max_{x,y\in[N]}|(T_{[1,r\rangle}G_r-M_{[r]})_{yx}|=\cO_\prec \Big(\frac{1}{\eta_*^{(r-1)/2}}+\frac{1}{\sqrt{N\eta_*}\ \eta_*^{(r-1)/2}}\Big).
\end{displaymath}
Recall that we abbreviated $M_{[r]}=M_{[1,r]}$. Putting everything together, it follows that
\begin{displaymath}
N^{-3/2}\sum_{x,y\in[N]}(T_{[1,r\rangle}G_r)_{yx}(T_{[r,s\rangle}G_s-M_{[r,s]})_{yy}(T_{[s,k]}-M_{[s,k]}A_k)_{xx}=\cO_{\prec}\Big(\frac{1}{\sqrt{N\eta_*}\ \eta_*^{k/2}}\Big)
\end{displaymath}
and we can include this term in the error in~\eqref{eq-mEexchangemeso}. Recall that we consider $N(\E\langle T_{[1,k]}-\m[T_1,\dots,T_k])$, i.e., the error terms obtained differ from~\eqref{eq-mEexchangemeso} by a factor of $N$.

\medskip
Next, let $\dvec{A}=(A_{jj})_{j=1}^N$ denote the (column) vector consisting of the diagonal entries of a matrix $A\in\C^{N\times N}$. In this notation, we have
\begin{equation*}
N^{-3/2}\sum_{x,y\in[N]}(T_{[1,r\rangle}G_r)_{yx}(M_{[r,s]})_{yy}(M_{[s,k]}A_k)_{xx}=N^{-3/2}\langle\dvec{M_{[r,s]}},T_{[1,r\rangle}G_r\dvec{M_{[s,k]}A_k}\rangle
\end{equation*}
with deterministic vectors $\dvec{M_{[r,s]}}$ and $\dvec{M_{[s,k]}A_k}$ satisfying
\begin{align*}
\|\dvec{M_{[r,s]}}\|&\leq\sqrt{N}\|M_{[r,s]}\|=\cO\Big(\frac{\sqrt{N}}{\eta_*^{(s-r)/2}}\Big),\\
\|\dvec{M_{[s,k]}A_k}\|&\leq\sqrt{N}\|M_{[s,k]}A_k\|=\cO\Big(\frac{\sqrt{N}}{\eta_*^{(k-s)/2}}\Big),
\end{align*} 
by~\eqref{eq-recalled1}. Recalling that $T_{[1,r\rangle}G_r$ contains $r-1$ traceless matrices by assumptions, we obtain
\begin{displaymath}
N^{-3/2}\sum_{x,y\in[N]}(T_{[1,r\rangle}G_r)_{yx}(M_{[r,s]})_{yy}(M_{[s,k]}A_k)_{xx}=\cO_\prec\Big(\frac{1}{\sqrt{N\eta_*}\ \eta_*^{k/2-1}}\Big)
\end{displaymath}
from~\eqref{eq-recalled1} and~\eqref{eq-multiGisotropic}. Hence, the term can be included in the error in~\eqref{eq-mEexchangemeso}.

\medskip
It remains to estimate the two terms in~\eqref{eq-exestimate} that involve one deterministic and one fluctuation term each. By the Cauchy-Schwarz inequality, we obtain
\begin{align*}
&N^{-3/2}\sum_{x,y\in[N]}(T_{[1,r\rangle}G_r)_{yx}(T_{[r,s\rangle}G_s-M_{[r,s]})_{yy}(M_{[s,k]}A_k)_{xx}\\
&\leq N^{-3/2}\|T_{[1,r\rangle}G_r\|\cdot\|\dvec{T_{[r,s\rangle}G_s-M_{[r,s]}}\|\cdot\|\dvec{M_{[s,k]}A_k}\|
\end{align*}
with a similar bound holding for the term involving $M_{[r,s]}$ and $(T_{[s,k]}-M_{[s,k]}A_k)$. Note that
\begin{align*}
\|\dvec{T_{[r,s\rangle}G_s-M_{[r,s]}}\|&\leq \sqrt{N}\max_{x\in[N]}|(T_{[r,s\rangle}G_s-M_{[r,s]})_{xx}|=\cO_\prec\Big(\frac{1}{\eta_*^{(s-r+1)/2}}\Big),\\
\|T_{[1,r\rangle}G_r\|&\leq \|M_{[r]}\|+N\max_{x,y\in[N]}|(T_{[1,r\rangle}G_r-M_{[r]})_{xy}|=\cO_\prec \Big(\frac{1}{\eta_*^{(r-1)/2}}+\frac{\sqrt{N}}{\eta_*^{r/2}}\Big)
\end{align*}
by~\eqref{eq-recalled1},~\eqref{eq-recalled2}, and the fact that $\|B\|\leq N\max_{x,y}|B_{xy}|$ for any matrix $B\in\C^{N\times N}$. Overall, we obtain
\begin{displaymath}
N^{-3/2}\sum_{x,y\in[N]}(T_{[1,s\rangle}G_s)_{yx}(T_{[s,r\rangle}G_r)_{yy}(T_{[r,k]})_{xx}=\cO_{\prec}\Big(\frac{1}{\sqrt{N\eta_*}\ \eta_*^{k/2}}\Big),
\end{displaymath}
i.e., the term is part of the error in~\eqref{eq-mEexchangemeso}. The other terms arising from $\partial_{\nu}(T_1\dots T_k)_{yx}$ with $\nu\in\{xy,yx\}^2$ are treated similarly. Hence, the entire $n=2$ contribution of~\eqref{eq-cumuexpmeso} can be included in the error term in~\eqref{eq-mEexchangemeso}.

\medskip
\underline{\smash{Computation of the $n=3$ term of~\eqref{eq-cumuexpmeso}}:} By a similar computation as in the $n=2$ case, the terms arising from $\partial_{\nu}(T_{[1,k]})_{yx}$ for $|\nu|=3$ involve either zero, two, or four diagonal entries of a resolvent chain. Whenever a term contains off-diagonal entries, we can include it in the error in~\eqref{eq-mEexchangemeso} by decomposing each resolvent chain entry into a deterministic and a fluctuation part. Note that $M_{(\cdot)}$ is not necessarily diagonal, i.e., applying~\eqref{eq-multiGisotropic} alone is not sufficient. However, as every fluctuation term contributes a factor $\smash{N^{-1/2}}$ and the presence of off-diagonal terms allows us to apply the Cauchy-Schwarz inequality to estimate the remaining (deterministic) double sum, we gain a factor $N^{-1/2}$ over the trivial bound as needed.

\medskip
It remains to evaluate the contributions that consist of four diagonal entries, which can only occur for $\nu\in\{(xy,yx,yx),(yx,xy,yx),(yx,yx,xy)\}$. Here,
\begin{align}
\partial_{\nu}(T_{[1,k]})_{yx}&=-\sum_{1\leq t\leq s\leq r\leq k}(T_{[1,t\rangle}G_t)_{yy}(T_{[t,s\rangle}G_s)_{xx}(T_{[s,r\rangle}G_r)_{yy}(T_{[r,k]})_{xx}\NN\\
&\quad-\sum_{1\leq r\leq s\leq t\leq k}(T_{[1,r\rangle}G_r)_{yy}(T_{[r,s\rangle}G_s)_{xx}(T_{[s,t\rangle}G_t)_{yy}(T_{[t,k]})_{xx}-\dots\label{eq-kap4term}
\end{align}
where the terms that are not written out in the last line involve two off-diagonal entries and, therefore, can be included in the error term. Since $\kappa(xy,\nu)=\kappa_4/N^2$ for the cases considered, we obtain
\begin{align*}
&\sum_{x,y}\sum_{\nu\in\{xy,yx\}^3}\frac{\kappa(xy,\nu)}{6}\E \partial_{\nu}(T_{[1,k]})_{yx}\\
&=\frac{\kappa_4}{2N^2}\sum_{x,y}\Big(-\sum_{1\leq t\leq s\leq r\leq k}(M_{[t]})_{yy}(M_{[t,s]})_{xx}(M_{[s,r]})_{yy}(M_{[r,k]}A_k)_{xx}\\
&\quad\quad\quad\quad\quad\quad-\sum_{1\leq r\leq s\leq t\leq k}(M_{[r]})_{yy}(M_{[r,s]})_{xx}(M_{[s,t]})_{yy}(M_{[t,k]}A_k)_{xx}\Big)+\cO\Big(\frac{N^\eps}{N\, \sqrt{N\eta_*}\ \eta_*^{k-a/2}}\Big)\\
&=-\kappa_4\sum_{1\leq r\leq s\leq t\leq k}\langle M_{[r]}\odot M_{[s,t]}\rangle\langle M_{[r,s]}\odot(M_{[t,k]}A_k)\rangle+\cO\Big(\frac{N^\eps}{N\, \sqrt{N\eta_*}\ \eta_*^{k-a/2}}\Big)
\end{align*}
where the last equality follows from the definition of the Hadamard product.

\medskip
Adding the contributions to~\eqref{eq-cumuexpmeso} together, the cumulant expansion evaluates to
\begin{align*}
-N\E\langle \underline{WT_{[1,k]}}\rangle&=\kappa_4\sum_{1\leq r\leq s\leq t\leq k}\Big(\frac{1}{N^2}\sum_{x,y}(M_{[r]})_{yy}(M_{[r,s]})_{xx}(M_{[s,t]})_{yy}(M_{[t,k]}A_k)_{xx}\Big)\\
&\quad+\cO\Big(\frac{N^\eps}{\sqrt{N\eta_*}\ \eta_*^{k-a/2}}\Big).
\end{align*}
We conclude that $N(\E\langle T_{[1,k]}\rangle-\m[T_1,\dots,T_k])$ coincides with the right-hand side of~\eqref{eq-Erecursion} up to an error term and an additional factor $\kappa_4$. Applying the recursion thus yields
\begin{displaymath}
N(\E\langle T_{[1,k]}\rangle-\m[T_1,\dots,T_k])=\kappa_4\cE[T_1,\dots,T_k]+\cO\Big(\frac{N^\eps}{\sqrt{N\eta_*}\ \eta_*^{k-a/2}}\Big)
\end{displaymath}
as claimed.\qed

\subsection{Proof of Lemma~\ref{lem-covapproxmeso} (Deterministic Approximation of $\E X_\alpha X_\beta$)}\label{sect-covapproxmeso}

We start by noting some general estimates for $X^{(k,a)}_\alpha$ and its derivatives that are needed for the proof of Lemma~\ref{lem-covapproxmeso}.

\begin{lemma}[A priori estimates]\label{lem-apriorimeso}
For $X^{(k,a)}_\alpha=\langle T_{[1,k]}\rangle-\E\langle T_{[1,k]}\rangle$ and its derivatives $\partial_{\nu}X^{(k,a)}_\alpha$ for any multi-index $\nu$, we have the estimate
\begin{equation}\label{eq-apriori1meso}
|\partial_{\nu}X^{(k,a)}_\alpha|=\cO_{\prec}\Big(\frac{1}{N\eta_*^{k-a/2}}\Big).
\end{equation}
Moreover, we have the more precise expansions for the first and second derivatives
\begin{align}
\partial_{xy}X^{(k,a)}_\alpha&=-\frac{1}{N}\Big[\sum_{r=1}^k(M_{(r,\dots,k,1,\dots,r)})_{yx}+\cO_{\prec}\Big(\frac{1}{\sqrt{N\eta_*}\ \eta_*^{k-a/2}}\Big)\Big],\label{eq-apriori2meso}\\
\partial_{vw}\partial_{xy}X^{(k,a)}_\alpha&=\frac{1}{N}\Big[\sum_{r=1}^k\Big(\sum_{s=1}^r(M_{(r,\dots,k,1,\dots,s)})_{wx}(M_{[s,r]})_{yv}\NN\\
&\quad\quad\quad\quad+\sum_{s=r}^k(M_{[r,s]})_{wx}(M_{(s,\dots,k,1,\dots,r)})_{yv}\Big)+\cO_{\prec}\Big(\frac{1}{\sqrt{N\eta_*}\ \eta_*^{k-a/2}}\Big)\Big].\label{eq-apriori3meso}
\end{align}
\end{lemma}

\begin{proof}
The bounds in~\eqref{eq-apriori2meso} and~\eqref{eq-apriori3meso} follow directly from~\eqref{eq-multiGisotropic}, Lemma~\ref{lem-mbounds}, and
\begin{align*}
\partial_{xy}\langle T_{[1,k]}\rangle&=-\frac{1}{N}\sum_{r=1}^k(T_{[r,k]}T_{[1,r\rangle}G_r)_{yx},\\
\partial_{vw}\partial_{xy}\langle T_{[1,k]}\rangle&=\frac{1}{N}\sum_{r=1}^k\Big(\sum_{s=1}^r(T_{[r,k]}T_{[1,s\rangle}G_s)_{wx}(T_{[s,r\rangle}G_r)_{yv}+\sum_{s=r}^k(T_{[r,s\rangle}G_s)_{wx}(T_{[s,k]}T_{[1,r\rangle}G_r)_{yv}\Big).
\end{align*}
The claim~\eqref{eq-apriori1meso} follows inductively by~\eqref{eq-multiGisotropic} and Lemma~\ref{lem-mbounds}. Note that~\eqref{eq-apriori1meso} is also true for $\nu=\emptyset$ since
\begin{displaymath}
X^{(k)}_\alpha=(\langle T_{[1,k]}\rangle-\m[T_1,\dots,T_k])-(\E\langle T_{[1,k]}\rangle-\m[T_1,\dots,T_k])=\cO_{\prec}\Big(\frac{1}{N\eta_*^{k-a/2}}\Big)
\end{displaymath}
as a consequence of~\eqref{eq-multiGaveraged} and Lemma~\ref{lem-mEexchangemeso}.
\end{proof}

\begin{proof}[Proof of Lemma~\ref{lem-covapproxmeso}]
We use proof by induction on the length of the multi-indices $\alpha$ and $\beta$. First, note that the left-hand side of~\eqref{eq-2ndorderLL1meso} vanishes if either $\langle T_{[1,k]}\rangle$ or $\langle T_{[k+1,k+\ell]}\rangle$ is zero, i.e., if one of the terms is indexed by the empty set. Comparing with~\eqref{eq-Minitial}, the base case is established. Due to the symmetry of the expression, it is further sufficient to carry out the induction step for one of the arguments only. Assume that~\eqref{eq-2ndorderLL1meso} holds for multi indices $\alpha$ of length $1,\dots,k-1$ and multi-indices $\beta$ of a fixed length $\ell\geq1$. We need to show that the deterministic approximation of $\smash{N^2\E X^{(k)}_\alpha X^{(\ell)}_\beta}$ is given by $\m[T_1,\dots,T_k|T_{k+1}\dots T_{k+\ell}]$ and estimate the error term. W.l.o.g. assume that each $A_j$ is either traceless or equal to the identity matrix.

\medskip
As a first step, consider resolvent chains $T_{[1,k]}$ that contain at least one deterministic matrix $A_j=\Id$, i.e., that are of the form $T_{[1,k]}=T_{[1,j\rangle}G_jG_{j+1}T_{\langle j+1,k]}$ with the indices being interpreted mod $k$ due to the cyclicity of the trace and the first entry of $\m[\cdot|\cdot]$. In this case, rewriting the product $G_jG_{j+1}$ in terms of a single resolvent allows us to obtain the claim directly from the induction hypothesis. We distinguish two cases depending on the imaginary parts of $z_j$ and $z_{j+1}$.

\medskip
\underline{\smash{Case 1: $\Im z_j$ and $\Im z_{j+1}$ have the same sign:}} Let $s:=\mathrm{sign}(\Im z_j)=\mathrm{sign}(\Im z_{j+1})$ and recall from~\eqref{eq-residueGs} that the product $G_jG_{j+1}$ can be rewritten as a contour integral using the residue theorem. We obtain a similar contour integral formula for $\m[\cdot|\cdot]$ from Corollary~\ref{cor-mproperties}(iii), namely
\begin{align}
&\m[T_1,\dots,T_{j-1},G_j,T_{j+1},\dots,T_k|\beta]\label{eq-residuem}\\
&=\frac{1}{2\pi\ri}\int_{\R}\frac{\m[\dots,T_{j-1},G(x+\ri\eta)A_{j+1},\dots|\beta]-\m[\dots,T_{j-1},G(x-\ri\eta)A_{j+1},\dots|\beta]}{(x+\ri s\eta-z_j)(x+\ri s\eta-z_{j+1})}\dx x.\NN
\end{align}
Recall that the properties of $\m[\cdot|\cdot]$ stated in Corollary~\ref{cor-mproperties} are a consequence of the recursion~\eqref{eq-Mrecursion} and thus independent of Lemma~\ref{lem-covapproxmeso}. Rewriting the left-hand side of~\eqref{eq-2ndorderLL1meso} using~\eqref{eq-residueM} yields a contour integral with an integrand of the form $\smash{N^2\E X^{(k)}_{\alpha'} X^{(\ell)}_\beta}$ for a multi-index $\alpha'$ of length $k-1$ and without the identity matrix $A_j=\Id$, i.e., the number of traceless matrices $a$ is unchanged. By the induction hypothesis, we thus obtain
\begin{align*}
&N^2\E X^{(k)}_\alpha X^{(\ell)}_\beta=\frac{1}{2\pi\ri}\int_{\R}\frac{\m[\dots,T_{j-1},G(x+\ri\eta)A_{j+1},\dots|\beta]-\m[\dots,T_{j-1},G(x-\ri\eta)A_{j+1},\dots|\beta]}{(x+\ri s\eta-z_j)(x+\ri s\eta-z_{j+1})}\dx x\\
&\quad+\cO\Big(\frac{N^\eps}{\sqrt{N\eta_*}\ \eta_*^{(k+\ell-1)-(a+b)/2}}\Big)
\end{align*}
where the integral evaluates to $\m[\alpha|\beta]$ by~\eqref{eq-residuem}. Hence,~\eqref{eq-2ndorderLL1meso} holds in the case considered.

\medskip
\underline{\smash{Case 2: $\Im z_j$ and $\Im z_{j+1}$ have opposite signs:}} Applying~\eqref{eq-resolventid} and the divided difference structure of $\m[\cdot|\cdot]$, it follows that
\begin{align*}
&N^2\E X^{(k,a)}_\alpha X^{(\ell,b)}_\beta\\
&=N^2\Big(\frac{(\langle T_{[1,j\rangle}G_jA_{j+1}T_{\langle j+1,k]}\rangle-\E\langle T_{[1,j\rangle}G_jA_{j+1}T_{\langle j+1,k]}\rangle)-(\langle T_{[1,j\rangle}T_{[j+1,k]}\rangle-\E\langle T_{[1,j\rangle}T_{[j+1,k]}\rangle)}{z_j-z_{j+1}}\Big) X^{(\ell,b)}_\beta\\
&=\frac{\m[T_1,\dots,T_{j-1},T_{j+1},\dots,T_k|\beta]-\m[T_1,\dots,T_{j-1},T_jA_{j+1},T_{j+2},\dots,T_k|\beta]}{z_j-z_{j+1}}\\
&\quad+\cO\Big(\frac{N^\eps}{\sqrt{N\eta_*}\eta_*^{k+\ell-(a+b)/2-1}|z_j-z_{j+1}|}\Big)\\
&=\m[T_1,\dots,T_k|\beta]+\cO\Big(\frac{N^\eps}{\sqrt{N\eta_*}\ \eta_*^{k+\ell-(a+b)/2}}\Big),
\end{align*}
since $|z_j-z_{j+1}|\geq2\eta_*$ in the case considered. This concludes the proof of~\eqref{eq-2ndorderLL1meso} for resolvent chains that contain at least one deterministic matrix $A_j=\Id$.

\medskip
It remains to consider $T_{[1,k]}$ for which all matrices $A_1,\dots,A_k$ are traceless, i.e., for which $a=k$. The argument is similar to the proof of Lemma~\ref{lem-mEexchangemeso}, i.e., we rewrite $\smash{N^2\E X^{(k,a)}_\alpha X^{(\ell,b)}_\beta}$ in terms of covariances of smaller chains and show that it satisfies the recursion~\eqref{eq-Mrecursion} up to an $\smash{N^\eps/(\sqrt{N\eta_*}\eta_*^{k-a/2}\eta_*^{\ell-b/2})}$ error.

\medskip
Combining~\eqref{eq-mesodifference1} and Lemma~\ref{lem-mEexchangemeso} yields the starting point
\begin{align}
X^{(k,a)}_\alpha&=m_1\Big(-(\langle \underline{WT_{[1,k]}}\rangle-\E\langle \underline{WT_{[1,k]}}\rangle)+(\langle T_{[2,k\rangle}G_kA_kA_1\rangle-\E\langle T_{[2,k\rangle}G_kA_kA_1\rangle)\NN\\
&\quad+\sum_{j=1}^{k-1}(\langle T_{[1,j\rangle}G_j\rangle-\E\langle T_{[1,j\rangle}G_j\rangle)\m[T_j,\dots,T_k]\label{eq-Edifference1meso}\\
&\quad+\sum_{j=2}^k\m[T_1,\dots,T_{j-1},G_j](\langle T_{[j,k]}\rangle-\E\langle T_{[j,k]}\rangle)\Big)+\cO_\prec\Big(\frac{1}{N\, \sqrt{N\eta_*}\ \eta_*^{k/2}}\Big).\NN
\end{align}
Next, multiply~\eqref{eq-Edifference1meso} by $\smash{N^2X^{(\ell)}_{\beta}}$ and compute the expectation. Applying the induction hypothesis for any terms for which the first factor involves a resolvent chain of length at most $k-1$ yields
\begin{align*}
N^2\E\Big(X^{(k)}_{\alpha}X^{(\ell)}_{\beta}\Big)&=\m[\alpha|\beta]-\sum_{j=1}^\ell \Big(\m[T_1,\dots,T_k,T_{k+j},\dots,T_{k+j-1},G_{k+j}]\\
&\quad-\kappa_4\sum_{r=1}^k\sum_{s=k+1}^{k+\ell}\Big(\sum_{t=k+1}^s\langle M_{[r]}\odot M_{(s,\dots,k+\ell,k+1,\dots,t)}\rangle\langle(M_{[r,k]}A_k)\odot M_{[t,s]}\rangle\NN\\
&\quad\quad+\sum_{t=s}^{k+\ell}\langle M_{[r]}\odot M_{[s,t]}\rangle\langle(M_{[r,k]}A_k)\odot M_{(t,\dots,k+\ell,k+1,\dots,s)}\rangle\Big)\\
&\quad +N^2\E\big((\langle \underline{WT_{[1,k]}}\rangle-\E\langle \underline{WT_{[1,k]}}\rangle)X^{(\ell)}_{\beta}\big)+\cO\Big(\frac{N^\eps}{\sqrt{N\eta_*}\ \eta_*^{(k+\ell)/2}}\Big),
\end{align*}
where we used the recursion~\eqref{eq-Mrecursion} to introduce $\m[\alpha|\beta]$. It remains to compute
\begin{align*}
N^2\E\big((\langle \underline{WT_{[1,k]}}\rangle-\E\langle \underline{WT_{[1,k]}}\rangle)X^{(\ell)}_{\beta}\big)=N^2\E\big(\langle \underline{WT_{[1,k]}}\rangle X^{(\ell)}_{\beta}\big).
\end{align*}
By cumulant expansion, we obtain
\begin{align}
N^2\E\big(\langle \underline{WT_{[1,k]}}\rangle X^{(\ell)}_{\beta}\big)&=-\sum_{j=1}^\ell\E\langle T_{[1,k]}T_{[k+j,k+\ell]}T_{[1,k+j\rangle}G_{k+j}\rangle\NN\\
&\quad+N\sum_{n=2}^3\sum_{x,y\in[N]}\sum_{\nu\in\{xy,yx\}^n}\frac{\kappa(xy,\nu)}{n!}\E \partial_{\nu}\Big((T_{[1,k]})_{yx}X^{(\ell)}_{\beta}\Big)\label{eq-cumuexpmeso2}\\
&\quad+\cO\Big(\frac{N^\eps}{\sqrt{N\eta_*}\ \eta_*^{(k+\ell)/2}}\Big)\NN,
\end{align}
where the $n=1$ term was evaluated using~\eqref{eq-underline} and the terms for $n\geq4$ were again estimated trivially using that $\kappa(xy,\nu)\sim N^{-(|\nu|+1)/2}$ due to the scaling of $W$ as well as the bounds from Lemmas~\ref{lem-mbounds} and~\ref{lem-apriorimeso}. By the isotropic local law~\eqref{eq-multiGisotropic} we have
\begin{displaymath}
\E\langle T_{[1,k]}T_{[k+j,k+\ell]}T_{[1,k+j\rangle}G_{k+j}\rangle=\m[T_1,\dots,T_k,T_{k+j},\dots,T_{k+j-1},G_{k+j}]+\cO\Big(\frac{N^\eps}{(N\eta_*)\ \eta_*^{(k+\ell)/2}}\Big),
\end{displaymath}
which yields an error term that can be included in the error in~\eqref{eq-2ndorderLL1meso}. Hence, it only remains to consider the contributions for $n=2$ and $n=3$.

\medskip
\underline{\smash{Estimate for the $n=2$ term of~\eqref{eq-cumuexpmeso2}}:} Evaluating the derivative $\smash{\partial_{\nu}((T_{[1,k]})_{yx}X^{(\ell)}_{\beta})}$ always yields at least one $xy$ or $yx$ entry of a resolvent chain, which we can use to apply isotropic resummation (cf. estimate for~\eqref{eq-exestimate} above). Note that $\smash{X^{(\ell)}_\beta}$ always contributes a factor $N^{-1}$ due to Lemma~\ref{lem-apriorimeso}, either from evaluating derivatives using~\eqref{eq-apriori2meso} and~\eqref{eq-apriori3meso}, or from the trivial estimate~\eqref{eq-apriori1meso}. The additional factor $N$ in front of the $x,y$ summation in~\eqref{eq-cumuexpmeso2} is thus balanced out. Overall, we obtain
\begin{displaymath}
N\sum_{x,y}\sum_{\nu\in\{xy,yx\}^2}\frac{\kappa(xy,\nu)}{2}\E \partial_{\nu}\Big((T_{[1,k]})_{yx}X^{(\ell)}_{\beta}\Big)=\cO\Big(\frac{N^\eps}{\sqrt{N\eta_*}\ \eta_*^{k-a/2}\eta_*^{\ell-b/2}}\Big).
\end{displaymath}
In particular, the $n=2$ term can be included in the error in the last line of~\eqref{eq-cumuexpmeso2}.

\medskip
\underline{\smash{Computation of the $n=3$ term of~\eqref{eq-cumuexpmeso2}}:} By applying the Leibniz rule, we distribute the derivatives as
\begin{displaymath}
\E \partial_{\nu}\Big((T_{[1,k]})_{yx}X^{(\ell)}_{\beta}\Big)=\sum_{\nu_1\cup\nu_2=\nu}\E\Big(\partial_{\nu_1}(T_{[1,k]})_{yx}\partial_{\nu_2}X^{(\ell)}_{\beta}\Big)
\end{displaymath}
and distinguish three cases for $\nu_1$:

\medskip
Case 1 ($|\nu_1|=3$): The derivative $\partial_{\nu_1}(T_{[1,k]})_{yx}$ arising in this case was already computed in the proof of Lemma~\ref{lem-mEexchangemeso}. Using~\eqref{eq-kap4term} and~\eqref{eq-multiGisotropic}, it follows that
\begin{align}
&N\sum_{x,y}\sum_{\nu\in\{xy,yx\}^3}\frac{\kappa(xy,\nu)}{6}\E \Big((\partial_{\nu}(T_{[1,k]})_{yx})X^{(\ell)}_{\beta}\Big)\NN\\
&=-\frac{\kappa_4}{N}\sum_{x,y}\sum_{1\leq r\leq s\leq t\leq k}(M_{[r]})_{yy}(M_{[r,s]})_{xx}(M_{[s,t]})_{yy}(M_{[t,k]}A_k)_{xx}\E X^{(\ell)}_{\beta}+\cO\Big(\frac{N^\eps}{\sqrt{N\eta_*}\ \eta_*^{k-a/2}\eta_*^{\ell-b/2}}\Big)\NN\\
&=\cO\Big(\frac{N^\eps}{\sqrt{N\eta_*}\ \eta_*^{k-a/2}\eta_*^{\ell-b/2}}\Big),
\end{align}
since $X^{(\ell)}_{\beta}$ is centered. We hence include the term in the error in~\eqref{eq-cumuexpmeso2}.

\medskip
Case 2 ($|\nu_1|=1$): Evaluating the $\partial_\nu$ derivative yields either zero, two, or four diagonal entries of a resolvent chain. Whenever the term contains an $xy$ or $yx$ entry, we use isotropic resummation to identify the term as part of the error in~\eqref{eq-cumuexpmeso2}. It remains to consider the case $\nu_1=(yx)$ and $\nu_2\in\{(xy,yx),(yx,xy)\}$. Here, we compute
\begin{align*}
\partial_{\nu_1}(T_{[1,k]})_{yx}&=-\sum_{r=1}^k(T_{[1,r\rangle}G_r)_{yy}(T_{[r,k]})_{xx},\\
\partial_{\nu_2}X^{(\ell)}_{\beta}&=\frac{1}{N}\sum_{s=k+1}^{k+\ell}\Big(\sum_{t=k+1}^s(T_{[s,k+\ell]}T_{[k+1,t\rangle}G_t)_{yy}(T_{[t,s\rangle}G_s)_{xx}\\
&\quad+\sum_{t=s}^{k+\ell}(T_{[s,t\rangle}G_t)_{yy}(T_{[t,k+\ell]}T_{[k+1,s\rangle}G_s)_{xx}\Big)+\dots,
\end{align*}
where the terms that are left out contain at least one off-diagonal entry of a resolvent chain and hence can be included in the error term in~\eqref{eq-cumuexpmeso2} by isotropic resummation. Applying~\eqref{eq-multiGisotropic} and recalling that $\kappa(xy,\nu)=\kappa_4/N^2$ in the case considered, it follows that
\begin{align}
&N\sum_{x,y}\frac{\kappa_4}{N^2}\E \Big(\big(\partial_{\nu_1}(T_{[1,k]})_{yx}\big)\big(\partial_{\nu_2}X^{(\ell)}_{\beta}\big)\Big)\NN\\
&=-\frac{\kappa_4}{N^2}\sum_{x,y}\Big(\sum_{r=1}^k\sum_{s=k+1}^{k+\ell}(M_{[r]})_{yy}(M_{[r,k]}A_k)_{xx}\Big(\sum_{t=k+1}^s(M_{(s,\dots,k+\ell,k+1,\dots,t)})_{yy}(M_{[t,s]})_{xx}\NN\\
&\quad+\sum_{t=s}^{k+\ell}(M_{[s,t]})_{yy}(M_{(t,\dots,k+\ell,k+1,\dots,s)})_{xx}\Big)\Big)+\cO\Big(\frac{N^\eps}{\sqrt{N\eta_*}\ \eta_*^{k-a/2}\eta_*^{\ell-b/2}}\Big)\label{eq-source2}
\end{align}
which can again be rewritten using the definition of the Hadamard product. Note that we obtain six terms of the form~\eqref{eq-source2} in total from applying the Leibniz rule.

\medskip
Case 3 ($|\nu_1|$ even): For both $|\nu_1|=0$ and $|\nu_1|=2$, the derivative always contains at least one off-diagonal entry of a resolvent chain. Therefore, we can apply isotropic resummation and include the term in the error in~\eqref{eq-cumuexpmeso2}.

\medskip
Adding the contributions to~\eqref{eq-cumuexpmeso2} back together, the cumulant expansion evaluates to
\begin{align*}
&-N^2\E\big(\langle \underline{WT_{[1,k]}}\rangle X^{(\ell)}_{\beta}\big)\\
&=\sum_{j=1}^\ell\m[T_1,\dots,T_k,T_{k+j},\dots,T_{k+\ell},T_{k+1},\dots,T_{k+j-1},G_{k+j}]\\
&\quad+\kappa_4\sum_{r=1}^k\sum_{s=k+1}^{k+\ell}\Big(\sum_{t=k+1}^s\langle M_{[r]}\odot M_{(s,\dots,k+\ell,k+1,\dots,t)}\rangle\langle(M_{[r,k]}A_k)\odot M_{[t,s]}\rangle\NN\\
&\quad\quad+\sum_{t=s}^{k+\ell}\langle M_{[r]}\odot M_{[s,t]}\rangle\langle(M_{[r,k]}A_k)\odot M_{(t,\dots,k+\ell,k+1,\dots,s)}\rangle\Big)
+\cO\Big(\frac{N^\eps}{\sqrt{N\eta_*}\ \eta_*^{k-a/2}\eta_*^{\ell-b/2}}\Big).
\end{align*}
We conclude that, up to an $\cO(N^\eps/\sqrt{N})$ error, $\smash{\E(X^{(k)}_\alpha X^{(\ell)}_\beta)}$ equals the right-hand side of~\eqref{eq-Mrecursion}. Applying the recursion yields
\begin{displaymath}
\E\Big(X^{(k)}_\alpha X^{(\ell)}_\beta\Big)=\m[\alpha|\beta]+\cO\Big(\frac{N^\eps}{\sqrt{N\eta_*}\ \eta_*^{k-a/2}\eta_*^{\ell-b/2}}\Big)
\end{displaymath}
which completes the induction step.
\end{proof}

It remains to establish the alternative integral representation for $m_{GUE}[\cdot|\cdot]$.

\begin{proof}[Proof of Corollary~\ref{cor-mintegral}]
We use proof by induction, starting with the base case $k=\ell=1$. Here, the key to establishing~\eqref{eq-integralrep} lies in the fact that the function $m_{GUE}[1|2]$ and its counterpart for $W$ being a GOE matrix only differ by a constant factor. More precisely, evaluating (92) of~\cite{CES-functCLT} for GOE ($\sigma=1$, $\kappa_4=0$, $\widetilde{\om}_2=0$) yields 
\begin{equation}\label{eq-limit2}
\lim_{N\rightarrow\infty}N^2\Cov(\langle G_1-\E G_1\rangle,\overline{\langle G_2-\E G_2\rangle})=2\cdot\frac{m_1'm_2'}{(1-m_1m_2)^2}=2\cdot m_{GUE}[1|2].
\end{equation}
This allows us to obtain the desired integral representation for $m_{GUE}[1|2]$ from~\cite[Thm.~2.3]{DiazJaramilloPardo2022}, which only applies to real symmetric Gaussian matrices. Applying the theorem yields
\begin{equation}\label{eq-limit1}
\lim_{N\rightarrow\infty}N^2\Cov(\langle G_1-\E G_1\rangle,\overline{\langle G_2-\E G_2\rangle})=\int_\R\int_\R\frac{1}{(x-z_1)^2}\frac{1}{(y-z_2)^2}u(x,y)\dx x\dx y,
\end{equation}
where $u$ is given by
\begin{equation}\label{eq-kernel2}
u(x,y)=\frac{\sqrt{4-x^2}\sqrt{4-y^2}}{\pi^2}\cdot\int_0^1\frac{1-w^2}{w^2(x-y)^2-wxy(1-w)^2+(1-w^2)^2}\dx w.
\end{equation}
In the notation of~\cite{DiazJaramilloPardo2022}, Equation~\eqref{eq-limit1} corresponds to considering a matrix-valued Gaussian process for which the distribution at time~1 coincides with the law of a GOE matrix $V$, as well as the functions $f(x)=(x-z_1)^{-1}$ and ${g(y)=(y-z_2)^{-1}}$ for fixed $z_1,z_2\in\C$ with $\Im z_1,\Im z_2\gtrsim1$. We remark that the functions $f$ and $g$ indeed satisfy the polynomial bound assumed in~\cite[Thm.~2.3]{DiazJaramilloPardo2022}. As~\eqref{eq-limit1} is obtained for the fixed times $s=t=1$, we omit the time dependence from the kernel. By substituting $v=(\frac{1-w}{1+w})^2$ and using partial fractions, the $w$-integration in~\eqref{eq-kernel2} can be carried out explicitly, yielding the form in~\eqref{eq-kernel}.

\medskip
Comparing~\eqref{eq-limit2} and~\eqref{eq-limit1}, we obtain~\eqref{eq-integralrep} in the case $k=\ell=1$.

\medskip
The induction step uses the divided difference structure from Corollary~\ref{cor-mproperties}(iii) and is similar to the proof of \cite[Lem.~4.1]{CES-thermalization}. Assume that the integral representation~\eqref{eq-integralrep} holds for $m_{GUE}[S_1|S_2]$ with $|S_1|=k$ and $|S_2|\leq\ell$ for some fixed $k\geq1$. We start by rewriting
\begin{align*}
&m_{GUE}[1,\dots,k|k+1,\dots,k+\ell+1]\\
&=\frac{m_{GUE}[1,\dots,k|k+2,\dots,k+\ell+1]-m_{GUE}[1,\dots,k|k+1,k+3,\dots,k+\ell+1]}{z_{k+2}-z_{k+1}},
\end{align*}
where the induction hypothesis applies to both summands in the nominator, respectively. Noting that
\begin{align*}
&\sum_{i=k+2}^{k+\ell}\frac{1}{(y-z_i)^2}\cdot\prod_{j\neq i}\frac{1}{y-z_j}-\sum_{\substack{i=k+1\\ i\neq k+2}}^{k+\ell}\frac{1}{(y-z_i)^2}\cdot\prod_{j\neq i}\frac{1}{y-z_j}\\
&=(z_{k+2}-z_{k+1})\sum_{i=k+1}^{k+\ell}\frac{1}{(y-z_i)^2}\cdot\prod_{j\neq i}\frac{1}{y-z_j},
\end{align*}
we obtain~\eqref{eq-integralrep} as claimed. Since $m_{GUE}[S_1|S_2]=m_{GUE}[S_2|S_1]$ by definition, the same argument applies for the other entry of $m_{GUE}[\cdot|\cdot]$.
\end{proof}

\subsection{Proof of Theorem~\ref{thm-resolventCLTmeso} (CLT for Resolvents)}\label{sect-proof-resolventCLT}
We use proof by induction over the number of factors on the left-hand side of~\eqref{eq-resolventCLTmeso}. To keep the notation simple, we drop the superscripts $(k_j,a_j)$ of $\smash{X_{\alpha_j}^{(k_j,a_j)}}$ throughout this section. First, the base case $p=1,2$ readily follows from the definition and Lemma~\ref{lem-covapproxmeso}, which give $\E X_{\alpha_1}=0$ and
\begin{align*}
N^2\E (X_{\alpha_1} X_{\alpha_2})=\m[\alpha_1|\alpha_2]+\cO\Big(\frac{N^\eps}{\sqrt{N\eta_*}\ \eta_*^{k_1-a_1/2}\eta_*^{k_2-a_2/2}}\Big),
\end{align*}
respectively. Next, fix $p\in\N$ and assume that
\begin{equation}\label{eq-indhyp1}
N^p\E\Big(\prod_{j=1}^p X_{\alpha_j}\Big)=\sum_{Q\in Pair(\{1,\dots n\})}\prod_{(i,j)\in Q}\m[\alpha_i|\alpha_j]+\cO\Big(\frac{N^\eps}{\sqrt{N\eta_*}\ \prod_{l=1}^p\eta_*^{k_l-a_l/2}}\Big)
\end{equation}
holds for $n=1,\dots,p$. It remains to consider $N^{p+1}\E (X_{\alpha_1}\dots X_{\alpha_{p+1}})$. When writing out any of the factors, we label the spectral parameters and matrices involved with a superscript $1,\dots,p+1$ according to the factor they appear in.

\medskip
We start by considering the case $k_1=1$ and inductively extend the statement to larger products. 
For the first part of the induction step, we need to compute
\begin{equation}\label{eq-product1}
N^{p+1}\E \Big((\langle T_1^{(1)}\rangle-\E\langle T_1^{(1)}\rangle)X_{\alpha_2}\dots X_{\alpha_{p+1}}\Big).
\end{equation}
Combining~\eqref{eq-Gidentity2} and Lemma~\ref{lem-mEexchangemeso} yields
\begin{align}
\langle T_1^{(1)}\rangle-\E\langle T_1^{(1)}\rangle&=-m(z_1^{(1)})\Big(\langle\underline{WT_1^{(1)}}\rangle+q_{1,1}^{(1)}\langle\underline{WG_1^{(1)}}\rangle\langle A_1^{(1)}\rangle\Big)\NN\\
&\quad-\frac{\kappa_4}{N}\cE[T_1]+\cO\Big(\frac{N^\eps}{N\ \sqrt{N\eta_*}\ \eta_*^{k_1-a_1/2}}\Big),\label{eq-smallchainleading}
\end{align}
where the superscript in $q^{(1)}$ indicates that the arguments are taken from $\alpha_1$. We use~\eqref{eq-smallchainleading} to replace the first factor in~\eqref{eq-product1}. The underlined terms are again treated by cumulant expansion. This yields
\begin{align}
&-m(z_1^{(1)})N^{p+1}\E\Big(\langle \underline{WT_1}\rangle X_{\alpha_2}\dots X_{\alpha_{p+1}}\Big)\NN\\
&=m(z_1^{(1)})\sum_{i=2}^{p+1}\sum_{j=1}^{k_i}\E\Big(\langle T_1^{(1)}T_j^{(i)}\dots T_{k_i}^{(i)}T_1^{(i)}\dots T_{j-1}^{(i)}G_j^{(i)}\rangle\cdot N^{p-1}\prod_{r\neq 1,i}X_{\alpha_r}\Big)\NN\\
&\quad-m(z_1^{(1)})N^p\sum_{n\geq2}\sum_{x,y\in[N]}\sum_{\nu\in\{xy,yx\}^n}\frac{\kappa(xy,\nu)}{n!}\E \partial_{\nu}\Big((T_1^{(1)})_{yx}X_{\alpha_2}\dots X_{\alpha_{p+1}}\Big),\label{eq-cumuexp3}
\end{align}
where the $n=1$ contribution was again evaluated using~\eqref{eq-underline} and the computations in the proof of Lemma~\ref{lem-apriorimeso}. Note that we obtain one term for each resolvent in the product $X_{\alpha_2}\dots X_{\alpha_{p+1}}$ from applying the Leibniz rule. By the local law~\eqref{eq-multiGaveraged}, we have
\begin{align*}
\E\langle T_1^{(1)}T_j^{(i)}\dots T_{k_i}^{(i)}T_1^{(i)}\dots T_{j-1}^{(i)}G_j^{(i)}\rangle=\m[T_1^{(1)},T_j^{(i)},\dots,T_{j-1}^{(i)},G_j^{(i)}]+\cO\Big(\frac{N^\eps}{N\eta_*\ \eta_*^{k_i-(a_1+a_i)/2}}\Big).
\end{align*}
Further, note that every $X_{\alpha_j}$ contains a normalized trace, which yields a total of $p$ factors~$N^{-1}$ when the derivative $\smash{\partial_{\nu}((T_1^{(1)})_{yx}X_{\alpha_2}\dots X_{\alpha_{p+1}})}$ is evaluated. We can hence argue as in the proof of Lemma~\ref{lem-covapproxmeso} to conclude that the contributions for $n\geq4$ in~\eqref{eq-cumuexp3} are sub-leading. Similarly, most of the terms arising for $n=2$ and $n=3$ in~\eqref{eq-cumuexp3} were already computed in the proof of Lemma~\ref{lem-covapproxmeso} and their treatment here is analogous. We, therefore, focus on the differences between~\eqref{eq-cumuexpmeso2} and~\eqref{eq-cumuexp3} below. Recall that $|\nu_1|$ denotes the total number of derivatives acting on $\smash{(T_1^{(1)})_{yx}}$ in a given term.

\medskip
\underline{\smash{Estimate for the $n=2$ term of~\eqref{eq-cumuexp3}}:} Compared to~\eqref{eq-cumuexpmeso2}, new terms only arise if ${|\nu_1|=0}$ and the two derivatives act on different factors of the product $X_{\alpha_2}\dots X_{\alpha_{p+1}}$. As the resulting derivative always includes the off-diagonal entry~$\smash{(T_1^{(1)})_{yx}}$, we obtain
\begin{displaymath}
N^p\sum_{x,y}\sum_{\nu\in\{xy,yx\}^2}\frac{\kappa(xy,\nu)}{2}\E \partial_{\nu}\Big((T_1^{(1)})_{yx}X_{\alpha_2}\dots X_{\alpha_{p+1}}\Big)=\cO\Big(\frac{N^\eps}{\sqrt{N\eta_*}\ \prod_{l=1}^{p+1}\eta_*^{k_l-a_l/2}}\Big)
\end{displaymath}
by isotropic resummation. Recall that every $X_{\alpha_j}$ contains a normalized trace such that the factor $N^{-p}$ in front of the sum is balanced out.

\medskip
\underline{\smash{Computation of the $n=3$ term of~\eqref{eq-cumuexp3}}:} We distinguish three cases for $|\nu_1|$.

\medskip
Case 1 ($|\nu_1|=3$): After evaluating the derivative, we apply~\eqref{eq-multiGisotropic} to get
\begin{align}
&N\sum_{x,y}\sum_{\nu\in\{xy,yx\}^3}\frac{\kappa(xy,\nu)}{6}\E \Big((\partial_{\nu}(T_1)_{yx})X_{\alpha_2}\dots X_{\alpha_{p+1}}\Big)\NN\\
&=-\kappa_4m(z_1^{(1)})^4\langle A_1^{(1)}\rangle\E\Big(X_{\alpha_2}\dots X_{\alpha_{p+1}}\Big)+\cO\Big(\frac{N^\eps}{\sqrt{N\eta_*}\ \prod_{l=1}^{p+1}\eta_*^{k_l-a_l/2}}\Big).\NN
\end{align}
In particular, adding the above terms for the cumulant expansions resulting from the two underlined terms in~\eqref{eq-smallchainleading} yields
\begin{align*}
&-m(z_1^{(1)})\Big(-\kappa_4m(z_1^{(1)})^4\langle A_1^{(1)}\rangle-\kappa_4q_{1,1}^{(1)}m(z_1^{(1)})^4\langle A_1^{(1)}\rangle\Big)\E\Big(X_{\alpha_2}\dots X_{\alpha_{p+1}}\Big)\\
&=\kappa_4\cE[T_1]\E\Big(X_{\alpha_2}\dots X_{\alpha_{p+1}}\Big).
\end{align*}

\medskip
Case 2 ($|\nu_1|=1$): Whenever the two remaining derivatives act on different factors of the product $X_{\alpha_2}\dots X_{\alpha_{p+1}}$,~\eqref{eq-apriori2meso} always yields an off-diagonal entry of a resolvent chain that can be used to apply isotropic resummation. We can thus include any terms of this build in the error in~\eqref{eq-cumuexp3}. Otherwise, we evaluate the term similar to~\eqref{eq-source2}. Let again $\nu_2=\nu\setminus\nu_1$. For $i=2,\dots,p+1$, it follows that
\begin{align*}
&N^p\sum_{x,y}\frac{\kappa_4}{N^2}\E \Big(\big(\partial_{\nu_1}(T_1)_{yx}\big)X_{\alpha_2}\dots X_{\alpha_{i-1}}\big(\partial_{\nu_2}X_{\alpha_i}\big)X_{\alpha_{i+1}}\dots X_{\alpha_{p+1}}\Big)\NN\\
&=-\kappa_4\sum_{s=1}^{k_i}\Big(\sum_{t=1}^sm(z_1^{(1)})^2\langle A_1^{(1)}\odot M_{[s,t]}^{(i)})\rangle\langle M_{(t,\dots,k_i,1,\dots,s)}^{(i)}\rangle\NN\\
&\quad +\sum_{t=s}^{k_i}m(z_1^{(1)})^2\langle A_1^{(1)}\odot M_{(s,\dots,k_i,1,\dots,t)}^{(i)})\rangle\langle M_{[t,s]}^{(i)})\rangle\Big)\E\Big(\prod_{r\neq1,i}X_{\alpha_r}\Big)+\cO\Big(\frac{N^\eps}{\sqrt{N\eta_*}\ \prod_{l=1}^{p+1}\eta_*^{k_l-a_l/2}}\Big),
\end{align*}
where the superscript in $M^{(i)}$ indicates that the arguments are taken from $\alpha_i$. Recall that~$\odot$ denotes the Hadamard product. We emphasize that this term is obtained six times in total when all mixed derivatives obtained from the Leibniz rule are summed up.

\medskip
The last case for $n=3$ ($|\nu_1|$ even) can be treated similarly to the corresponding cases of~\eqref{eq-cumuexpmeso2}. We omit the details.

\medskip
Overall, the cumulant expansion~\eqref{eq-cumuexp3} evaluates to
\begin{align*}
&N^{p+1}\E \Big(-m(z_1^{(1)})\langle \underline{WT_1^{(1)}}\rangle X_{\alpha_2}\dots X_{\alpha_{p+1}}\Big) \\
&=N^{p-1}\E\Big(\prod_{r\neq1,j}X_{\alpha_r}\Big)\cdot m(z_1^{(1)})\sum_{i=2}^{p+1}\Big[\sum_{j=1}^{k_i}\m[T_1^{(1)},T_j^{(i)}\dots T_{j-1}^{(i)},G_j^{(i)}]\\
&\quad+q_{1,1}^{(1)}\sum_{j=1}^{k_i}\m[G_1^{(1)},T_j^{(i)}\dots T_{j-1}^{(i)},G_j^{(i)}]\langle A_1^{(1)}\rangle\\
&\quad+\kappa_4\sum_{s=1}^{k_i}\Big(\sum_{t=1}^sm(z_1^{(1)})^2\langle A_1^{(1)}\odot M_{[s,t]}^{(i)}\rangle\langle M_{(t,\dots,k_i,1,\dots,s)}^{(i)}\rangle\\
&\quad+\sum_{t=s}^{k_i}m(z_1^{(1)})^2\langle A_1^{(1)}\odot M_{(t,\dots,k_i,1,\dots,s)}^{(i)}\rangle\langle M_{[s,t]}^{(i)}\rangle\Big)\\
&\quad +\kappa_4q_{1,1}^{(1)}\sum_{s=1}^{k_i}\Big(\sum_{t=1}^sm(z_1^{(1)})^2\langle A_1^{(1)})_{xx}\odot M_{[s,t]}^{(i)})\rangle\langle M_{(t,\dots,k_i,1,\dots,s)}^{(i)}\rangle\\
&\quad+\sum_{t=s}^{k_i}m(z_1^{(1)})^2\langle A_1^{(1)}\odot M_{(t,\dots,k_i,1,\dots,s)}^{(i)}\rangle\langle M_{[s,t]}^{(i)}\rangle\Big)\langle A_1^{(1)}\rangle\Big]+\cO\Big(\frac{N^\eps}{\sqrt{N\eta_*}\ \prod_{l=1}^{p+1}\eta_*^{k_l-a_l/2}}\Big).
\end{align*}
Adding the contribution for all three terms in~\eqref{eq-smallchainleading} together allows rewriting~\eqref{eq-product1} as
\begin{align*}
&N^{p+1}\E \Big((\langle T_1^{(1)}\rangle-\E\langle T_1^{(1)}\rangle)X_{\alpha_2}\dots X_{\alpha_{p+1}}\Big)\\
&=\sum_{i=2}^{p+1}\m[T_1^{(1)}|T_1^{(i)},\dots,T_{k_i}^{(i)}]\cdot N^{p-1}\E\Big(\prod_{r\neq1,j}X_{\alpha_r}\Big)+\cO\Big(\frac{N^\eps}{\sqrt{N\eta_*}\ \prod_{l=1}^{p+1}\eta_*^{k_l-a_l/2}}\Big),
\end{align*}
where we used~\eqref{eq-Mrecursion} with $k=1$ and $\ell=k_i$. As the remaining product only consists of $p-1$ factors, we can apply the induction hypothesis~\eqref{eq-indhyp1} and conclude
\begin{equation}\label{eq-indhyp2}
N^{p+1}\E\Big(\prod_{j=1}^{p+1} X_{\alpha_j}\Big)=\sum_{Q\in Pair([p+1])}\prod_{(i,j)\in Q}\m[\alpha_i|\alpha_j]+\cO\Big(\frac{N^\eps}{\sqrt{N\eta_*}\ \prod_{l=1}^{p+1}\eta_*^{k_l-a_l/2}}\Big)
\end{equation}
for the case $\smash{\alpha_1=\{(z_1^{(1)},A_1^{(1)})\}}$. Assume next that~\eqref{eq-indhyp2} holds for all multi-indices $\alpha_1$ that consist of up to $k_1-1$ pairs $\smash{(z_j^{(1)},A_j^{(1)})}$, and consider $\alpha_1$ of length $k_1$. As in the base case, we replace the factor $X_{\alpha_1}$ in the product by its leading term using~\eqref{eq-Edifference1meso}. This gives
\begin{align}
&N^{p+1}\E \Big(X_{\alpha_1}\dots X_{\alpha_{p+1}}\Big)\NN\\
&=N^{p+1}m(z_1^{(1)})\E \Big(\big(-\langle \underline{WT_1^{(1)}\dots T_{k_1}^{(1)}}\rangle\NN\\
&\quad+(\langle T_2^{(1)}\dots T_{k_1-1}^{(1)}G_{k_1}^{(1)}A_{k_1}^{(1)}A_1^{(1)}\rangle-\E\langle T_2^{(1)}\dots T_{k_1-1}^{(1)}G_{k_1}^{(1)}A_{k_1}^{(1)}A_1^{(1)}\rangle\NN\\
&\quad +\frac{\kappa_4}{N}\cE[T_2^{(1)},\dots,T_{k_1-1}^{(1)},G_{k_1}^{(1)}A_{k_1}^{(1)}A_1^{(1)}])\NN \\
&\quad+\sum_{j=1}^{k-1}(\langle T_1^{(1)}\dots T_{j-1}^{(1)}G_j^{(1)}\rangle-\E\langle T_1^{(1)}\dots T_{j-1}^{(1)}G_j^{(1)}\rangle+\frac{\kappa_4}{N}\cE[T_1^{(1)},\dots,T_{j-1}^{(1)},G_j^{(1)}])\m[T_j^{(1)},\dots,T_{k_1}^{(1)}]\NN\\
&\quad+\sum_{j=2}^k\m[T_1^{(1)},\dots,T_{j-1}^{(1)},G_j^{(1)}](\langle T_j^{(1)}\dots T_{k_1}^{(1)}\rangle-\E\langle T_j^{(1)}\dots T_{k_1}^{(1)}\rangle+\frac{\kappa_4}{N}\cE[T_j^{(1)},\dots,T_{k_1}^{(1)}])\big)\NN\\
&\quad+\frac{\kappa_4}{N}\cE[T_1^{(1)},\dots,T_{k_1}^{(1)}]\Big)X_{\alpha_2}\dots X_{\alpha_{p+1}}\Big)+\cO\Big(\frac{N^\eps}{\sqrt{N\eta_*}\ \prod_{l=1}^{p+1}\eta_*^{k_l-a_l/2}}\Big),\label{eq-factorreplacedgeneral}
\end{align}
where the underlined term can again be treated by cumulant expansion. Similar to~\eqref{eq-cumuexp3}, we evaluate
\begin{align}
&N^{p+1}\E \Big(-m(z_1^{(1)})\langle \underline{WT_1^{(1)}\dots T_{k_1}^{(1)}}\rangle X_{\alpha_2}\dots X_{\alpha_{p+1}}\Big)\NN\\
&=m(z_1^{(1)})\sum_{i=2}^{p+1}\Bigg(\sum_{j=1}^{k_i}\m[T_1^{(1)},\dots, T_{k_1}^{(1)},T_j^{(i)},\dots ,T_{j-1}^{(i)},G_j^{(i)}]\NN\\
&\quad+\kappa_4\sum_{r=1}^{k_1}\sum_{s=k+1}^{k+\ell}\Big(\sum_{t=1}^s\langle M_{[r]}^{(1)}\odot M_{(s,\dots,k_i,1,\dots,t)}^{(i)}\rangle\langle(M_{[r,k]}^{(1)}A_k^{(1)})\odot M_{[t,s]}^{(i)}\rangle\label{eq-cumuexp4}\\
&\quad+\sum_{t=s}^{k_i}\langle M_{[r]}^{(1)}\odot M_{[s,t]}^{(i)}\rangle\langle(M_{[r,k]}^{(1)}A_k^{(1)})\odot M_{(t,\dots,k_i,1,\dots,s)}^{(i)}\rangle\Big)\Bigg)N^{p-1}\E\Big(\prod_{r\neq 1,i}X_{\alpha_r}\Big)\NN\\
&\quad +\cO\Big(\frac{N^\eps}{\sqrt{N\eta_*}\ \prod_{l=1}^{p+1}\eta_*^{k_l-a_l/2}}\Big),\NN
\end{align}
and apply the induction hypothesis~\eqref{eq-indhyp1} for the product of $p-1$ factors. Note that any terms remaining in~\eqref{eq-factorreplacedgeneral} now involve at most $k_1-1$ resolvents in the first factor. Hence,~\eqref{eq-indhyp2} applies and we obtain, e.g.,
\begin{align*}
&N^{p+1}\E\Big((\langle T_2^{(1)}\dots T_{k_1-1}^{(1)}G_{k_1}^{(1)}A_{k_1}^{(1)}A_1^{(1)}\rangle-\E\langle T_2^{(1)}\dots T_{k_1-1}^{(1)}G_{k_1}^{(1)}A_{k_1}^{(1)}A_1^{(1)}\rangle)\prod_{r=2}^{p+1}X_{\alpha_r}\Big)\\
&=\sum_{r=2}^{p+1}\m[T_2^{(1)},\dots, T_{k_1-1}^{(1)},G_{k_1}^{(1)}A_{k_1}^{(1)}A_1^{(1)}|T_1^{(r)},\dots,T_{k_i}^{(r)}]\Big(\sum_{Q\in Pair([p+1]\setminus\{1,i\})}\prod_{(i,j)\in Q}\m[\alpha_i|\alpha_j]\Big)\NN\\
&\quad+\cO\Big(\frac{N^\eps}{\sqrt{N\eta_*}\ \prod_{l=1}^{p+1}\eta_*^{k_l-a_l/2}}\Big).
\end{align*}
Moreover, note that~\eqref{eq-cumuexp4} contains the source term of the recursion for $\cE[\cdot]$ (cf. Lemma~\ref{lem-mEexchangemeso}), which we can combine with the terms involving $\cE[\cdot]$ in~\eqref{eq-factorreplacedgeneral}. Using~\eqref{eq-Erecursion}, the terms cancel. We conclude that~\eqref{eq-factorreplacedgeneral} evaluates to
\begin{align*}
N^{p+1}\E \Big(X^{(k_1)}_{\alpha_1}\dots X^{(k_{p+1})}_{\alpha_{p+1}}\Big)&=\sum_{r=2}^{p+1}\Big[m(z_1^{(1)})\Big(\sum_{j=1}^{k_r}\m[T_1^{(1)},\dots, T_{k_1}^{(1)},T_j^{(r)},\dots ,T_{j-1}^{(r)},G_j^{(r)}]\\
&\quad+\kappa_4\sum_{r=1}^{k_1}\sum_{s=k+1}^{k+\ell}\Big(\sum_{t=1}^s\langle M_{[r]}^{(1)}\odot M_{(s,\dots,k_i,1,\dots,t)}^{(i)}\rangle\langle (M_{[r,k]}^{(1)}A_k^{(1)})\odot M_{[t,s]}^{(i)}\rangle\NN\\
&\quad+\sum_{t=s}^{k_i}\langle M_{[r]}^{(1)}\odot M_{[s,t]}^{(i)}\rangle\langle (M_{[r,k]}^{(1)}A_k^{(1)})\odot M_{(t,\dots,k_i,1,\dots,s)}^{(i)}\rangle\Big)\NN\\
&\quad +\m[T_2^{(1)},\dots, T_{k_1-1}^{(1)},G_{k_1}^{(1)}A_{k_1}^{(1)}A_1^{(1)}|T_1^{(r)},\dots,T_{k_r}^{(r)}]\\
&\quad+\sum_{j=2}^k\m[T_1^{(1)},\dots,T_{j-1}^{(1)},G_j^{(1)}]\m[T_j^{(1)},\dots,T_{k_1}^{(1)}|T_1^{(r)},\dots,T_{k_r}^{(r)}]\NN\\
&\quad+\sum_{j=2}^k\m[T_1^{(1)},\dots,T_{j-1}^{(1)},G_j^{(1)}|T_1^{(r)},\dots,T_{k_r}^{(r)}]\m[T_j^{(1)},\dots,T_{k_1}^{(1)}]\Big)\Big]\\
&\quad\times\Big(\sum_{Q\in Pair([p+1]\setminus\{1,i\})}\prod_{(i,j)\in Q}\m[\alpha_i|\alpha_j]\Big)+\cO\Big(\frac{N^\eps}{\sqrt{N\eta_*}\ \prod_{l=1}^{p+1}\eta_*^{k_l-a_l/2}}\Big)\\
&=\sum_{Q\in Pair([p+1])}\prod_{(i,j)\in Q}\m[\alpha_i|\alpha_j]+\cO\Big(\frac{N^\eps}{\sqrt{N\eta_*}\ \prod_{l=1}^{p+1}\eta_*^{k_l-a_l/2}}\Big)
\end{align*}
where the last equation follows from~\eqref{eq-Mrecursion}. Hence,~\eqref{eq-indhyp2} also holds for $\alpha_1$ of length $k_1$. Moreover,~\eqref{eq-indhyp1} stays true if the product on the left-hand side contains $p+1$ factors, which concludes the proof of~\eqref{eq-resolventCLTmeso}.\qed

\subsection{Proof of Theorem~\ref{thm-functCLT} (Multi-Point Functional CLT)}\label{sect-functCLTproof}
 
The proof of the multi-point functional CLT in Theorem~\ref{thm-functCLT} consists of two parts. In the first step, we use Helffer-Sjöstrand representation (see~\cite{Davies1995}) to express $f_1(W)\dots f_k(W)$ as an integral of products of resolvents at different spectral parameters. This relates the linear statistics $Y_{\alpha}$ back to the resolvent chains studied in Section~\ref{sect-results1}. The second step is the computation of the leading terms, which establishes the covariance structure in~\eqref{eq-covfunctionsgeneral}.

\medskip
By eigenvalue rigidity (see e.g., \cite[Thm.~7.6]{EKYY2013} or~\cite{EYY2012}), the spectrum of $W$ is contained in $[-2-\eps,2+\eps]$ for any small $\eps>0$ with very high probability. In particular, we have $(f\cdot\chi)(W)=f(W)$ with very high probability for any smooth cutoff function $\chi$ that is, e.g., equal to one on $[-5/2,5/2]$ and equal to zero on $[-3,3]^c$. It is thus sufficient to consider $f_j\in H^p_0([-3,3])=:H^p_0$, i.e., Soboloev functions on $\R$ that are non-zero only on $[-3,3]$. Moreover, recall that every deterministic matrix $A_j$ in the product $F_{[1,k]}$ can be decomposed as $\smash{A_j=\langle A_j\rangle\Id+\mathring{A}_j}$ with $\smash{\langle\mathring{A_j}\rangle=0}$. We thus assume w.l.o.g. that the deterministic matrices $A_j$ are either traceless or equal to the identity matrix. Moreover, we restrict the following argument to the case $a=k$, i.e., all deterministic matrices are traceless, and fix $p=\lceil k/2\rceil+1$ (resp. $q=\lceil \ell/2\rceil+1$) throughout. The proof in the general case is analogous and hence omitted.

\medskip
Let $f\in H^p_0$ and define the \emph{almost analytic extension} of $f$ of order $p$ by
\begin{equation}\label{eq-almostanalytic}
f_{\C}(z)=f_{\C,p}(x+\ri\eta):=\Big[\sum_{j=0}^{p-1}\frac{(\ri\eta)^j}{j!}f^{(j)}(x)\Big]\widetilde{\chi}(N^\gamma\eta)
\end{equation}
where $\widetilde{\chi}$ is a smooth cutoff function that is equal to one on $[-5,5]$ and vanishes on $[-10,10]^c$. Note that~\eqref{eq-almostanalytic} together with~\eqref{eq-testfunct} implies the bound
\begin{equation}\label{eq-aacintbound}
\int_\R|\partial_{\overline{z}}f_{\C,p}(x+\ri\eta)|\dx x\lesssim \eta^{p-1}\|f\|_{H^p}\lesssim \eta^{p-1}N^{\gamma p}.
\end{equation}
By Helffer-Sjöstrand representation, we have 
\begin{equation}\label{eq-HSformula1}
f(\lambda)=\frac{1}{\pi}\int_\C\frac{\partial_{\overline{z}}f_{\C}(z)}{\lambda-z}\dx^2z,
\end{equation}
where $\dx^2z=\dx x\dx \eta$ denotes the Lebesgue measure on $\C\equiv\R^2$ with $z=x+\ri\eta$ and $\partial_{\overline{z}}=(\partial_x+\ri\partial_{\eta})/2$. Applying~\eqref{eq-HSformula1} for $f_1,\dots,f_k$, respectively, we obtain
\begin{align}\label{eq-HSformula2}
Y_{\alpha}^{(k,a)}&=\frac{N}{\pi^k}\int_{\C^k}\Big[\prod_{j=1}^k(\partial_{\overline{z}}(f_j)_{\C})(z_j)\Big]\Big(\langle G(z_1)\dots G(z_k)A_k\rangle-\E\langle G(z_1)\dots G(z_k)A_k\rangle\Big)\dx^2 z_{[k+\ell]}
\end{align}
with $\dx^2 z_{[k+\ell]}=\dx^2 z_1\dots\dx^2 z_{k+\ell}$. Let $c>0$ and define
\begin{displaymath}
\eta_0:=N^{-1+c}.
\end{displaymath}
We start by showing that the contribution from the regime $|\eta_j|\leq\eta_0$ for some $j\in[k]$ to the integral~\eqref{eq-HSformula2} is negligible. W.l.o.g. assume that $|\eta_j|\leq\eta_0$ only for the single index $j=1$. The general case is similar and yields an even smaller bound (cf. proof of~\cite[Thm.~2.6]{CES-thermalization}). Our key tool is the following variant of Stokes' theorem
\begin{equation}\label{eq-Stokes}
\int_{-10}^{10}\int_{\widetilde{\eta}}^{10}\partial_{\overline{z}}\Phi(x+\ri \eta)h(x+\ri \eta)\dx x\dx \eta=\frac{1}{2\ri}\int_{-10}^{10}\Phi(x+\ri\widetilde{\eta})h(x+\ri\widetilde{\eta})\dx x,
\end{equation}
which holds for any $\widetilde{\eta}\in[0,10]$, and for any $\Phi, h\in H^1(\C)\equiv H^1(\R^2)$ such that $\partial_{\overline{z}}h=0$ on the domain of integration and $\Phi$ vanishes on the left, right, and top boundary of the domain of integration. Applying~\eqref{eq-Stokes} repeatedly for the variables $z_2,\dots,z_k$ and introducing the interval notation $\dx x_{[i,j]}=\dx x_i\dx x_{i+1}\dots\dx x_{j}$ as well as $\dx\eta_{[i,j]}=\dx\eta_i\dx\eta_{i+1}\dots\dx\eta_{j}$ for $i<j$, we obtain
\begin{align*}
&\Big|\int\dx x_{[k]}\int_{\substack{|\eta_i|\geq\eta_0,\\ i\in[2,k]}}\dx\eta_{[2,k]}\int_{-\eta_0}^{\eta_0}\dx\eta_1\Big[\prod_{j=1}^k(\partial_{\overline{z}}(f_j)_{\C})(z_j)\Big]\big(\langle G(z_1)A_1\dots G(z_k)A_k\rangle\\
&\quad-\E\langle G(z_1)A_1\dots G(z_k)A_k\rangle\big)\Big|\\
&=\frac{1}{2^{k-1}}\Big|\int\dx x_{[k]}\int_{-\eta_0}^{\eta_0}\dx\eta_1(\partial_{\overline{z}}(f_1)_{\C})(z_1)\Big[\prod_{j=2}^k(f_j)_{\C}(x_j+\ri\eta_0)\Big]\\
&\quad\times\big(\langle G(z_1)A_1G(x_2+\ri\eta_0)\dots G(x_k+\ri\eta_0)A_k\rangle-\E\langle G(z_1)A_1G(x_2+\ri\eta_0)\dots G(x_k+\ri\eta_0)A_k\rangle\big)\Big|\\
&=:I_1+I_2,
\end{align*}
where $I_1$ and $I_2$ contain the $|\eta_1|\leq\eta_r:=N^{-5k}$ and the $\eta_r\leq|\eta_1|\leq\eta_0$ regime of the $\eta_1$ integration, respectively. For the smallest values of $\eta_1$, the trivial estimate
\begin{displaymath}
|\langle G(z_1)A_1\dots G(z_k)A_k\rangle|\leq\prod_j\|G(z_j)A_j\|\leq\prod_j|\eta_j|^{-1}
\end{displaymath}
together with~\eqref{eq-aacintbound} implies that
\begin{displaymath}
I_1\lesssim N^{-k}
\end{displaymath}
as the $x_j$-integral of $(f_j)_{\C}(x_j+\ri\eta_0)$ for $j\in[2,k]$ is of order one due to Assumption~\ref{as-functions}. For $I_2$, we use the bound
\begin{displaymath}
|\langle G(z_1)A_1\dots G(z_k)A_k\rangle|\prec N^{k/2-1}\prod_{j\in [k]}\frac{1}{\rho(x_i+\ri N^{-2/3})}\Big(1+\frac{1}{N|\eta_j|}\Big)
\end{displaymath}
from~\cite[Lem.~6.1]{CES-optimalLL}. Here, $\rho(z):=\pi^{-1}|\Im m(z)|$ for $z\in\C\setminus\R$ denotes the harmonic extension of the semicircle density with $\rho(x+\ri 0)=\rho_{sc}(x)$. This yields
\begin{displaymath}
I_2\prec \eta_0(N\eta_0)^{k/2}\|f_1\|_{H^p}.
\end{displaymath}
Overall, we conclude that
\begin{align}
Y_{\alpha}^{(k,a)}&=\frac{N}{\pi^k}\int_{\R^k}\dx x_{[k]}\int_{[\eta_0,10]^k}\dx\eta_{[k]}\Big[\prod_{j=1}^k(\partial_{\overline{z}}(f_j)_{\C})(z_j)\Big]\NN\\
&\quad\times\big(\langle G(z_1)A_1\dots G(z_k)A_k\rangle-\E\langle G(z_1)A_1\dots G(z_k)A_k\rangle\big)\label{eq-backtoresolv}\\
&\quad+\cO_{\prec}\Big(\eta_0(N\eta_0)^{k/2}\max_j\|f_j\|_{H^p}\Big)\NN
\end{align}
Equation~\eqref{eq-backtoresolv} now reduces the proof of the multi-point functional CLT for general $f_1,\dots,f_k$ to the CLT for resolvents in Theorem~\ref{thm-resolventCLTmeso}.

\medskip
It remains to compute the covariance structure~\eqref{eq-covfunctionsgeneral}. By~\eqref{eq-backtoresolv} and Lemma~\ref{lem-covapproxmeso}, we have
\begin{align}
&\E\Big(Y^{(k,a)}_{\alpha}Y^{(\ell,b)}_{\beta}\Big)\NN\\
&=\frac{1}{\pi^{k+\ell}}\int_{\R^k}\dx x_{[k]}\int_{[\eta_0,10]^k}\dx\eta_{[k]}\Big[\prod_{i=1}^k(\partial_{\overline{z}}(f_i)_{\C,p})(z_i)\Big]\int_{\R^\ell}\dx x_{[k+1,k+\ell]}\int_{[\eta_0,10]^\ell}\dx\eta_{[k+1,k+\ell]}\NN\\
&\quad\times\Big[\prod_{j=k+1}^\ell(\partial_{\overline{z}}(f_j)_{\C,q})(z_j)\Big]\m[G(z_1)A_1,\dots,G(z_k)A_k|G(z_{k+1})A_{k+1},\dots,G(z_{k+\ell})A_{k+\ell}]\NN\\
&\quad+\cO_{\prec}\Big(\frac{N^\eps\max_{i\in[k]}\|f_i\|_{H^p}\max_{j\in[k+1,k+\ell]}\|f_j\|_{H^q}}{\sqrt{N}}\Big)\label{eq-covfunctionsintegral}
\end{align}
where we estimated the error coming from Lemma~\ref{lem-covapproxmeso} as
\begin{align*}
&\frac{1}{\pi^{k+\ell}}\int_{\R^{k+\ell}}\dx x_{[k+\ell]}\int_{[\eta_0,10]^{k+\ell}}\dx\eta_{[k+\ell]}\Big[\prod_{j=1}^{k+\ell}(\partial_{\overline{z}}(f_j)_{\C})(z_j)\Big]\\
&\quad\times \Big(N^2\E\big[(\langle G(z_1)\dots A_k\rangle-\E\langle G(z_1)\dots A_k\rangle)(G(z_{k+1})\dots A_{k+\ell}\rangle-\E\langle G(z_{k+1})\dots A_k\rangle)\big]\\
&\quad\quad-\m[G(z_1)A_1,\dots,G(z_k)A_k|G(z_{k+1})A_{k+1},\dots,G(z_{k+\ell}) A_{k+\ell}]\Big)\\
&=\cO\Big(\frac{N^\eps\eta_0^{3/2}\max_{i\in[k]}\|f_i\|_{H^p}\max_{j\in[k+1,k+\ell]}\|f_j\|_{H^q}}{\sqrt{N}}\Big).
\end{align*}
More precisely, we considered the regime $\eta_1\leq\dots\leq\eta_k$ and $\eta_{k+1}\leq\dots\leq\eta_{k+\ell}$, as all other regimes give the same contribution by symmetry, and applied~\eqref{eq-Stokes} for the variables $i\in[2,k]$ and $j\in[k+2,k+\ell]$. Estimating the remaining $\partial_{\overline{z}}(f_1)_{\C,p}(z_1)$ and $\partial_{\overline{z}}(f_{k+1})_{\C,q}(z_{k+1})$ using~\eqref{eq-aacintbound}, we obtain a bound of order $|\eta_0|^{p+q}$. Recall that applying Lemma~\ref{lem-covapproxmeso} yields an $\smash{(\sqrt{N\eta_*}\eta_*^{k+\ell-(a+b)/2})^{-1}}$ error, where $\eta_*=\min_j\eta_j=\eta_0$ due to the choice of domain of integration.
\qed

\subsection{Proof of Corollary~\ref{cor-covarianceLL} (Limiting Covariance in Theorem~\ref{thm-functCLT})}\label{sect-bigformula}
Equation~\eqref{eq-covfunctionsprecise} is immediate from the explicit formula for $\m_{GUE}[\cdot|\cdot]$ in~\cite[Thm.~2.4]{JRcompanion}. Moreover, the proof of~\eqref{eq-phis1} identical to~\cite[Lem.~4.1]{CES-thermalization}, as the integral defining $\Phi_{\pi}$ only involves first-order free cumulants. We hence only focus on the proof of~\eqref{eq-phis2}. Abbreviate $U=U_1\cup U_2$ and note that
\begin{align*}
\Phi_{\pi_1\times\pi_2,U_1\times U_2}(f_1,\dots,f_{k+\ell})&=\frac{1}{\pi^{k+\ell}}\Big(\int_{\R^{|U|}}\int_{[\eta_0,10]^{|U|}}\Big[\prod_{j\in U}(\partial_{\overline{z}}(f_j)_{\C})(z_j)\Big]m_{\circ\circ}[U_1|U_2]\prod_{j\in U}\dx^2 z_j\Big)\\
&\quad\times\prod_{\substack{B\in\pi_1\cup\pi_2\\ not\ marked}}\Big(\int_{\R^{|B|}}\int_{[\eta_0,10]^{|B|}}\Big[\prod_{j\in B}(\partial_{\overline{z}}(f_j)_{\C})(z_j)\Big]m_{\circ}[B]\prod_{j\in B}\dx^2 z_j\Big),
\end{align*}
where we set again $\dx^2 z_j=\dx x_j\dx\eta_j$ for $z_j=x_j+\ri\eta_j$. In particular, the product in the last line can again be evaluated using~\cite[Lem.~4.1]{CES-thermalization}. It remains to compute the integral involving $m_{\circ\circ}[U_1|U_2]$. We claim that
\begin{align}
&\frac{1}{\pi^{|U|}}\int_{\R^{|U|}}\int_{[\eta_0,10]^{|U|}}\Big[\prod_{j\in U}(\partial_{\overline{z}}(f_j)_{\C})(z_j)\Big]m[U_1|U_2]\prod_{j\in U}\dx^2 z_j\NN\\
&=\mathrm{sc}[i_1,\dots,i_n|i_{n+1},\dots,i_{n+m}]+\cO\Big(\eta_0^2(N\eta_0)^{(|U|)/2}\max_{i\in[k]}\|f_i\|_{H^p}\max_{j\in[k+1,k+\ell]}\|f_j\|_{H^q}\Big)\label{eq-claimscformula}
\end{align}
where $U_1=\{i_1,\dots,i_n\}$ and $U_2=\{i_{n+1},\dots,i_{n+m}\}$. The corresponding result for $\mathrm{sc}_{\circ\circ}[\cdot|\cdot]$ is then immediate from the second-order moment-cumulant relation~\eqref{eq-mcrelation2}.

\medskip
To establish~\eqref{eq-claimscformula}, we use the explicit integral representation for $m[\cdot|\cdot]$ from Corollary~\ref{cor-mintegral} and rewrite the resulting multi-integral involving the kernel~\eqref{eq-kernel}. Let again $U:=U_1\cup U_2$. Noting that both $|x-z_j|$ and $|y-z_j|$ are bounded from below by $\eta_0$ for any $j$ and recalling the bound~\eqref{eq-aacintbound} for the almost analytic extension, we have
\begin{align*}
&\int_{[-2,2]^2}\int_{\R^{|U|}}\int_{[\eta_0,10]^{|U|}}\Big|\Big[\prod_{j\in U}(\partial_{\overline{z}}(f_j)_{\C})(z_j)\Big]\Big(\sum_{i\in U_1}\frac{1}{(x-z_i)^2}\cdot\prod_{j\neq i}\frac{1}{x-z_j}\Big)\\
&\quad\times\Big(\sum_{i\in U_2}\frac{1}{(y-z_i)^2}\cdot\prod_{j\neq i}\frac{1}{y-z_j}\Big)\frac{u(x,y)}{2}\Big|\Big[\prod_{j\in U}\dx^2 z_j\Big]\dx x\dx y<\infty
\end{align*}
for any fixed $N$, as the $z_j$ integrations are bounded by an ($N$-dependent) constant, which is integrable w.r.t. to the measure $\nu(\dx x,\dx y)=u(x,y)\dx x\dx y$ (cf.~Remark to Corollary~\ref{cor-mintegral}). Hence, Fubini's theorem allows interchanging the order of integration and move the $x$ and $y$ integrations inside. A brief calculation using~\eqref{eq-Stokes} yields
\begin{displaymath}
\int_{\R}\int_{[\eta_r,10]}(\partial_{\overline{z}}(f_j)_{\C})(z_j)\frac{1}{x-z_j}\dx^2 z_j=\pi f(x)+\cO(\eta_r)
\end{displaymath}
with $\eta_r=N^{-5k}$, and we further have
\begin{align*}
\frac{1}{\pi}\int_{\R}\int_{[\eta_0,10]}(\partial_{\overline{z}}(f_j)_{\C})(z_j)\Big(\frac{1}{x-z_j}\Big)^2\dx^2 z_j=f'(x)+\cO(\eta_r)
\end{align*}
using integration by parts. Hence,
\begin{align*}
&\frac{1}{\pi^{|U|}}\int_{\R^{|U|}}\int_{[\eta_0,10]^{|U|}}\Big[\prod_{j\in U}(\partial_{\overline{z}}(f_j)_{\C})(z_j)\Big]\\
&\quad\times\Big(\sum_{i\in U_1}\frac{1}{(x-z_i)^2}\cdot\prod_{j\neq i}\frac{1}{x-z_j}\Big)\Big(\sum_{i\in U_2}\frac{1}{(y-z_i)^2}\cdot\prod_{j\neq i}\frac{1}{y-z_j}\Big)\Big[\prod_{j\in U}\dx^2 z_j\Big]\\
&=\Big(\sum_{i\in U_1}f_i'(x)\cdot\prod_{j\neq i}f_j(x)\Big)\Big(\sum_{i\in U_2}f_i'(y)\cdot\prod_{j\neq i}f_j(y)\Big)\\
&\quad+\cO\Big(\eta_0^2(N\eta_0)^{(|U|)/2}\max_{i\in[k]}\|f_i\|_{H^p}\max_{j\in[k+1,k+\ell]}\|f_j\|_{H^q}\Big)
\end{align*}
such that~\eqref{eq-claimscformula} follows from the Leibniz rule. Recall the $[\eta_r,\eta_0]$ regime can be added back in exchange for an $\cO(\eta_0^2(N\eta_0)^{(|U|)/2}\max_{i\in[k]}\|f_i\|_{H^p}\max_{j\in[k+1,k+\ell]}\|f_j\|_{H^q})$ error. This yields~\eqref{eq-claimscformula}, which concludes the proof of~\eqref{eq-phis2}.
\qed

\subsection{Proof of Corollary~\ref{cor-thermalization} (Application to Thermalization)}\label{app-thermalization}
Throughout the proof, we set $\smash{\widetilde{f}_j(x)=\re^{\ri t_jx}\chi(x)}$ with a symmetric smooth cutoff function~$\chi$ that is equal to one on $[-5/2,5/2]$ and equal to zero on $[-3,3]^c$. By eigenvalue rigidity, the functions $\smash{\widetilde{f}_j(W)}$ and $f_j(W)=\re^{\ri t_jW}$ coincide with high probability and we can use them interchangeably. The deterministic approximation as well as the Gaussian fluctuations around it are now immediate from~\cite[Cor.~2.7]{CES-optimalLL} and Theorem~\ref{thm-functCLT}, respectively. Note that we use the multi-point functional CLT for the macroscopic regime due to $t\in\R$ being $N$-independent. The limiting variance can be read off from Corollary~\ref{cor-covarianceLL}. Observing that
\begin{displaymath}
\overline{\langle A_1(t)A_2\rangle}=\langle \re^{\ri tW}A_1^*\re^{-\ri tW}A_2^*\rangle
\end{displaymath}
we set $f_1(x)=\re^{\ri tx}$, $f_2(x)=\re^{-\ri tx}$, $f_3(x)=\re^{\ri tx}$, and $f_4(x)=\re^{-\ri tx}$ as well as $A_3=A_1^*$, and $A_4=A_2^*$ to apply~\eqref{eq-covfunctionsprecise}. As $A_1$ and $A_2$ are assumed to be traceless, only the terms corresponding to the permutations
\begin{equation}\label{eq-NCAcontributions}
\pi\in\{(14)(23),(13)(24),(1)(24)(3),(14)(2)(3),(1)(24)(3),(13)(2)(4)\}\subset\NCA(2,2)
\end{equation}
as well as the terms corresponding to the marked partitions
\begin{align}
\pi\in&\Big\{\big\{\{\underline{1}\},\{2\}\big\}\times\big\{\{\underline{3}\},\{4\}\big\},\big\{\{\underline{1}\},\{2\}\big\}\times\big\{\{3\},\{\underline{4}\}\big\},\big\{\{1\},\{\underline{2}\}\big\}\times\big\{\{\underline{3}\},\{4\}\big\},\NN\\
&\quad\big\{\{1\},\{\underline{2}\}\big\}\times\big\{\{3\},\{\underline{4}\}\big\}\Big\}\label{eq-markedcontributions}
\end{align}
contribute to the limiting variance of $\langle A_1(t)A_2\rangle$. Note that the marked blocks in~\eqref{eq-markedcontributions} are distinguished by underlining.

\medskip
It remains to discuss the $t\rightarrow\infty$ limit. We have
\begin{displaymath}
\int_{-2}^2\re^{\ri t x}\rho_{sc}(x)\dx x=\frac{J_1(2t)}{t}
\end{displaymath}
where $J_1$ is a Bessel function of the first kind obeying the asymptotics
\begin{displaymath}
J_1(x)=-\cos\Big(x+\frac{\pi}{2}\Big)\sqrt{\frac{2}{\pi x}}+\cO\Big(\frac{1}{x^{3/2}}\Big), \quad x\gg1.
\end{displaymath}
In particular,
\begin{displaymath}
\mathrm{sc}_\circ[1]=\mathrm{sc}[1]=\frac{J_1(2t)}{t}=\cO\Big(\frac{1}{t^{3/2}}\Big),\quad t\gg1.
\end{displaymath}
Hence, it readily follows that the term corresponding to $\langle A_1A_3\rangle \langle A_2A_4\rangle=\langle |A_1|^2\rangle\langle |A_2|^2\rangle$ is the largest among the contributions from $\NCA(2,2)$ for large $t$, giving
\begin{displaymath}
\mathrm{sc}_{\circ}[1,4]\mathrm{sc}_{\circ}[2,3]=\Big(1-\frac{J_1(2t)J_1(2t)}{t^2}\Big)\Big(1-\frac{J_1(2t)J_1(2t)}{t^2}\Big)=1+\cO\Big(\frac{1}{t^3}\Big),\quad t\gg1,
\end{displaymath}
where we also used the symmetry $J_1(-x)=-J_1(x)$. Moreover, we obtain that, e.g.,
\begin{displaymath}
\mathrm{sc}_{\circ}[1]\mathrm{sc}_{\circ}[2,3]\mathrm{sc}_{\circ}[4]=\cO\Big(\frac{1}{t^3}\Big),\quad t\gg1,
\end{displaymath}
with the remaining permutations in~\eqref{eq-NCAcontributions} yielding contributions of comparable or lower order in the $t\rightarrow\infty$ limit.

\medskip
Lastly, we consider the marked partitions in~\eqref{eq-markedcontributions}, which correspond to the term of~$\Var[\xi]$ that contains $\langle A_1A_2\rangle\langle A_3A_4\rangle=|\langle A_1A_2\rangle|^2$. As we work in the macroscopic regime of Theorem~\ref{thm-functCLT}, Equation~\eqref{eq-defsc2} coincides with the limiting covariance structure of the CLT in~\cite[Sect.~2]{LytovaPasturCLT}. Hence, we obtain, e.g.,
\begin{equation}\label{eq-calculationcc}
\mathrm{sc}[1|4]=\frac{1}{2\pi^2}\int_{-2}^2\int_{-2}^2\frac{1-\cos(t(x-y))}{(x-y)^2}\frac{4-xy}{\sqrt{4-x^2}\sqrt{4-y^2}}\dx x\dx y.
\end{equation}
In particular, the cutoff $\chi$ does not enter the computation. Note that the expressions for $\mathrm{sc}[1|4]$ and $\mathrm{sc}[2|3]$ resp. $\mathrm{sc}[1|3]$ and $\mathrm{sc}[2|4]$ yield the same contribution by symmetry. The integral on the right-hand side of~\eqref{eq-calculationcc} is finite, however, it will grow with $t$ as $t\rightarrow\infty$. To identify the asymptotics, we distinguish between the contributions of the bulk regime
\begin{displaymath}
\frac{4-xy}{\sqrt{4-x^2}\sqrt{4-y^2}}=\cO(1),
\end{displaymath}
and the edge regime where the denominator $(\sqrt{4-x^2}\sqrt{4-y^2})^{-1}$ becomes singular. As
\begin{displaymath}
\lim_{t\rightarrow\infty}\frac{1}{t}\int_{-2}^2\int_{-2}^2\frac{1-\cos(t(x-y))}{(x-y)^2}=4\pi,
\end{displaymath}
the contribution of the bulk is readily identified to be $\cO(t)$. In the edge regime, we expand the square root in the denominator and further consider the contributions around the diagonal ($|x-y|\lesssim t^{-1}$) and away from it separately whenever $x$ and $y$ are close to the same value. This also yields a bound of order $\cO(t)$, implying that $\mathrm{sc}[1|4]=\cO(t)$. Recalling the identity $\mathrm{sc}_{\circ\circ}[1|4]=\mathrm{sc}[1|4]-\mathrm{sc}_\circ[1,4]$ from~\eqref{eq-mcrelation2}, we obtain
\begin{displaymath}
\mathrm{sc}_{\circ\circ}[1|4]\mathrm{sc}_{\circ}[2]\mathrm{sc}_{\circ}[3]=\cO\Big(\frac{1}{t^2}\Big), \quad t\gg1.
\end{displaymath}
The other marked partitions in~\ref{eq-markedcontributions} give rise to terms of comparable order. Summing up all contributions yields~\eqref{eq-varlimit}. The proof of~\eqref{eq-covlimit} is analogous and hence omitted.
\qed

\appendix
\section{Proof of Corollary~\ref{cor-Eproperties} (Meta Argument)}\label{app-meta}
Recall from Lemma~\ref{lem-mEexchangemeso} that
\begin{displaymath}
\E\langle T_1\dots T_k\rangle=\m[T_1,\dots,T_k]+\frac{\kappa_4}{N}\cE[T_1,\dots,T_k]+\cO\Big(\frac{N^\eps}{N\, \sqrt{N\eta_*}\ \eta_*^{k-a/2}}\Big),
\end{displaymath}
i.e., $\cE[T_1,\dots,T_k]$ constitutes the first subleading term of $\E\langle T_1\dots T_k\rangle$. In particular, we have
\begin{equation}\label{eq-metacandidate}
N(\E\langle T_1\dots T_k\rangle-\m[T_1,\dots,T_k])=\kappa_4\cE[T_1,\dots,T_k]+\cO\Big(\frac{N^\eps}{\sqrt{N\eta_*}\ \eta_*^{k-a/2}}\Big),
\end{equation}
where the quantity on the left-hand side of~\eqref{eq-metacandidate} satisfies the properties stated in Corollary~\ref{cor-Eproperties}. For $\E\langle T_1\dots T_k\rangle$, this is immediate from the cyclicity of the trace and the resolvent identity $G_kG_1=\frac{G_k-G_1}{z_k-z_1}$, while the corresponding properties for $\m[\cdot]$ follow from~\eqref{eq-defM} and~\cite[Lem.~5.4]{CES-thermalization}. Note that~\eqref{eq-Edivdif2} is a special case of~\eqref{eq-Edivdif1} and that~\eqref{eq-Edivdif3} is obtained by iterating~\eqref{eq-Edivdif2}. Once the formula~\eqref{eq-Edivdif3} is established, the permutation symmetry readily follows from the divided difference structure. Hence, Corollary~\ref{cor-Eproperties}(iii) follows from~(i) and~(ii).

\medskip
It remains to show that $\cE[\cdot]$ satisfies the same cyclicity and divided difference properties as the quantity on the left-hand side of~\eqref{eq-metacandidate}. Note that simply taking the $N\rightarrow\infty$ limit and applying Lemma~\ref{lem-mEexchangemeso} is not sufficient if $A_j\neq\Id$ for some $j$, as the deterministic matrices are themselves $N$-dependent quantities. Instead, let $L\in\N$ and consider the $NL\times NL$ Wigner matrix $\mathscr{W}$ as well as the deterministic matrices $\mathscr{A}_1,\dots,\mathscr{A}_k\in\C^{NL\times NL}$. Here, $\mathscr{W}$ is defined using the same random variables $\chi_d,\chi_{od}$ as~$W$ (i.e., $\sqrt{N}W$ and $\sqrt{NL}\mathscr{W}$ have the same entry distribution) and we define
\begin{displaymath}
\mathscr{A}_j:=A_j\otimes \Id_{L\times L},\quad j=1,\dots,k,
\end{displaymath}
where $\otimes$ denotes the tensor product, i.e.,
\begin{displaymath}
\mathscr{A}_j=\begin{pmatrix}(A_j)_{11}\Id_{L\times L}&\cdots &(A_j)_{1N}\Id_{L\times L}\\ \vdots &\ddots&\vdots\\  (A_j)_{N1}\Id_{L\times L}&\cdots &(A_j)_{NN}\Id_{L\times L}\end{pmatrix}.
\end{displaymath}
Next, let $\mathscr{G}_j:=(\mathscr{W}-z_j)^{-1}$ and $\mathscr{T}_j:=\mathscr{G}_j\mathscr{A}_j$ for $j=1,\dots,k$ and denote by $\curlym[\mathscr{T}_1,\dots,\mathscr{T}_k]$ the deterministic approximation of $\langle \mathscr{T}_1\dots \mathscr{T}_k\rangle$. As both $\mathscr{W}$ and $W$ are Wigner matrices and $\langle \mathscr{A}_i\mathscr{A}_j\rangle=\langle (A_iA_j)\otimes\Id_{L\times L}\rangle=\langle A_iA_j\rangle$ by definition of the tensor product, it follows from the closed form of $\m[\cdot]$ in~\cite[Thm.~2.6]{CES-thermalization} that
\begin{equation}\label{eq-curlyaresame}
\curlym[\mathscr{T}_1,\dots,\mathscr{T}_k]=\m[T_1,\dots,T_k].
\end{equation}
Similarly, the $1/N$ error term of $\E\langle \mathscr{T}_1\dots \mathscr{T}_k\rangle$ is given by $\cE[T_1,\dots,T_k]$, since~\eqref{eq-curlyaresame} ensures that we obtain the same recursion from Definition~\ref{def-E}.

\medskip
We thus conclude that
\begin{align*}
&|\cE[T_1,\dots,T_k]-\cE[T_2,\dots,T_k,T_1]|\\
&\leq |\cE[T_1,\dots,T_k]-N(\E\langle \mathscr{T}_1\dots \mathscr{T}_k\rangle-\curlym[\mathscr{T}_1,\dots,\mathscr{T}_k])|\\
&\quad+|\cE[\mathscr{T}_2,\dots,\mathscr{T}_k,\mathscr{T}_1]-N(\E\langle \mathscr{T}_2\dots \mathscr{T}_k,\mathscr{T}_1\rangle-\curlym[\mathscr{T}_2,\dots,\mathscr{T}_k,\mathscr{T}_1])|\\
&=\cO\Big(\frac{(NL)^\eps}{\sqrt{NL\eta_*}\ \eta_*^{k-a/2}}\Big)
\end{align*}
by Lemma~\ref{lem-mEexchangemeso} for any $\eta_*\gg (NL)^{-1}$. Letting $L\rightarrow\infty$ while keeping all other parameters $N,z_1,\dots,z_k,A_1,\dots,A_k$ fixed yields
\begin{displaymath}
\cE[T_1,\dots,T_k]=\cE[T_2,\dots,T_k,T_1],
\end{displaymath}
i.e., $\cE[\cdot]$ is cyclic as claimed in part (a) of Corollary~\ref{cor-Eproperties}. Similarly, we obtain that
\begin{displaymath}
\Big|\cE[T_1,\dots,T_{k-1},G_k]-\frac{\cE[T_2,\dots,T_{k-1},G_kA_1]-\cE[T_1,\dots,T_{k-1}]}{z_k-z_1}\Big|=\cO\Big(\frac{(NL)^\eps}{\sqrt{NL\eta_*}\ \eta_*^{k-a/2}}\Big)
\end{displaymath}
whenever $z_1\neq z_k$ and $A_k=\Id$, which implies~\eqref{eq-Edivdif1}. Recalling that (i) and (ii) of Corollary~\ref{cor-Eproperties} imply (iii), the proof is complete.
\qed

\renewcommand*{\bibname}{References}

\let\oldthebibliography\thebibliography
\let\endoldthebibliography\endthebibliography
\renewenvironment{thebibliography}[1]{
  \begin{oldthebibliography}{#1}
    \setlength{\itemsep}{0.5em}
    \setlength{\parskip}{0em}
}
{
  \end{oldthebibliography}
}

\bibliographystyle{plain}
\bibliography{References}

\begin{thebibliography}{10}

\bibitem{BaiSilverstein2004}
Z.~D. Bai and J.~W. Silverstein.
\newblock {CLT} for linear spectral statistics of large-dimensional sample
  covariance matrices.
\newblock {\em Ann. Probab.}, 32(1A):553--605, 2004.

\bibitem{BaiYao2005}
Z.~D. Bai and J.~Yao.
\newblock On the convergence of the spectral empirical process of {W}igner
  matrices.
\newblock {\em Bernoulli}, 11:1059--1092, 2005.

\bibitem{BaoHe2021}
Z.~Bao and Y.~He.
\newblock Quantitative {CLT} for linear eigenvalue statistics of {W}igner
  matrices.
\newblock {\em Preprint, arXiv:2103.05402}, 2021.

\bibitem{BaoSchnelliXu2022}
Z.~Bao, K.~Schnelli, and Y.~Xu.
\newblock Central limit theorem for mesoscopic eigenvalue statistics of the
  free sum of matrices.
\newblock {\em Int. Math. Res. Not.}, 2022(7):5320--5382, 2022.

\bibitem{BekermanLebleSerfaty2018}
F.~Bekerman, T.~Leblé, and S.~Serfaty.
\newblock {CLT} for fluctuations of $\beta$-ensembles with general potential.
\newblock {\em Electron. J.Probab.}, 23:1--31, 2018.

\bibitem{BekermanLodhia2018}
F.~Bekerman and A.~Lodhia.
\newblock Mesoscopic central limit theorem for general $\beta$-ensembles.
\newblock {\em Ann. Inst. H. Poincar\'{e} Probab. Stat.}, 54(4):1917--1938,
  2018.

\bibitem{BorotGuionnet2013}
G.~Borot and A.~Guionnet.
\newblock Asymptotic expansion of $\beta$ matrix models in the one-cut regime.
\newblock {\em Commun. Math. Phys.}, 317:447--483, 2013.

\bibitem{BourgardeModyPain2022}
P.~Bourgade, K.~Mody, and M.~Pain.
\newblock Optimal local law and central limit theorem for $\beta$-ensembles.
\newblock {\em Commun. Math. Phys.}, 390:1017--1079, 2022.

\bibitem{CES-nonHermCLT}
G.~Cipolloni, L.~Erd\H{o}s, and D.~Schröder.
\newblock Central limit theorem for linear eigenvalue statistics of
  non-{H}ermitian random matrices.
\newblock {\em Comm. Pure Appl. Math.}, 76:946--1034, 2019.

\bibitem{CES-ETH}
G.~Cipolloni, L.~Erd\H{o}s, and D.~Schröder.
\newblock Eigenstate thermalization hypothesis for {W}igner matrices.
\newblock {\em Commun. Math. Phys.}, 388(2):1005–1048, 2021.

\bibitem{CES-optimalLL}
G.~Cipolloni, L.~Erd\H{o}s, and D.~Schröder.
\newblock Optimal multi-resolvent local laws for {W}igner matrices.
\newblock {\em Electron. J. Probab.}, 27:1--38, 2022.

\bibitem{CES-thermalization}
G.~Cipolloni, L.~Erd\H{o}s, and D.~Schröder.
\newblock Thermalization for {W}igner matrices.
\newblock {\em J. Funct. Anal.}, 282(8), 2022.

\bibitem{CES-functCLT}
G.~Cipolloni, L.~Erd\H{o}s, and D.~Schröder.
\newblock Functional central limit theorems for {W}igner matrices.
\newblock {\em Ann. Appl. Probab.}, 33(1):447--489, 2023.

\bibitem{CollinsMingoSniadySpeicher2007}
B.~Collins, J.~Mingo, P.~\'{S}niady, and R.~Speicher.
\newblock Second order freeness and fluctuations of random matrices: {III}.
  higher order freeness and free cumulants.
\newblock {\em Documenta Math.}, 12:1--70, 2007.

\bibitem{Davies1995}
E.~B. Davies.
\newblock The functional calculus.
\newblock {\em J. London Math. Soc.}, 52(1):166–176, 1995.

\bibitem{Deutsch1991}
J.~Deutsch.
\newblock Quantum statistical mechanics in a closed system.
\newblock {\em Phys. Rev. A}, 43:2046--2049, 1991.

\bibitem{DiazJaramilloPardo2022}
M.~Diaz, A.~Jaramillo, and J.~C. Pardo.
\newblock Fluctuations for matrix-valued {G}aussian processes.
\newblock {\em Ann. Henri Poincaré}, 58(4):2216--2249, 2022.

\bibitem{DiazMingo2022}
M.~Diaz and J.A. Mingo.
\newblock On the analytic structure of second-order non-commutative probability
  spaces and functions of bounded {F}réchet variation.
\newblock {\em Random Matrices: Theory Appl.}, page 2250044, 2022.

\bibitem{ErdoesJi2021}
L.~Erd\H{o}s and H.~C. Ji.
\newblock Functional {CLT} for non-{H}ermitian random matrices.
\newblock {\em Preprint, arXiv:2112.11382}, 2021.

\bibitem{EKYY2013}
L.~Erd\H{o}s, A.~Knowles, H.-T. Yau, and J.~Yin.
\newblock The local semicircle law for a general class of random matrices.
\newblock {\em Electron. J. Probab.}, 18:1--58, 2013.

\bibitem{EYY2012}
L.~Erd\H{o}s, H.-T. Yau, and J.~Yin.
\newblock Rigidity of eigenvalues of generalized {W}igner matrices.
\newblock {\em Adv. Math.}, 229:1435--1515, 2012.

\bibitem{LambertLedouxWebb2019}
M.~Ledoux G.~Lambert and C.~Webb.
\newblock Quantitative normal approximation of linear statistics of
  $\beta$-ensembles.
\newblock {\em Ann. Probab.}, 47(5):2619--2685, 2019.

\bibitem{Guionnet2002}
A.~Guionnet.
\newblock Large deviation upper bounds and central limit theorems for
  non-commutative functionals of {G}aussian large random matrices.
\newblock {\em Ann. Inst. H. Poincar\'{e} Probab. Stat.}, 38:341--384, 2002.

\bibitem{He2019}
Y.~He.
\newblock Mesoscopic linear statistics of {W}igner matrices of mixed symmetry
  class.
\newblock {\em J. Stat. Phys.}, 175:932--959, 2019.

\bibitem{HeKnowles2017}
Y.~He and A.~Knowles.
\newblock Mesoscopic eigenvalue statistics of {W}igner matrices.
\newblock {\em Ann. Appl. Probab.}, 27(3):1510--1550, 2017.

\bibitem{HeKnowles2020}
Y.~He and A.~Knowles.
\newblock Mesoscopic eigenvalue density correlations of {W}igner matrices.
\newblock {\em Probab. Theory Relat. Fields}, 177:147--216, 2020.

\bibitem{JiLee2020}
H.~C. Ji and J.~O. Lee.
\newblock Gaussian fluctuations for linear spectral statistics of deformed
  {W}igner matrices.
\newblock {\em Random Matrices: Theory Appl.}, 9(3):2050011, 2020.

\bibitem{Johansson1998}
K.~Johansson.
\newblock On fluctuations of eigenvalues of random {H}ermitian matrices.
\newblock {\em Duke Math. J.}, 91(1):151--204, 1998.

\bibitem{KhorunzhyKhoruzhenkoPastur1995}
A.~M. Khorunzhi, B.~A. Khoruzhenko, and L.~A. Pastur.
\newblock On the {1/N} corrections to the {G}reen functions of random matrices
  with independent entries.
\newblock {\em J. Phys. A Math. Gen.}, 28:L31, 1995.

\bibitem{KhorunzhyKhoruzhenkoPastur1996}
A.~M. Khorunzhi, B.~A. Khoruzhenko, and L.~A. Pastur.
\newblock Asymptotic properties of large random matrices with independent
  entries.
\newblock {\em J. Math. Phys.}, 37:5033--5060, 1996.

\bibitem{LandonSosoe2022}
B.~Landon and P.~Sosoe.
\newblock Almost-optimal bulk regularity conditions in the {CLT} for {W}igner
  matrices.
\newblock {\em Preprint, arXiv:2204.03419}, 2022.

\bibitem{LiSchnelliXu2021}
Y.~Li, K.~Schnelli, and Y.~Xu.
\newblock Central limit theorem for mesoscopic eigenvalue statistics of
  deformed {W}igner matrices and sample covariance matrices.
\newblock {\em Ann. Inst. H. Poincar\'{e} Probab. Statist.}, 57(1):506--546,
  2021.

\bibitem{LiXu2020}
Y.~Li and Y.~Xu.
\newblock On fluctuations of global and mesoscopic linear statistics of
  generalized {W}igner matrices.
\newblock {\em Bernoulli}, 27(2):1057--1076, 2021.

\bibitem{Lytova2013}
A.~Lytova.
\newblock On non-{G}aussian limiting laws for certain statistics of {W}igner
  matrices.
\newblock {\em Zh. Mat. Fiz. Anal. Geom.}, 9:536--581, 2013.

\bibitem{LytovaPasturCLT}
A.~Lytova and L.~Pastur.
\newblock Central limit theorem for linear eigenvalue statistics of the
  {W}igner and the sample covariance random matrices.
\newblock {\em Metrika}, 69:153--172, 2009.

\bibitem{MaleMingoPecheSpeicher2020}
C.~Male, J.~A. Mingo, S.~Peché, and R.~Speicher.
\newblock Joint global fluctuations of complex {W}igner and deterministic
  matrices.
\newblock {\em Random Matrices: Theory Appl.}, 11(2):2250015, 2022.

\bibitem{MSBook}
J.~A. Mingo and R.~Speicher.
\newblock {\em Free Probability and Random Matrices}.
\newblock Vol. 35, Fields Institute Research Monographs, Springer, New York,
  2017.

\bibitem{JRcompanion}
J.~Reker.
\newblock Fluctuation moments for regular functions of {W}igner matrices.
\newblock {\em Preprint}, 2023.

\bibitem{Vova2023}
V.~Riabov.
\newblock Mesoscopic eigenvalue statistics for {W}igner-type matrices.
\newblock {\em Preprint, arXiv:2301.01712}, 2023.

\bibitem{RudnickSarnak1994}
Z.~Rudnick and P.~Sarnak.
\newblock The behavior of eigenstates of arithmetic hyperbolic manifolds.
\newblock {\em Comm. Math. Phys.}, 161:195--213, 1994.

\bibitem{Shcherbina2011}
M.~Shcherbina.
\newblock Central limit theorem for linear eigenvalue statistics of the
  {W}igner and sample covariance random matrices.
\newblock {\em Zh. Mat. Fiz. Anal. Geom.}, 7:176--192, 2011.

\bibitem{Shcherbina2013}
M.~Shcherbina.
\newblock Fluctuations of linear eigenvalue statistics of $\beta$ matrix models
  in the multi-cut regime.
\newblock {\em J. Stat. Phys.}, 151:1004--1034, 2013.

\bibitem{Silverstein1990}
J.~W. Silverstein.
\newblock Weak convergence of random functions defined by the eigenvectors of
  sample covariance matrices.
\newblock {\em Ann. Probab.}, 18(3):1174--1194, 1990.

\bibitem{Silverstein2022}
J.~W. Silverstein.
\newblock Weak convergence of a collection of random functions defined by the
  eigenvectors of large dimensional random matrices.
\newblock {\em Random Matrices: Theory Appl.}, 11(4):2250033, 2022.

\bibitem{SinaiSoshnikov1998A}
Y.~G. Sinai and A.~B. Soshnikov.
\newblock Central limit theorem for traces of large random symmetric matrices
  with independent matrix elements.
\newblock {\em Bull. Brazilian Math. Soc.}, 29:1--24, 1998.

\bibitem{SinaiSoshnikov1998B}
Y.~G. Sinai and A.~B. Soshnikov.
\newblock A refinement of {W}igner's semicircle law in a neighborhood of the
  spectrum edge for random symmetric matrices.
\newblock {\em Funct. Anal. its Appl.}, 32:114--131, 1998.

\bibitem{Soshnikov2000}
A.~B. Soshnikov.
\newblock The central limit theorem for local linear statistics in classical
  compact groups and related combinatorial identities.
\newblock {\em Ann. Probab.}, 28(3):1353--1370, 2000.

\bibitem{SosoeWong2013}
P.~Sosoe and P.~Wong.
\newblock Regularity conditions in the {CLT} for linear eigenvalue statistics
  of {W}igner matrices.
\newblock {\em Adv. Math.}, 249:37--87, 2013.

\end{thebibliography}

\end{document}